\documentclass[12pt]{tac}
\usepackage{amssymb,amsmath,amsfonts,mathrsfs}
\usepackage{hyperref}
\usepackage[utf8]{inputenc}
\usepackage{color}
\usepackage{graphicx}
\usepackage{amscd}
\usepackage[shellescape]{gmp} 
\usepackage{stmaryrd}  
\usepackage{mathtools}
\usepackage{comment} 
\usepackage{multirow} 
\usepackage[all]{xy}   
\usepackage[dvipsnames]{xcolor}
\usepackage{soul} 
  
\usepackage{tikz}
\usepackage{tikz-cd}
\usetikzlibrary{decorations.pathmorphing}

\tikzset{snake it/.style={decorate, decoration=snake}}

\usetikzlibrary{decorations.markings}
\usetikzlibrary{decorations.pathmorphing}

\usepackage[rawfloats=true]{floatrow} 
\restylefloat{figure}

\hypersetup{
    colorlinks=true,
    citecolor=violet,
    filecolor=Salmon,
    linkcolor=OliveGreen,
    urlcolor=Maroon
}

\newtheorem{theorem}{Theorem}[section]
\newtheorem{lemma}[theorem]{Lemma}
\newtheorem{remark}[theorem]{Remark}

\newtheorem{proposition}[theorem]{Proposition}
\newtheorem{prop}[theorem]{Proposition}

\newtheorem{corollary}{Corollary}

\newtheorem{example}[theorem]{Example}

\newcommand{\RR}{\mathbb{R}}

\let\oldtocsection=\tocsection
\let\oldtocsubsection=\tocsubsection

\usepackage{chngcntr}
\counterwithin{figure}{section}

\title[Foams, iterated wreath products, field extensions, Sylvester sums]{Foams, iterated wreath products, field extensions and Sylvester sums}

\copyrightyear{2025}

\author[Mee Seong Im and Mikhail Khovanov (appendix by Lev Rozansky)]{Mee Seong Im and Mikhail Khovanov (with an appendix by Lev Rozansky)}

\address{\parbox{\linewidth}{Department of Mathematics, Johns Hopkins University, Baltimore, MD 21218, USA\\ 
Department of Mathematics, United States Naval Academy, Annapolis, MD 21402 USA}\\[7pt] 
\parbox{\linewidth}{Department of Mathematics, Johns Hopkins University, Baltimore, MD 21218, USA\\ 
Department of Mathematics, Columbia University, New York, NY 10027, USA}\\[7pt]
Department of Mathematics, University of North Carolina, Chapel Hill, NC 27599, USA \\
}

\eaddress{\href{mailto:meeseong@jhu.edu}{meeseong@jhu.edu} \CR \href{mailto:khovanov@jhu.edu}{khovanov@jhu.edu} \CR \href{mailto:rozansky@email.unc.edu}{rozansky@email.unc.edu}}

\date{\today}


\providecommand{\keywords}[1]{\textbf{\textit{Key words and phrases.}} #1}


\allowdisplaybreaks 
\begin{document}

\newcommand{\R}{\mathbb{R}}
\newcommand{\Q}{\mathbb{Q}}
\newcommand{\Z}{\mathbb{Z}}
\newcommand{\N}{\mathbb{N}}
\newcommand{\C}{\mathbb{C}}
\newcommand{\V}{\mathbf{V}}
\newcommand{\FF}{\mathsf{FF}}

\renewcommand\SS{\ensuremath{\mathbb{S}}}

\renewcommand{\l}{\lbrace}
\renewcommand{\r}{\rbrace}

\newcommand{\lra}{\longrightarrow}
\newcommand{\Ext}{\mathsf{Ext}}
\newcommand{\Fl}{\mathsf{Fl}}
\newcommand{\GL}{\mathsf{GL}}
\newcommand{\Gr}{\mathsf{Gr}}
\newcommand{\Hom}{\mathsf{Hom}}
\newcommand{\Ind}{\mathsf{Ind}}
\newcommand{\MF}{\mathsf{MF}}
\newcommand{\HMF}{\mathsf{HMF}}
\newcommand{\Res}{\mathsf{Res}}
\newcommand{\SL}{\mathsf{SL}}
\newcommand{\Syl}{\mathsf{Syl}}
\newcommand{\Sym}{\mathsf{Sym}}
\newcommand{\triv}{\mathsf{triv}}
\newcommand{\Id}{\mathsf{Id}}
\newcommand{\mc}{\mathcal}
\newcommand{\mcR}{\mathcal{R}}
\newcommand{\mf}{\mathfrak}
\newcommand{\mcC}{\mathcal{C}}
\newcommand{\mcO}{\mathcal{O}}
\newcommand{\U}{\mathsf{U}}

\newcommand{\Cat}{\mathsf{Cat}}
\newcommand{\End}{\mathsf{End}}
\newcommand{\tr}{\mathsf{tr}}
\newcommand{\trG}{\mathsf{tr}_{\mathsf{Gr}}} 

\renewcommand{\det}{\mathsf{det}}
\newcommand{\undbeta}{\underline{\beta}}
\newcommand{\mcF}{\mathcal{F}}
\newcommand{\Gal}{\mathsf{Gal}}
\newcommand{\mchar}{\mathsf{char}} 
\newcommand{\Fr}{\mathsf{Fr}}
\newcommand{\mthH}{\mathsf{H}} 
\newcommand{\one}{\mathbf{1}} 
\newcommand{\nchar}{\mathsf{char}\,} 

\newcommand{\myrightleftarrows}[1]{\mathrel{\substack{\xrightarrow{#1} \\[-.9ex] \xleftarrow{#1}}}}

\newcommand{\CC}{\mathbb{C}}
\newcommand{\kk}{\mathbf{k}} 
\newcommand{\ukk}{\underline{\kk}} 
\newcommand{\undx}{\underline{x}}
\newcommand{\undy}{\underline{y}}
\newcommand{\undz}{\underline{z}}
\newcommand{\unda}{\underline{a}}

\newcommand{\okk}{\overline{\kk}}  
\newcommand{\okksep}{\overline{\kk}^{sep}} 
 
\newcommand{\indf}{\mathsf{Ind}} 
\newcommand{\resf}{\mathsf{Res}} 
\newcommand{\gdim}{\mathsf{gdim}} 
\newcommand{\rk}{\mathsf{rk}}
\newcommand{\mmat}{\mathrm{Mat}}

\newcommand{\Pa}{\mathsf{Pa}}   
\newcommand{\Cob}{\mathsf{Cob}} 
\newcommand{\Cobtwo}{\Cob_2}   
\newcommand{\Kob}{\mathsf{Kob}}
\newcommand{\Cobal}{\Cob_{\alpha}}  

\newcommand{\cobal}{\mathsf{Cob}_{\alpha}} 

\newcommand{\udcob}{\underline{\mathsf{DCob}}}

\newcommand{\dmod}{\mathsf{-mod}}   
\renewcommand{\pmod}{\mathsf{-pmod}}    

\newcommand{\oplusop}[1]{{\mathop{\oplus}\limits_{#1}}}
\newcommand{\ang}[1]{\langle #1 \rangle } 
\newcommand{\bbn}[1]{\mathbb{B}^{#1}}

\newcommand{\pseries}[1]{\kk\llbracket #1 \rrbracket}
\newcommand{\rseries}[1]{R\llbracket #1 \rrbracket}
\newcommand{\ovfi}[1]{\overline{F_{#1}}}
\newcommand{\delcatv}[1]{\mathrm{Rep}(S_{#1})}
\newcommand{\delcatbar}[1]{\underline{\mathrm{Rep}}(S_{#1})}
\newcommand{\undrep}{\underline{\mathsf{Rep}}}
\newcommand{\mfg}{\mathfrak{g}} 
\newcommand{\brak}[1]{\langle #1 \rangle}  



\newcommand{\cC}{{\mathcal C}} 
\newcommand{\cA}{{\mathcal A}}
\newcommand{\cZ}{{\mathcal Z}}
\newcommand{\cD}{{\mathcal D}}
\newcommand{\M}{{\mathcal M}}
\newcommand{\be}{{\bf 1}}
\newcommand{\bl}{{\bf s}}

\newcommand{\eps}{{\varepsilon}}
\newcommand{\sq}{$\square$}
\newcommand{\bi}{\bar \imath}
\newcommand{\bj}{\bar \jmath}
\newcommand{\Ve}{\mbox{Vec}}
\newcommand{\sVec}{\mbox{sVec}}
\newcommand{\Rep}{\mathrm{Rep}}
\newcommand{\Ver}{\mbox{Ver}}
\newcommand{\Mod}{\mbox{Mod}}
\newcommand{\Bimod}{\mbox{Bimod}}
\newcommand{\id}{\mbox{id}}
\newcommand{\inv}{\mbox{inv}}
\newcommand{\ot}{\otimes}
\newcommand{\iHom}{\underline{\mbox{Hom}}}

\maketitle

\begin{abstract}
 Certain foams and relations on them are introduced to interpret functors and natural transformations in categories of representations of iterated wreath products of cyclic groups of order two. We also explain how patched surfaces with defect circles and foams relate to separable field extensions and Galois theory and explore a relation between overlapping foams and Sylvester double sums. In the appendix, joint with Lev Rozansky, we compare traces in two-dimensional TQFTs coming from matrix factorizations with those in field extensions.   
\end{abstract}

\tableofcontents

%
%

\section{Introduction} 
\label{section:intro}

The goal of this work is to explore possible interactions between the theory of decorated 2-dimensional complexes and parts of representation theory of finite groups, Galois theory and the theory of resultants. We refer to decorated 2-dimensional complexes as foams, usually imposing local structure requirements on these foams, including labeling of its zero, one, and two-dimensional facets and the presence of defects, such as zero and one-dimensional defects on facets and one-dimensional defects on seams of a foam. 

The existence of foam interpretation has been missing from these fields, and the present paper is a largely informal introduction to such an interpretation. This approach can be motivated by the multitude of biadjoint functors present in these theories, existence of many interesting natural transformations between compositions of these functors, and possibility to form exterior tensor products (i.e., via the direct product of groups).
Together, these properties hint at an interpretation of natural transformations between the suitable functors in the above theories via two-dimensional topological structures with additional decorations and singularities. 

Induction and restriction functors between categories of representations of finite groups are biadjoint and natural transformations between their compositions have a graphical presentation via systems of oriented arcs and circles in the plane, see~\cite{Kh3}.  In these diagrams regions are labeled by finite groups and lines by inclusions of groups. More generally, two-sided adjoint functors have a graphical interpretations via isotopies of arcs in the plane~\cite{Kh1,La1,La2,KL,KQ}.

In Section~\ref{sec-foams} we explain how such planar diagrams can be refined to foams when some of the groups have a direct product decomposition. Sections~\ref{foams} and~\ref{sec_defect} treat a special case when the groups are iterated wreath products of the symmetric group $S_2$ (equivalently, cyclic group $C_2$).

Iterated wreath products have been extensively studied in the last several decades. For some background on wreath products, see, for example,~\cite[Chapter 1]{Me} and \cite[Chapter 7]{Ro}. For representation-theoretic aspects of wreath products, see~\cite{Ke, CST, OOR, IW1, IW2, IO}. 
Natural transformations between compositions of induction and restriction functors between iterated wreath products of $S_2$ can be depicted by suitable foams. Facets of this foam labeled $n$ correspond to the $n$-th iterated wreath product group $G_n$, the group of symmetries of a full binary tree of depth $n$. Seams correspond to the induction and restriction for the inclusion $G_n\times G_n\subset G_{n+1}$ as an index two subgroup. Graphical calculus for these foams and its relation to representation theory of $G_n$ are developed in Sections~\ref{foams} and~\ref{sec_defect}. 

These foams are different from $\SL(N)$ or $\GL(N)$ foams. The latter are commonly used in link homology and categorification.
In particular, they can be used to describe the Soergel category~\cite{RWe,We,RW2} and to construct $\GL(N)$ link homology via foam evaluation~\cite{Kh2,RW1}; they also appear in categorified quantum groups~\cite{QR}. 

In Section~\ref{galois_extensions}, we explain how automorphisms of a commutative Frobenius algebra give rise to a decorated two-dimensional topological quantum field theory (2D TQFT) with defect circles. A further refinement is sometimes possible, along the lines of Turaev's homotopy quantum field theories (QFTs)~\cite{Tu1,Tu2,MS}, Landau--Ginzburg orbifolds~\cite{IV,BH,BR,LS,KW}, and orbifolded Frobenius algebras \cite{Ka}. We describe a useful way to encode a representation of the fundamental group of a surface, Poincar\'e dual to the standard description. We explain that surfaces in decorated TQFTs that come from separable field extensions and the standard trace on them admits a straightforward evaluation.

\vspace{0.1in} 

Section~\ref{sec-sylvester} contains a couple of curious connections of foams to Galois theory  and to  polynomial interpolation. Suppose given a degree $N$ irreducible polynomial $f(x)$ over a ground field $\kk$ with the maximal for that degree Galois group $\Gal(F/\kk)\cong S_N$, where $F$ is the splitting field. In Section~\ref{subsec_base_change} we identify $F$ and suitable intermediate fields of the extension with state spaces of MOY theta-webs, upon a  base change from the symmetric functions to $\kk$ via  coefficients of $f(x)$. In Section~\ref{subsec_over_sylv} we interpret the Sylvester double sums~\cite{Sy} that describe subresultants and related identities and expressions in the field of polynomial interpolation~\cite{DHKS,DKSV,KSV} via evaluation of overlapping foams.   

\vspace{0.1in} 

In the appendix (Section~\ref{subsec_mf}), written jointly with Lev Rozansky, we connect evaluations of closed surfaces in  2D TQFTs that come from matrix factorizations and Landau--Ginzburg models with the ones discussed in the present paper coming from field extensions. The connection is given by the formula \eqref{eq_hessian} and depicted in Figure~\ref{fig5_3_2}. It allows to express the field extension evaluation via that for the Landau-Ginzburg model, showing that the latter is at least as informative as the former.  

\subsection*{Acknowledgments}
\label{subsection:ackn}
The authors are grateful to Johan De Jong, Louis-Hadrien Robert, Alvaro Martinez Ruiz and Lev Rozansky for valuable discussions. 
The authors also thank the extremely thorough reports from the anonymous referee.
M.S.I. was partially supported by Naval Academy Research Council (NARC) at Annapolis, MD, and M.K. was partially supported by NSF grants DMS-1807425, DMS-2446892  and Simons Collaboration Award 994328 while working on this paper.

%
%

\section{Foams and functors between representation categories of direct products of groups}
\label{sec-foams}

\subsection{Diagrammatics of induction and restriction}
\label{subsection:diagrammatics-ind-res}
Given an inclusion of finite groups $H\subset G$ (or, more generally, an inclusion with $H$ of finite index in $G$) and a ground field $\kk$, induction and restriction functors $\indf_H^G$ and $\resf_G^H$ between categories of $\kk H$-modules and $\kk G$-modules are biadjoint, that is, adjoint on both the left and the right. Diagrammatics of  biadjoint functors for induction and restriction of finite groups is explained in~\cite[Section 3.2]{Kh3} and in~\cite{Kh1,La2,KL,KQ} in general.  Natural transformations between compositions of these functors can be depicted by planar diagrams of arcs and circles in the plane with regions labeled by $G$ and $H$ in a checkerboard fashion. 

\vspace{0.1in}

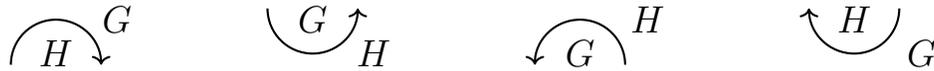
\begin{figure}
    \begin{center}
    \begin{tikzpicture}[scale=0.6]
    \begin{scope}[shift={(0,-0.5)}]
    \draw[thick,->] (0,0) arc (180:0:1cm);
    \node at (1,0.25) {\Large $H$};
    \node at (2.35,1) {\Large $G$};
    \end{scope}     
   
\begin{scope}[shift={(6,0.5)}]
    \draw[thick,->] (0,0) arc (180:360:1cm);
    \node at (1,-0.25) {\Large $G$};
    \node at (2.35,-1) {\Large $H$};
\end{scope}   

\begin{scope}[shift={(14,-0.5)}]
    \draw[thick,->] (0,0) arc (0:180:1cm);
    \node at (-1,0.25) {\Large $G$};
    \node at (.5,1) {\Large $H$};
\end{scope}

\begin{scope}[shift={(20,0.5)}]
    \draw[thick,->] (0,0) arc (360:180:1cm); 
    \node at (-1,-0.25) {\Large $H$};
    \node at (0.5,-1) {\Large $G$};
\end{scope}
    \end{tikzpicture}
    \end{center}
    \caption{Oriented cups and caps natural transformations, for induction and restriction between $H$- and $G$-modules, with $H\subset G$ a finite index subgroup.}
    \label{fig_cups_caps}
\end{figure}

Biadjointness can be encoded by four natural transformations that can be depicted by the four oriented cup and cap diagrams in Figure~\ref{fig_cups_caps}. Biadjointness is equivalent to the isotopy invariance of diagrams or arcs and circles built from these diagrams, and the four generating isotopy relations are shown in Figure~\ref{fig_isotopy}. 

\vspace{0.1in}

\begin{figure}
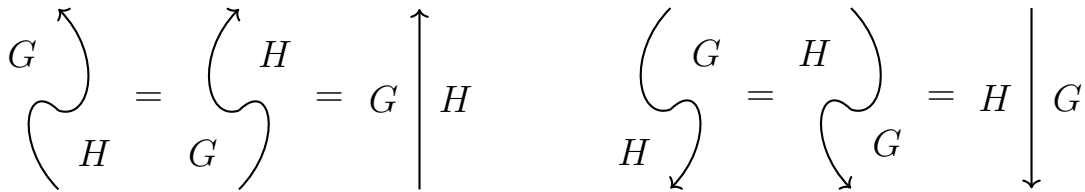

    \begin{center}

    \end{center}
    \caption{Biadjointness isotopy relations on compositions of cups and caps.}
    \label{fig_isotopy}
\end{figure}

Further extension of this construction adds diagrammatics for induction and restriction between many finite groups and additional diagrams for functor isomorphisms and other natural transformations between compositions of these functors~\cite[Section 3.2]{Kh3}. 

\vspace{0.1in}

\begin{figure}
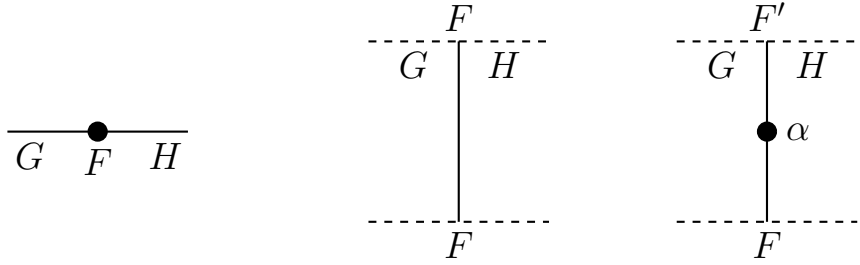

\begin{center}
%
\end{center}

\caption{Left: functor $F$. Middle: identity map $\id_F:F \Rightarrow F$. Right: natural transformation $\alpha: F\Rightarrow F'$.}
\label{fig0_0}
\end{figure}

The induction functor and, more generally, a functor $F:\kk H\dmod \lra \kk G\dmod$, for finite groups $H$ and $G$, can be depicted by a dot on a horizontal line, with intervals to the right and left of the dot labeled by $H$ and $G$, respectively, see Figure~\ref{fig0_0} left. The identity natural transformation $\id_F$ of $F$ is depicted by a vertical line in the plane, see Figure~\ref{fig0_0} middle. A natural transformation $\alpha: F\Rightarrow F'$ between two such functors is depicted by a dot on a vertical line, with intervals below and above the dot labeled by $F$ and $F'$, respectively, see Figure~\ref{fig0_0} right. 

\vspace{0.1in} 

In these considerations, it is natural to restrict to functors $F$ that admit biadjoint functors, that is, there exists a functor $\overline{F}$ which is both left and right adjoint to $F$, with biadjointness isomorphisms fixed. This allows to add ``cup" and ``cap" diagrams, their compositions and suitable isotopies to our graphical calculus, see for instance~\cite[Section 3.2]{Kh3} as well as the discussion of isotopies and biadjointness in~\cite{Kh1,La1,La2} and~\cite[Chapter 7]{KQ}.

\subsection{Extending to foams}
\label{subsection:extending-to-foams}
Starting from the planar diagrammatics of induction and restriction functors for finite groups one easily makes one step to its extension to foam diagrammatics for these functors, once direct products of groups are used. 
\vspace{0.1in} 

Suppose that group $H$ is the direct product, $H\cong H_1\times H_2$. Then suitable endofunctors and natural transformations between them in the category of $H$-modules can be reduced to exterior tensor products of those in categories of $H_1$-modules and $H_2$-modules. Diagrammatically, the $H$-plane that carries information about natural transformations of endofunctors in the category of $\kk H$-modules is converted into two parallel planes, one for each term $H_1, H_2$ in the direct product. 

\vspace{0.1in} 

For instance, a natural transformation $\alpha_i: F_i \lra F'_i$ between endofunctors $F_i,F'_i$ in the category of $H_i$-modules can be depicted by a dot on a vertical line in the $H_i$-plane, see Figure~\ref{fig1_0} left, for $i=1,2$. Bottom and top endpoints of the vertical line denote functors $F_i$ and $F'_i$, respectively. 

\vspace{0.1in}

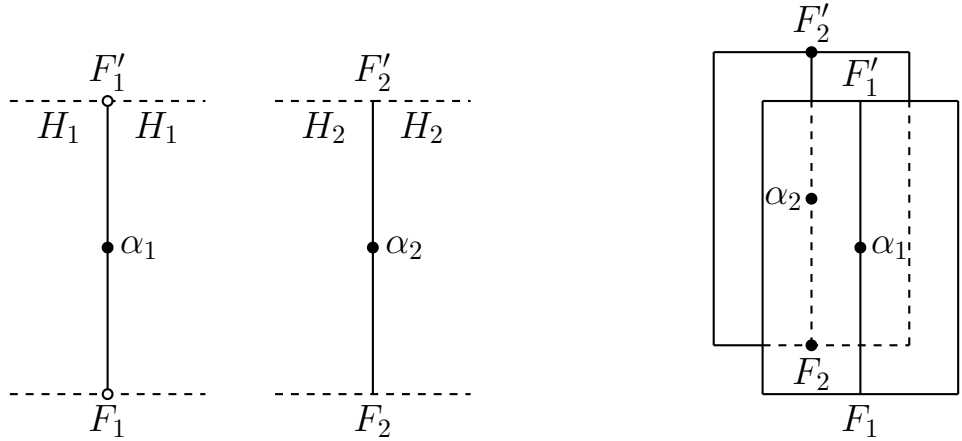
\begin{figure}
\begin{center}
\begin{tikzpicture}[scale=0.65]

\begin{scope}[shift={(0,0)}]

\draw[thick,dashed](-2,0) -- (2,0);

\node at (0,-0.65) {\Large $F_1$};


\draw[thick] (0,0) -- (0,4);

\draw[thick,fill] (0.2,1.5) arc (0:360:.2);

\node at (1,1.5){\Large $\alpha_1$};

\draw[thick,dashed](-2,4) -- (0,4);

\node at (-1,3) {\Large $H_1$};

\draw[thick,dashed](0,4) -- (2,4);


\node at (1,3) {\Large $H_1$};

\node at (0,4.65) {\Large $F_1'$};
\end{scope}

\begin{scope}[shift={(6,0)}]
\draw[thick,dashed](-2,0) -- (2,0);

\node at (0,-0.65) {\Large $F_2$};

\draw[thick](0,0) -- (0,4);
\draw[thick,fill] (0.2,1.5) arc (0:360:.2);
\node at (1,1.5) {\Large $\alpha_2$};

\draw[thick,dashed](-2,4) -- (0,4);

\node at (-1,3) {\Large $H_2$};

\draw[thick,dashed](0,4) -- (2,4);

\node at (1,3) {\Large $H_2$};

\node at (0,4.65) {\Large $F_2'$};
\end{scope}
 
\begin{scope}[shift={(16.5,-1)}]

\draw[thick] (-3,0) -- (3,0);

\node at (0,-0.65) {\Large $F_1$};

\draw[thick] (3,0) -- (3,4);

\draw[thick] (-3,0) -- (-3,4);

\draw[thick] (-3,4) -- (3,4);

\node at (0,4.65) {\Large $F_1'$};

\draw[thick](0,0) -- (0,4);

\draw[thick,fill] (0.2,2.5) arc (0:360:.2);

\node at (1,2.5) {\Large $\alpha_1$};

\draw[thick](-4,1.5) -- (-3,1.5);

\draw[thick,dashed] (-3,1.5) -- (2,1.5);

\draw[thick,dashed] (2,1.5) -- (2,4);

\draw[thick] (2,4) -- (2,5.5);

\draw[thick](-4,1.5) -- (-4,5.5);

\draw[thick](-4,5.5) -- (2,5.5);

\node at (-1,6.20) {\Large $F_2'$};

\draw[thick,dashed](-1,1.5) -- (-1,4);

\draw[thick](-1,4) -- (-1,5.5);

\draw[thick,fill] (-0.8,3) arc (0:360:.2);

\node at (-2,3) {\Large $\alpha_2$};

\node at (-1,0.85){\Large  $F_2$};

\end{scope}
\end{tikzpicture}
\end{center}

\caption{Diagrams of natural transformations $\alpha_1,\alpha_2$ and of their exterior tensor product.}
\label{fig1_0} 
\end{figure}

Then the natural transformation
\[\alpha_1\boxtimes \alpha_2 \ : \  F_1\boxtimes F_2 \Rightarrow F'_1 \boxtimes F'_2
\] 
between endofunctors in the category of $H_1\times H_2$-modules can be depicted by placing the two diagrams in parallel next to each other, see Figure~\ref{fig1_0} right. 

\vspace{0.1in}

When some of the groups are direct products, diagrammatic presentation of functors and their compositions as sequences of dots on a line can be refined to presentations via suitable graphs that come with a projection on a line. Suppose we are given an inclusion of groups $H_1\times H_2\subset G$. 
Denote the induction functor  $\indf_{H_1\times H_2}^G$ from $H_1\times H_2$-modules to $G$-modules by a vertex with $H_1,H_2$ lines flowing in and $G$ line flowing out, see Figure~\ref{fig3pic} left. 
The restriction functor is depicted by having a $G$-line split into $H_1$ and $H_2$ lines, see Figure~\ref{fig3pic} middle. 
One can then build diagrams for compositions of these functors, see Figure~\ref{fig3pic} right, for instance. These graphs come with projections onto $\RR^1$, to keep track of the order of composition of functors. 

\vspace{0.1in}

\begin{figure}
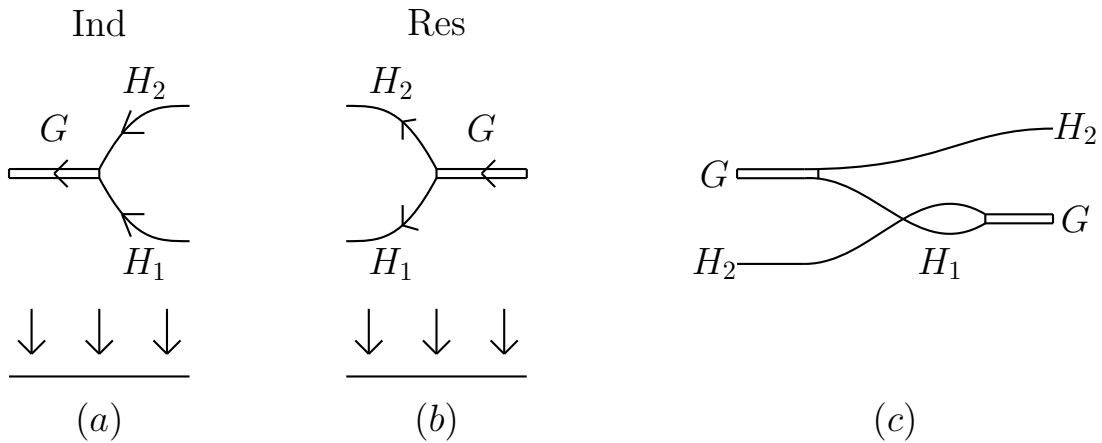

\begin{center}

\end{center}
    \caption{Diagrams of the induction (a) and restriction (b) functors. Diagram (c) is a composition of one restriction, one permutation, and one induction functor, going from the category of $G\times H_2$-modules to that of $H_2\times G$-modules. The composition is depicted and read from right to left.}
    \label{fig3pic}
\end{figure}

Natural transformations between these compositions can be naturally depicted by foams that extend between such diagrams. First off, identity natural transformation from the induction functor to itself (respectively, from the  restriction functor to itself) is depicted by the direct product foam, the graph depicting this functor times the unit interval $[0,1]$, see Figure~\ref{fig1_1}. 

\vspace{0.1in} 

\begin{figure}[ht]
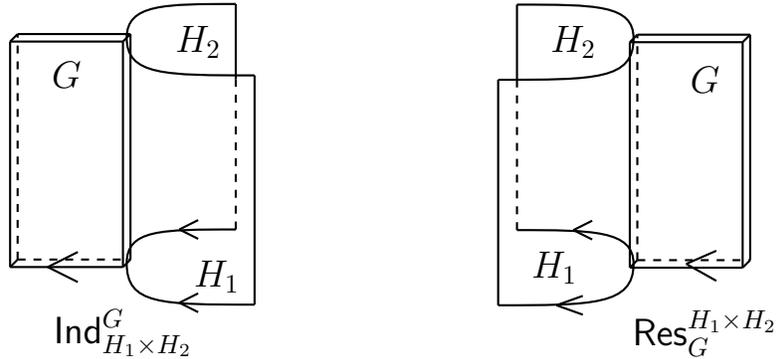

\begin{center}

\end{center}
 
\caption{Identity natural transformations on induction and restriction functors $\indf_{H_1\times H_2}^G$ and 
 $\resf_{G}^{H_1\times  H_2}$, respectively.}
\label{fig1_1}
 
\end{figure}

Singular lines in these foams are referred to as \emph{seam}  lines.
A natural transformation $a$ from the induction functor to itself may be denoted by a dot on a seam line, labeled $a$, see Figure~\ref{fig1_7} left, and likewise for an endomorphism of the restriction functor. A central element $c\in Z(\kk G)$ in the center $Z(\kk G)$ of the group algebra $\kk G$ is denoted by a dot floating in a facet $G$ labeled $c$, see Figure~\ref{fig1_7} right. 

\vspace{0.1in} 

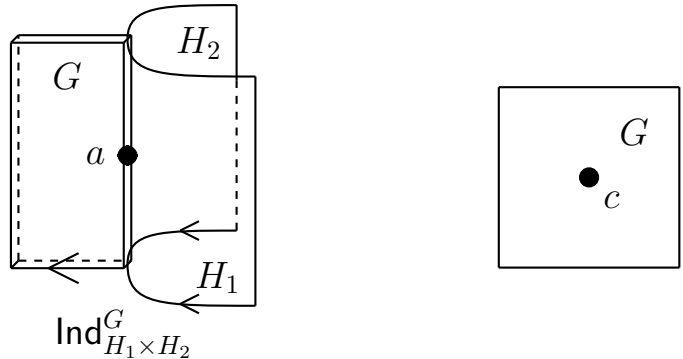
\begin{figure}
\begin{center}
    
\begin{tikzpicture}[scale=0.6,decoration={
    markings,
    mark=at position 0.60 with {\arrow{>}}}]

\begin{scope}[shift={(0,0)}]


\draw[thick] (-2,1) -- (2,1);
\draw[thick] (2,1) -- (2,6);
\draw[thick] (-2,1) -- (-2,6);
\draw[thick] (-2,6) -- (2,6);
 
\draw[thick,dashed] (-1.8,1.2) -- (2.2,1.2); 

\draw[thick,dashed] (-1.8,1.2) -- (-1.8,6.2);

\draw[thick] (2.2,1.2) -- (2.2,6.2);

\draw[thick] (-1.8,6.2) -- (2.2,6.2);

\draw[thick,postaction={decorate}] (6.5,0) .. controls (3,0) and (2.1,0) .. (2.1,1.1);

\draw[thick,postaction={decorate},dashed] (6,2) .. controls (3,2) and (2.1,2) .. (2.1,1.1);

\draw[thick] (6.5,0) -- (6.5,5.1);

\node at (-0.5,4.75) {\Large $G$};

\node at (4.5,1) {\Large $H_1$};

\node at (4.5,6) {\Large $H_2$};

\draw[thick,postaction={decorate}] (6.5,5.1) .. controls (3.1,5.1) and (2.1,5.1) .. (2.1,6.1);

\draw[thick,postaction={decorate}] (6,7) .. controls (3.1,7) and (2.1,7) .. (2.1,6.1);

\draw[thick,dashed] (6,2) -- (6,5);

\draw[thick] (6,5) -- (6,7);

\draw[thick] (-2,1) -- (-1.8,1.2);

\draw[thick] (2,1) -- (2.2,1.2);

\draw[thick] (-2,6) -- (-1.8,6.2);

\draw[thick] (2,6) -- (2.2,6.2);

\node at (1.5,-0.8) {\Large $\Ind_{H_1\times H_2}^G$};

\draw[thick] (0,1) -- (.8,1.4);

\draw[thick] (0,1) -- (.8,.6);

\draw[thick,fill] (2.35,3.5) arc (0:360:0.25);

\node at (1.25,3.5) {\Large $a$};

\end{scope}
 
\begin{scope}[shift={(13,1.5)}]
\draw[thick] (-0.5,-0.5) -- (-0.5,4.5);
\draw[thick] (-0.5,4.5) -- (4.5,4.5);
\draw[thick] (-0.5,-0.5) -- (4.5,-0.5);
\draw[thick] (4.5,-0.5) -- (4.5,4.5);

\draw[thick,fill] (2.35,2) arc (0:360:0.25);

\node at (2.65,1.40) {\Large $c$};

\node at (3.5,3.5) {\Large $G$};

\node at (2,-2) {};
\end{scope}

\end{tikzpicture}
\end{center}
    \caption{Left: natural transformation $a$, an endomorphism of the induction functor. Right: central element $c$ of $\kk G$.}
    \label{fig1_7}
\end{figure}

The functor $V\otimes -$ of the tensor product with a representation $V$ of $G$ is denoted by a dot on a line, with label $V$ and the regions to the sides of the dot labeled $G$, see Figure~\ref{fig1_8} left. 
Identity natural transformation $V\otimes - \Rightarrow V\otimes - $
is depicted by a vertical line (\emph{defect or seam line}) labeled $V$, see Figure~\ref{fig1_8} middle. 
A homomorphism $\gamma:V_1\lra V_2$ of $G$-modules induces a natural transformation $V_1\otimes - \lra V_2\otimes -$ of the functors $V_1\otimes -$ and $V_2\otimes -$, which we also denote by $\gamma$; it is depicted by a dot on a \emph{defect line} for $V$, see Figure~\ref{fig1_8} right, with the defect line label changing from $V_1$ to $V_2$.

\begin{figure}
\begin{center}
    \begin{tikzpicture}[scale=0.6]

\begin{scope}[shift={(0,0)}]
    \draw[thick] (-1,0) -- (5,0);
    \draw[thick,fill] (2.25,0) arc (0:360:0.25);  
    \node at (-0.6,-0.75) {\Large $G$};
    \node at (2,-1) {\Large $V$};
    \node at (4.6,-0.75) {\Large $G$};
    \node at (0,-2.5) {};
    \end{scope}

\begin{scope}[shift={(9.25,0)}]
    \draw[thick] (-0.5,0) rectangle (5.5,4);

    \draw[thick] (2.5,0) -- (2.5,4);

    \node at (2.5,-0.75) {\Large $V$};
    \node at (2.5, 4.75) {\Large $V$};

    \node at (1,2.5) {\Large $G$};
    \node at (4,2.5) {\Large $G$};
\end{scope}

\begin{scope}[shift={(18.5,0)}]

    \draw[thick] (-0.5,0) rectangle (5.5,4);
 
    \draw[thick] (2.5,0) -- (2.5,4);
    \draw[thick,fill] (2.75,2) arc (0:360:0.25);
    \node at (1.75,2) {\Large $\gamma$};

    \node at (2.5,-0.75) {\Large $V_1$};
    \node at (2.5, 4.75) {\Large $V_2$};
    
    \node at (0.5,2.75) {\Large $G$};
    \node at (4.5,2.75) {\Large $G$};
\end{scope}

    \end{tikzpicture}
\end{center}
    \caption{Left: notation for the functor $V\otimes -$ of the tensor product with $V$. Middle: identity natural transformation on $V\otimes -$. Right: natural transformation $\gamma$: $V_1\otimes - \lra V_2\otimes -$.}
    \label{fig1_8}
\end{figure}
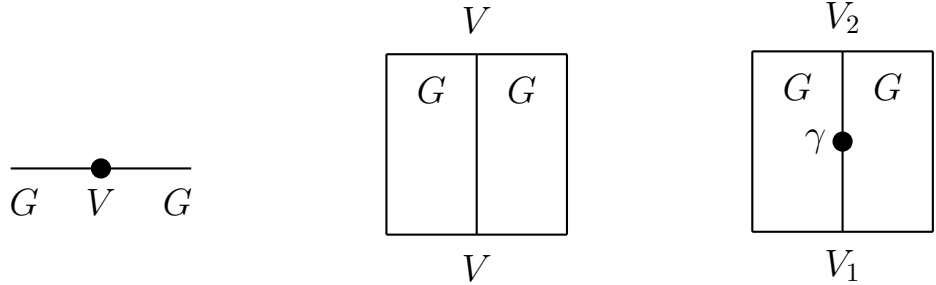

%
%

\section{Foams for the iterated wreath products of \texorpdfstring{$S_2$'s}{S2}}  
\label{foams}

\subsection{Iterated wreath products of \texorpdfstring{$S_2$'s}{S2s}}
\label{subsection:iterated-wreath-products}
For some background on the wreath product, see \cite{CST, OOR, IW1, IW2, IO}. Denote by $G_n$ the $n$-th iterated wreath product of the symmetric group $S_2$. It can be defined as the group of symmetries of the full binary tree $T_n$ of depth $n$. This binary tree has a root, $2^n$ leaf vertices, and all paths from the root to leaf vertices have length $n$. The tree $T_n$ has $2^{n+1}-1$ vertices. The leaf vertices can be naturally labeled from $1$ to $2^n$ inductively on $n$ so that the vertices of the left branch are labeled by $1$ through $2^{n-1}$ and those of the right branch are labeled by $2^{n-1}+1$ through $2^n$. See Figure~\ref{figT} for the case when $n=4$.

\vspace{0.1in}

\begin{figure}
\begin{center}

\end{center}

\caption{Tree $T_4$.}
\label{figT}
\end{figure} 

For small values of $n$, the group $G_n$ has the following form: 
\begin{itemize}
    \item $G_0=\{1\}$ is the trivial group, 
    \item $G_1=S_2$ is the symmetric group of order two,
    \item $G_2 = S_2\wr S_2 = (S_2\times S_2)\rtimes S_2$  has order $8$ and is isomorphic to the dihedral group $D_4$. 
\end{itemize}
Note that group $G_n$ has order $2^{2^n-1}$. 

\vspace{0.1in} 

The group $G_n$ has an index two subgroup naturally isomorphic to $G_{n-1}\times G_{n-1}$, which we also denote by 
\begin{equation}
    G^{(1)}_{n-1}\ := \ G_{n-1}\times G_{n-1} \xhookrightarrow{\iota_{n-1}} G_n. 
\end{equation}
The embedding consists of symmetries that fix the two branches of the tree, one to the left and the other to the right, of the root. The inclusion of this subgroup is denoted by $\iota_{n-1}$.  There is a coset decomposition 
\begin{equation}\label{eq_coset} 
G_n  = G^{(1)}_{n-1} \sqcup  G_{n-1}^{(1)}\, \beta_{n} = G^{(1)}_{n-1} \sqcup  \beta_{n} \, G^{(1)}_{n-1} , 
\end{equation} 
where $\beta_{n}$ is the involution that transposes the left and right branches of $T_n$. Notice the coincidence of left and right cosets 
\[ (G_{n-1}\times G_{n-1})\, \beta_{n} = \beta_{n}\, (G_{n-1}\times G_{n-1}),
\] 
which holds for cosets of any index two subgroup. In particular, the left and right cosets are also double cosets. 
Furthermore, for $g_1,g_2\in G_{n-1}$, 
\[ (g_1, g_2) \, \beta_{n} = \beta_{n} \, (g_2, g_1), 
\]
that is, moving through $\beta_{n}$ switches the order of the two terms in the product $G_{n-1}\times G_{n-1}$. Denote by $\tau$ the transposition involution of $G_{n-1}^{(1)}=G_{n-1}\times G_{n-1}$, 
\begin{equation} \label{eq_tau} 
    \tau (g_1, g_2) \ := \ (g_2, g_1), \ \ g_1,g_2\in G_{n-1}. 
\end{equation}
Then 
\begin{equation}
    \tau(g) = \beta_n\, g \, \beta_n, \ \ g\in G_{n-1}^{(1)}. 
\end{equation}

\vspace{0.1in} 

By induction on $n$, we can canonically identify $G_n$ with a subgroup of the symmetric group $S_{2^n}$. When $n=0$, both $G_0$ and $S_{2^0}=S_1$ are the trivial group. For the induction step, given an inclusion $j_{n-1}: G_{n-1}\hookrightarrow S_{2^{n-1}}$, we  realize $G_n\subset S_{2^n}$ as the subgroup generated by:
\begin{itemize}
    \item permutations of $\{1,\dots, 2^{n-1}\}$ in $G_{n-1}$,
    \item permutations of $\{2^{n-1}+1,\dots, 2^n\}$ in $G_{n-1}$ (obtained by shifting all indices by $2^{n-1}$),
    \item permutation $\beta_n= (1,2^{n-1}+1)(2,2^{n-1}+2)\cdots (2^{n-1},2^n).$
\end{itemize}

Here we inductively identify $\beta_n\in G_n$ with its image in $S_{2^n}$. The subgroup  
$G_{n-1}^{(1)}$
is given by products of permutations of the first and the second type on the above list. As we have already mentioned, it is a normal subgroup of index $2$, with $\{1,\beta_n\}$ a set of coset representatives.

\subsection{A description of the center of \texorpdfstring{$G_n$}{Gn}}
\label{subsection:center-Gn}
The center of $G_n$ is an order two subgroup, 
\begin{equation}
    Z(G_n) = \{1,c_n\}, 
    \qquad  
    c_n:=(1,2)(3,4)\cdots (2^n-1,2^n).
\end{equation}

\vspace{0.1in}

Define $G_{n-k}^{(k)}$ as the subgroup $(G_{n-k})^{\times 2^{k}}\subset G_n$ given by permutations that fix all 
nodes of the full binary tree at distance at most $k-1$ from the root. There is a chain of inclusions 
\begin{equation}
    \{1\}=G_{0}^{(n)} \subset G_1^{(n-1)} \subset \ldots \subset G_{n-2}^{(2)}\subset G_{n-1}^{(1)}\subset G_n^{(0)}=G_n. 
\end{equation}
Each inclusion 
\begin{equation}
     G_{n-k-1}^{(k+1)} \subset G_{n-k}^{(k)} 
\end{equation}
is that of an index $2^{2^k}$ normal subgroup, with the quotient isomorphic $S_2^{\times 2^k}$.

\subsection{Induction and restriction bimodules}
\label{subsection:induction-restriction-bimodules}
We denote $\kk G_n$, viewed as a bimodule over itself, by $(n)$. Denote $\kk G_{n-1}^{(1)} := \kk (G_{n-1}\times G_{n-1})$ by $(n-1,n-1)$ and even by $(n-1)^{(1)}$, to further compactify the notation, and extend these notations to tensor products of bimodules. For instance
\[  (n)_{(n-1)^{(1)}}(n)\ := \ \kk G_n \otimes_{\kk G_{n-1}^{(1)}} \kk G_n,
\] 
is naturally a $\kk G_n$-bimodule. 

\vspace{0.1in} 

Using notations from~\cite[Section 3.2]{Kh1}, we write down the biadjointness maps: 
\begin{enumerate}
    \item\label{item:biadjoint1} $\alpha_{n-1}^n: (n)_{(n-1)^{(1)}} (n)  \lra (n)$, where $x\otimes y \mapsto xy, \ x,y\in (n) = \kk G_n$,   
    \item\label{item:biadjoint2} $\gamma_{n-1}^n: (n-1)^{(1)} \lra {}_{(n-1)^{(1)}} (n)_n(n)_{(n-1)^{(1)}}$, where $x\mapsto x\otimes 1 = 1\otimes x$, $x\in (n-1)^{(1)}$,
    \item\label{item:biadjoint3}  $\alpha^{n-1}_n: {}_{(n-1)^{(1)}} (n)_n(n)_{(n-1)^{(1)}}\cong {}_{(n-1)^{(1)}} (n)_{(n-1)^{(1)}}\lra (n-1)^{(1)} $ takes $g\in (n)$ to $p_{n-1}(g)\in (n-1)^{(1)}$ by 
    $p_{n-1}(g) = 
    \begin{cases}
    g &  \mathrm{if} \  g\in (n-1)^{(1)}, \\
    0 & \mathrm{otherwise}. 
    \end{cases}$
    \item\label{item:biadjoint4}  $\gamma^{n-1}_n: (n) \lra (n)_{(n-1)^{(1)}}(n)$, where $x\mapsto  1 \otimes x + \beta_{n}\otimes \beta_{n}x$, and $x\in (n)$.  
\end{enumerate}
These four bimodule maps (or morphisms of functors) are represented by the four diagrams in Figure~\ref{fig1_3}.


\begin{figure}
\begin{center}
\begin{tikzpicture}[scale=0.6]

\begin{scope}[shift={(0,0)}]
\draw[thick, dashed] (-2.5,0) -- (2.5,0);
\draw[thick, dashed] (-2.5,3) -- (2.5,3);
\draw[thick,<-] (0,2) arc (90:180:2);
\draw[thick] (2,0) arc (0:90:2);
\node at (-2,-.5) { $I$};
\node at (2,-.5) { $R$};
\node at (0,-1.5) { $\alpha_{n-1}^n$};
\node at (2,2) { $n$};
\node at (0,.65) { $(n-1)^1$};
\end{scope}

\begin{scope}[shift={(7,0)}]
\draw[thick, dashed] (-2.5,0) -- (2.5,0);
\draw[thick, dashed] (-2.5,3) -- (2.5,3);
\draw[thick,->] (-2,3) arc (180:270:2);
\draw[thick] (0,1) arc (270:360:2);
\node at (-2,3.5) {$R$};
\node at (2,3.5) {$I$};
\node at (0,-1.5) { $\gamma_{n-1}^n$};
\node at (0,2) { $n$};
\node at (1.5,.55) { $(n-1)^1$};
\end{scope}

\begin{scope}[shift={(14,0)}]
\draw[thick, dashed] (-2.5,0) -- (2.5,0);
\draw[thick, dashed] (-2.5,3) -- (2.5,3);
\draw[thick] (0,2) arc (90:180:2);
\draw[thick,->] (2,0) arc (0:90:2);
\node at (-2,-.5) {$R$};
\node at (2,-.5) {$I$};
\node at (0,-1.5) {$\alpha_{n}^{n-1}$};
\node at (0,.5) {$n$};
\node at (1.55,2.48) {$(n-1)^1$};
\end{scope}

\begin{scope}[shift={(21,0)}]
\draw[thick, dashed] (-2.5,0) -- (2.5,0);
\draw[thick, dashed] (-2.5,3) -- (2.5,3);
\draw[thick] (-2,3) arc (180:270:2);
\draw[thick,<-] (0,1) arc (270:360:2);
\node at (-2,3.5) {$I$};
\node at (2,3.5) {$R$};
\node at (0,-1.55) {$\gamma_{n}^{n-1}$};
\node at (1.35,.5) {$n$};
\node at (0,2.15) {$(n-1)^1$};
\end{scope}
\end{tikzpicture}
\end{center}

\caption{Diagrams for natural biadjointness transformations. Letters $R$ and $I$ stand for \emph{restriction} and \emph{induction} functors, respectively.} 
\label{fig1_3}
\end{figure}
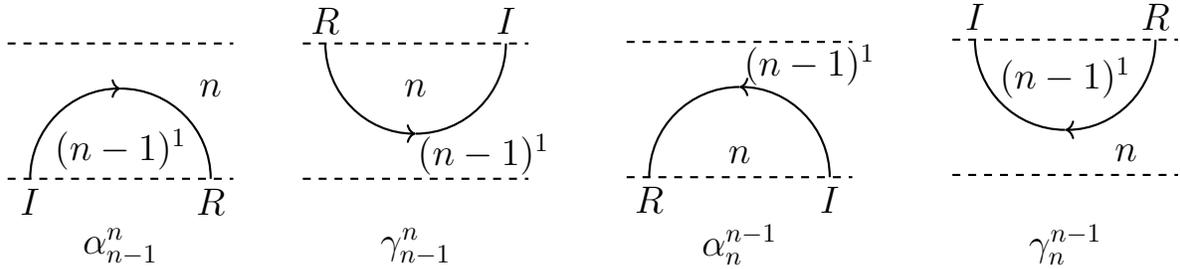

\begin{prop} These four natural transformations turn functors $I_{n-1}^n$ and $R_n^{n-1}$ into a cyclic biadjoint pair. 
\end{prop} 


We refer the reader to~\cite[Section 3.2]{Kh1} for details, in the general case of a finite index subgroup. In particular, planar isotopy relations between compositions of these cups and caps hold, see Figure~\ref{fig_isotopy}, where the general case of $H\subset G$ of finite index is shown. 

\vspace{0.1in} 

For our specific case, we have  
obvious relations in Figure~\ref{fig1_4}. 

\vspace{0.1in}

\begin{figure}
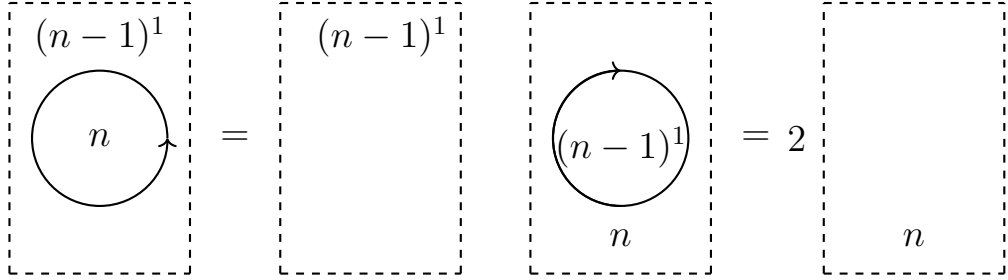

\begin{center}
 
\end{center}

\caption{Some simple relations on diagrams.   } 
\label{fig1_4}
\end{figure}

We now refine these planar diagrams to a foam description for these and related intertwiners between compositions of induction and restriction functors $I_{n-1}^n$ and $R_n^{n-1}$. We denote the induction and restriction functors by trivalent vertices in graphs as shown in Figure~\ref{fig1_5} left. 

\vspace{0.1in}

\begin{figure}
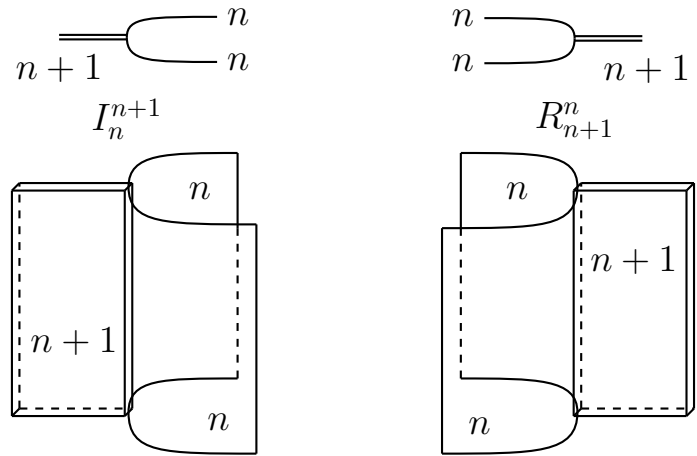

\begin{center}
%
\end{center}

\caption{Induction and restriction functors $I_{n}^{n+1},R_{n+1}^{n}$ (top figures) and identity natural transformation on them (bottom figures).} 
\label{fig1_5}
\end{figure} 

Identity natural transformations for these functors are shown in Figure~\ref{fig1_5} right.
The four biadjointness transformations are shown in Figure~\ref{fig1_6}. 

\vspace{0.1in}

\begin{figure}
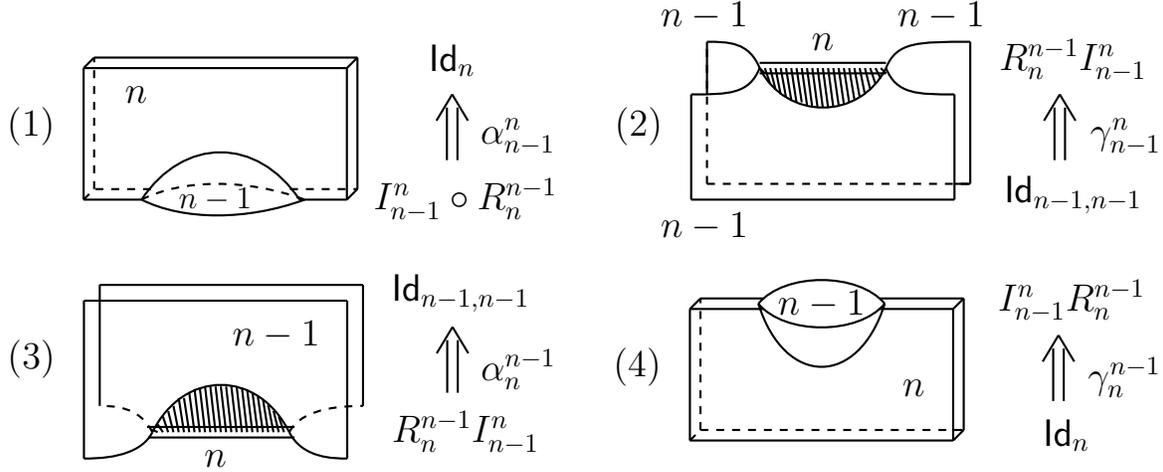

\begin{center}

%
\end{center}

\caption{The four biadjointness transformations for $I_{n-1}^n,R_n^{n-1}$. } 
\label{fig1_6}
\end{figure} 

Biadjointness relations translate into the isotopy properties of foam glued from these four foams. One out of four possible isotopy relations is shown in Figure~\ref{fig_isotopy}. 


\subsection{Mackey induction-restriction formula and decomposition of \texorpdfstring{$\Ind$-$\Res$}{Ind-Res} functor}
\label{subsection:Mackey-induction-restriction-formula}

To the transposition automorphism $\tau$ of $G_{n-1,n-1}$ taking $g_1\times g_2$ to $g_2\times g_1$, we associate $\kk G_{n-1,n-1}$-bimodule $B_{12}$ given by $\kk G_{n-1,n-1}$ with 
 the left action twisted by $\tau$. Denote by 
$T_{12}$ the invertible endofunctor of $\kk G_{n-1,n-1}\dmod$ given by tensoring with $B_{12}.$  

\begin{prop}
\label{prop_canonical_decomp_functors}
There is a canonical decomposition of functors
\begin{equation}\label{eq_func_isom}
    R^{n-1}_n\circ I^n_{n-1} \ \cong \ \Id  \oplus T_{12}.
\end{equation}
\end{prop} 
Diagrams for the three functors in this isomorphism are shown in Figure~\ref{fig2_1}. In Figure~\ref{fig2_2}, we describe the direct sum decomposition via foams. 

\proof
The composition $R^{n-1}_n\circ I^n_{n-1}$ is given by tensoring with the $G_{n-1,n-1}$-bimodule $\kk G_n$. The proposition follows from the Mackey induction-restriction formula. Namely, $G_n$ decomposes as the disjoint union of two $(G_{n-1,n-1},G_{n-1,n-1})$-cosets. One of them is $G_{n-1}$, giving the identity functor as a direct summand of $R^{n-1}_n\circ I^n_{n-1}$. The other is $G_{n-1}\tau G_{n-1}$. The latter coset is represented by the transposition of the two copies of $G_{n-1}$ in $G_{n-1}\tau G_{n-1}$, corresponding to the bimodule $B_{12}$ above and the functor $T_{12}$. 
\endproof

\begin{figure}
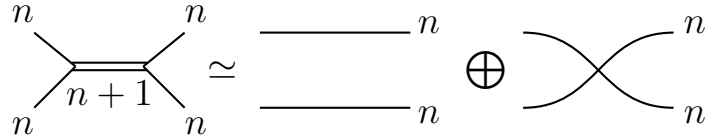

\begin{center}

\end{center}

\caption{Diagrams for the three functors in \eqref{eq_func_isom}.} 
\label{fig2_1}
\end{figure}

\begin{figure}
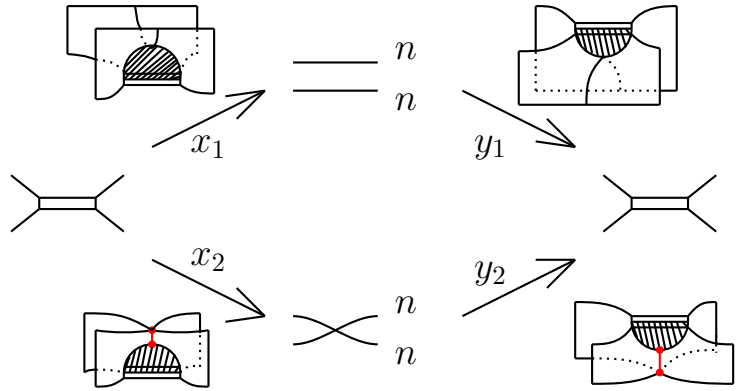

\begin{center}
%
\end{center}

\caption{Maps (foams) describing the direct sum decomposition in \eqref{eq_func_isom}. 
}
\label{fig2_2}
\end{figure}

\vspace{0.1in} 

The  direct sum decomposition property translates into the following relations: 
\begin{align*}
&  y_1 x_1 + y_2 x_2 =\id_{RI}, \\
x_1 y_1 &=\id, \qquad  \qquad  x_2 y_2 = \id, \\
x_1 y_2 &= 0, \ \qquad \qquad x_2 y_1 = 0.
\end{align*}
Foam equivalents of these relations are shown in Figure~\ref{fig2_3}. 

\vspace{0.1in}

\begin{figure}
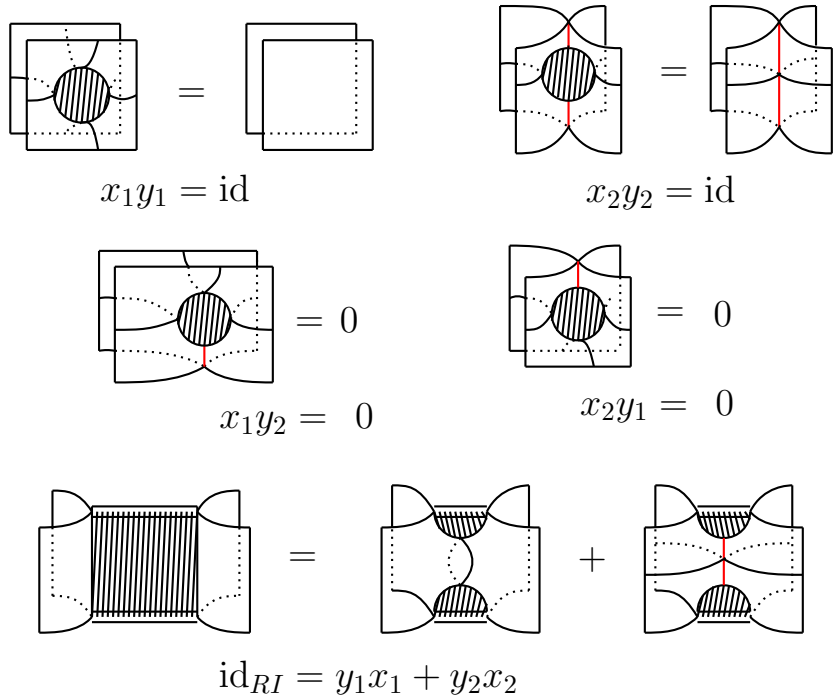

\begin{center}

\end{center}

\caption{Direct sum decomposition relations.}
\label{fig2_3}
\end{figure}

Here and in the rest of the paper the reader should keep in mind that our foams are pictures, in a sense, but guided by topological interpretations. 
 
When depicted in $\R^3$, the foams for the maps $x_2,y_2$ are immersed, and contain ``overlap" or ``intersection" lines or seams. Using the biadjointness of the induction and restriction functors, these immersed foams can be converted into foams in Figure~\ref{fig2_4}, depicting mutually-inverse natural transformations, denoted $\ell(\beta_n)$ and $\ell(\beta_n)'$, respectively, between functors $R_n^{n-1}$ and $T_{12}R_n^{n-1}$. 

\vspace{0.1in}

\begin{figure}
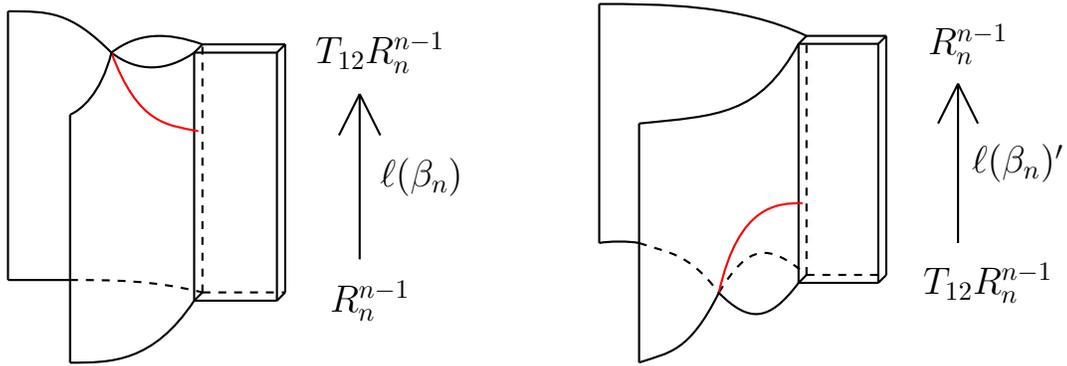

\begin{center}

\end{center}

\caption{Intersection seams giving mutually-inverse functor isomorphisms $ T_{12}R_n^{n-1}\cong R_n^{n-1}$.} 
\label{fig2_4}
\end{figure}

Reflecting these diagrams about the $yz$-plane gives dual (biadjoint) mutually-inverse natural transformations between the functors $I_{n-1}^n$ and $I_{n-1}^nT_{12}$. Figures~\ref{fig2_5} and~\ref{fig2_5_1} depict relations  that these two maps are mutually-inverse isomorphisms. 

\vspace{0.1in}

\begin{figure}
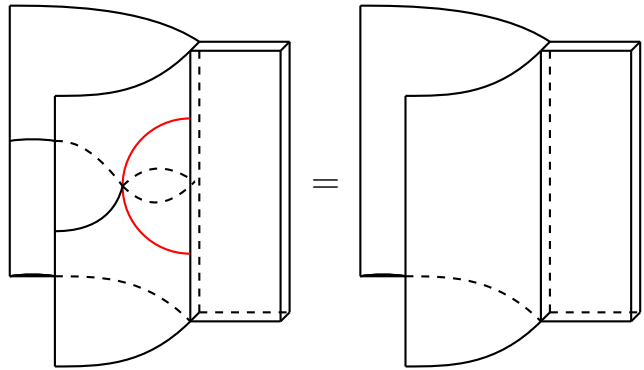

\begin{center}

\end{center}

\caption{Relation $\ell(\beta_n)'\ell(\beta_n)=\id$ allows one to undo an immersion seam that goes out and back into an $(n,n-1)$-seam.} 
\label{fig2_5}
\end{figure}

\begin{figure}
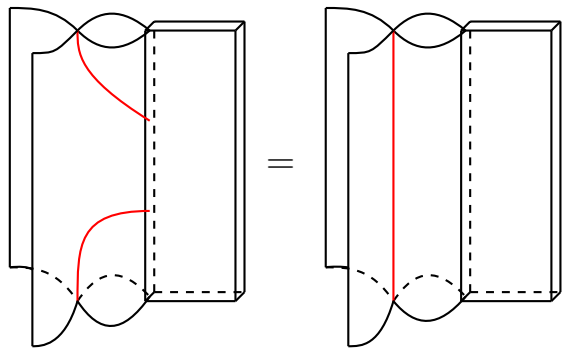

\begin{center}
%
\end{center}

\caption{Relation $\ell(\beta_n)\ell(\beta_n)'=\id$ cancels two adjacent immersion points on an $(n,n-1)$-seam.}
\label{fig2_5_1}
\end{figure}

\begin{figure}
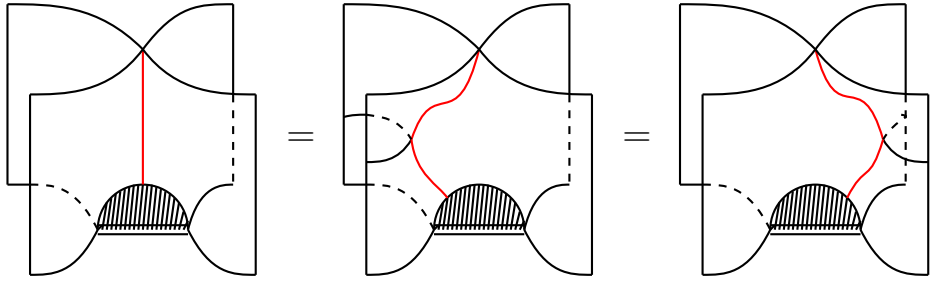

\begin{center}
%
\end{center}

\caption{ Deforming an immersion seam and moving its endpoint.} 
\label{fig2_6}
\end{figure}

Endpoints of immersion seams can move freely along the $(n,n-1)$-seam lines, see Figure~\ref{fig2_6}. 
Deforming intersecting facets of these foams embedded in $\R^3$ gives a number of obvious relations, one of which is shown in Figure~\ref{fig4_7}.  

\vspace{0.1in}

\begin{figure}
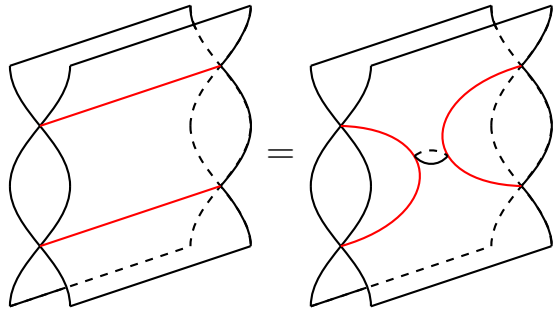

\begin{center}
%
\end{center}

    \caption{An isotopy of immersed surfaces in $\mathbb{R}^3$. Intersection lines are shown in red.}
    \label{fig4_7}
\end{figure}

Together, these relations allow to reduce the number of immersion points along an $(n,n-1)$-seam to one or none. If such a seam closes into a disk which carries no additional decorations, it can then be reduced  to either two parallel planes (top left relation in Figure~\ref{fig2_3}), if the original number of immersion endpoints along a seamed circle is even, or to $0$ (either one of the middle row relations in Figure~\ref{fig2_3}), if the number of immersion endpoints along a seamed circle is odd. 

\vspace{0.1in}

Functor $T_{12}$ is just the permutation functor, induced by the transposition of two copies of the group $G_{n-1}$ in the direct product, and satisfies the relations
\[ T_{12}T_{12} = \id, \qquad \qquad   T_{12}T_{23}T_{12}=T_{23}T_{12}T_{23}.
\] 
The corresponding relations on immersed foams are given in Figure~\ref{fig7_1}. The last relation is induced by a foam with three facets and a triple intersection point of these facets, where three intersection seams meet. Also see \cite[Figure 12]{CS98}.

\vspace{0.1in}

\begin{figure}
\begin{center}

\end{center}
    \caption{Reidemeister moves on immersed foams.}
    \label{fig7_1}
\end{figure}


For a seam $C$ that is closed into a circle and bounds a disk $D$, as in the two top rows of Figure~\ref{fig2_3}, consider the number $m$ of immersion points on it (points where a red segment ends). Using the above relations preserving the parity of $m$ we can reduce the foam to have at most one immersion point along $C$. From Figure~\ref{fig2_3} relations we then see that a diagram evaluates to $0$ if $m$ is odd. If $m$ is even, diagram can be simplified to one without $C$ and disk $D$, and immersion endpoints along $C$ matched in pairs. 

\vspace{0.1in}

There are also obvious isotopy relations, some of which can be obtained from Figure~\ref{fig_isotopy} by substituting a direct product $H_1\times H_2$ for $H$ and converting diagrams into foams, see also \cite{Kh3}. 

\subsection{Central elements and bubbles}
\label{subsection:central-elements-bubbles}
Recall that the center $Z(G_n)\cong S_2$ is the symmetric group of order two, with the nontrivial element $c_n = (1,2)(3,4)\cdots (2^n-1,2^n)$. Via the inclusion $\iota_{n-1}$ this element can be defined inductively as $c_n = \iota_{n-1}(c_{n-1}\times c_{n-1})$. We denote $c_n$ by a dot on a facet labeled $n$, see Figure~\ref{fig2_9} left. 
Square of the dot is the identity, see Figure~\ref{fig2_9} right.
Element $c_n$ can also be thought of as an endomorphism of the identity functor on $\kk G_n\dmod$.

\vspace{0.1in} 

The center $Z(\kk G_n)$ is a commutative algebra with a basis parametrized by conjugacy classes of $G_n$. Iterating the bubbles and dots construction allows us to construct various elements of the center. 

\vspace{0.1in}

\begin{figure}
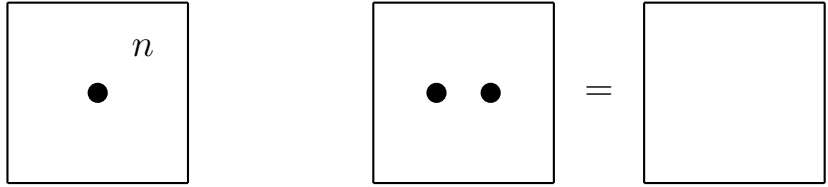

\begin{center}

\end{center}

\caption{The central element $c_n$ and a relation on it: the square of the dot is the identity.} 
\label{fig2_9}
\end{figure}

As a first example, 
taking an $n$-facet, we can create an $(n-1)$-bubble on it and insert a dot into one or both facets of the bubble, see  Figures~\ref{fig2_10} and~\ref{fig2_12}. 

\vspace{0.1in}

\begin{figure}
\begin{center}

%
\end{center}

\caption{The simplest relations on $c$-bubbles. Top row: the value of the empty bubble is $2$ using the composed map $x\mapsto x+\beta_n^2 x = 2x$, where $x\in \kk G_n$, in Section~\ref{subsection:central-elements-bubbles} (here, we do not multiply by $c_{n-1}^{(1)}$ since there are no defects and $\beta_n$ is the involution $\beta_n^2 =1$). Second row, left: a dot can be placed anywhere on the bubble. 
Second row, right: we use the fact that $c_n^2 = 1$; also see Figure~\ref{fig2_9}.}
\label{fig2_10}
 
\end{figure}

To compute the corresponding central elements, we factor these foams into a composition of elementary foams and compute the corresponding natural transformations. For instance, bubble with a single dot is a composition of three elementary foams, see Figure~\ref{fig2_12}. 

\vspace{0.1in}

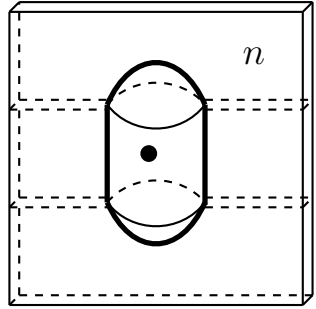
\begin{figure}
\begin{center}
\begin{tikzpicture}[scale=.65]

\draw[thick] (0,-1.1) -- (0,4.9);
\draw[thick,dashed] (0.2,-.9) -- (6.2,-.9);
\draw[thick] (0,-1.1) -- (6,-1.1);
\draw[thick] (0,4.9) -- (.2,5.1);

\draw[thick] (0,4.9) -- (6,4.9);
\draw[thick] (0.2,5.1) -- (6.2,5.1);
\draw[thick,dashed] (0,-1.1) -- (.2,-.9);
\draw[thick] (6,-1.1) -- (6,4.9);
\draw[thick] (6,-1.1) -- (6.2,-.9);
\draw[thick] (6.2,-0.9) -- (6.2,5.1);
\draw[thick] (6,4.9) -- (6.2,5.1);
\draw[thick,dashed] (0.2,-.9) -- (0.2,0.9);
\draw[thick,dashed] (0.2,1.1) -- (0.2,2.9);
\draw[thick,dashed] (0.2,3.1) -- (0.2,5.1);
\draw[line width=.7mm] (2,1) -- (2,3);

\draw[thick,dashed] (0,0.9) -- (0.2,1.1);
\draw[thick,dashed] (0,2.9) -- (0.2,3.1);
\draw[thick,dashed] (0,0.9) -- (2,0.9); 
\draw[thick,dashed] (0.2,1.1) -- (2,1.1); 
\draw[thick,dashed] (0,2.9) -- (2,2.9); 
\draw[thick,dashed] (0.2,3.1) -- (2,3.1);
\draw[thick,dashed] (4,0.9) -- (6,0.9); 
\draw[thick,dashed] (4,1.1) -- (6.2,1.1); 
\draw[thick,dashed] (4,2.9) -- (6,2.9);
\draw[thick,dashed] (4,3.1) -- (6.2,3.1);
\draw[thick,dashed] (6,0.9) -- (6.2,1.1);
\draw[thick,dashed] (6,2.9) -- (6.2,3.1);

\draw[thick,fill] (3,2) arc (0:360:0.25);
\draw[thick] (2,3) .. controls (2.5,2.35) and (3.5,2.35) .. (4,3);
\draw[thick,dashed] (2,3) .. controls (2.5,3.6) and (3.5,3.6) .. (4,3);
\draw[line width=.7mm] (2,3) .. controls (2.5,4.15) and (3.5,4.15) .. (4,3);
\draw[line width=.7mm] (4,.95) -- (4,3);
\draw[thick,dashed] (2,1) .. controls (2.5,1.6) and (3.5,1.6) .. (4,1); 
\draw[thick] (2,1) .. controls (2.5,0.35) and (3.5,0.35) .. (4,1); 
\draw[line width=.7mm] (2,1) .. controls (2.5,-.13) and (3.5,-.13) .. (4,1); 
\node at (5,4) {\Large $n$};
\draw[thick] (4,3) -- (4,2.9);
\end{tikzpicture}
\end{center}

\caption{An example of the composition $\kk G_n \rightarrow \kk G_n \otimes \kk G_{n-1}^{(1)}\otimes \kk G_n \rightarrow \kk G_n$. } 
\label{fig2_12}
\end{figure}

These three foams are local singular maximum and minimum, and adding a dot to a facet. 
The corresponding bimodule map is  composition 
\begin{equation*}
    \kk G_n \stackrel{\gamma_n^{n-1}}{\rightarrow} \kk G_n \otimes_{n-1} \kk G_n \stackrel{c_{n-1}\times 1}{\longrightarrow} \kk G_n \otimes_{n-1} \kk G_n \stackrel{\alpha_{n-1}^n}{\rightarrow} \kk G_n,
\end{equation*}
 where $\otimes_{n-1}$ denotes the tensor product over the subalgebra $\kk G_{n-1}^{(1)}$ and $\gamma_n^{n-1}$ and $\alpha^n_{n-1}$ are given by formulas \eqref{item:biadjoint1} and \eqref{item:biadjoint4}, see also Figure~\ref{fig1_6}. 
 
 \vspace{0.1in} 
 
 For $x\in \kk G_n$ we compute the composition 
\begin{equation*}
x\mapsto 1\otimes x+ \beta_n\otimes \beta_n x\mapsto 1\otimes  c_{n-1}^{(1)}x + \beta_n\otimes c_{n-1}^{(1)}\beta_n x\mapsto 
 c_{n-1}^{(1)}x + \beta_n  c_{n-1}^{(1)}\beta_n x.
 \end{equation*}
This endomorphism of the identity functor is the multiplication by the central element 
\begin{equation*}
c_{n-1}^{(1)} + \beta_n  c_{n-1}^{(1)}\beta_n=c_{n-1}^{(1)}+c_{n-1}^{(2)}=c_{n-1}\times 1 + 1\times c_{n-1}
\end{equation*} 
(where we skipped the inclusion map $\iota_n$), 
also implying the first relation in Figure~\ref{fig2_10}. 
Another easy computation gives the second  
relation in Figure~\ref{fig2_10}.

\vspace{0.1in}



Iterating the bubble construction, one can produce more general central elements of the group algebra $\kk G_n$. One can keep splitting some facets of the bubble into thinner facets and placing dots on some of these facets. An example is shown in Figure~\ref{fig2_11}, with the foam there describing the central element 
\begin{equation}
    c_{n-2}^{(1)}+c_{n-2}^{(2)}+c_{n-2}^{(3)}+c_{n-2}^{(4)}.
\end{equation}
Here $c_{n-2}^{(i)}$ stands for the $i$-th copy of $c_{n-2}$ in the direct product $G_{n-2}^{\otimes 4}\subset G_n$. 


\begin{figure}
\begin{center}
\begin{tikzpicture}[scale=1]

\begin{scope}[shift={(0,0)}]


\draw[thick] (0,0) -- (0,4);
\draw[thick] (0,0) -- (4,0);
\draw[thick] (4,0) -- (4,4);
\draw[thick] (0,4) -- (4,4);
\draw[thick] (3.75,2) arc (0:360:1.75);
\node at (3.5,3.5) {$n$};

\draw[thick,dotted] (0.25,2) .. controls (0.75,3.2) and (3.25,3.2) .. (3.75,2);


\draw[thick] (3,1.5) arc (0:360:1);

\node at (2,3.3) {$n-1$};

\draw[thick] (0.25,2) .. controls (0.35,1.6) and (0.9,1.3) .. (1,1.3); 
\draw[thick] (3,1.3) .. controls (3.1,1.3) and (3.65,1.6) .. (3.75,2); 

\draw[thick,dotted] (1,1.3) .. controls (1.25,1.8) and (2.75,1.8) .. (3,1.3);

\draw[thick] (1,1.3) .. controls (1.25,0.8) and (2.75,0.8) .. (3,1.3);

\draw[thick,fill] (2.25,0.95) arc (0:360:0.25);

\node at (2,2.05) {$n-2$};
\end{scope}

\begin{scope}[shift={(8,0.25)}]


\draw[thick] (-1,1.9) -- (0.25,1.9);
\draw[thick] (-1,2.1) -- (0.25,2.1);

\draw[thick] (3.75,1.9) -- (5,1.9);
\draw[thick] (3.75,2.1) -- (5,2.1);

\draw[thick] (0.25,2.1) .. controls (0.75,4) and  (3.25,4) .. (3.75,2.1);

\draw[thick] (3,0.8) .. controls (3.15,0.9) and (3.65,1.0) .. (3.75,1.9);

\draw[thick] (0.25,1.9) .. controls (0.35,1.0) and (0.9,0.9) .. (1,0.8);

\draw[thick] (0.25,1.9) -- (0.25,2.1);
\draw[thick] (3.75,1.9) -- (3.75,2.1);

\begin{scope}[shift={(0,-0.50)}]

\draw[thick] (1,1.4) .. controls (1.25,2.45) and (2.75,2.45) .. (3,1.4);
\draw[thick] (1,1.2) .. controls (1.25,0.15) and (2.75,0.15) .. (3,1.2);

\draw[thick] (1,1.2) -- (1,1.4);
\draw[thick] (3,1.2) -- (3,1.4);

\draw[thick, fill] (2.25,0.43) arc (0:360:0.25);

\node at (-0.1,1.50) {$n-1$};

\node at (3.35,0.55) {$n-2$};

\end{scope}

\node at (-0.75,2.4) {$n$};

\node at (2,2.7) {$n-1$};

\end{scope}

\end{tikzpicture}
\end{center}

\caption{A more complicated bubble describing a central element.
The middle cross-section of this bubble is shown on the right.  } 
\label{fig2_11}
\end{figure}
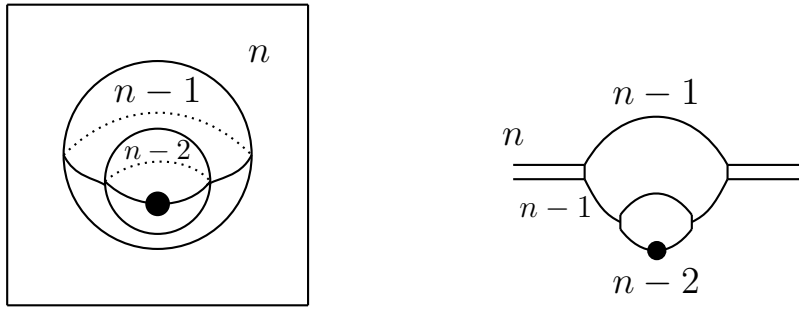

%
%

\section{Defect lines and networks}  \label{sec_defect}

\subsection{Tensoring with induced representations}


Here, $\ukk$ denotes the trivial representation of $H$. 

\begin{lemma}\label{lemma:ind-res-isom}
Let $H\subseteq G$ be a  subgroup, and $M$ a $G$-module. There is a natural in $M$ isomorphism \
\begin{equation}\Ind_H^G \circ \Res_G^H (M) \stackrel{\sim}{\rightarrow} \Ind_H^G (\ukk) \otimes M.
\end{equation}
Consequently, the composition of restriction and induction functors is isomorphic to the functor of tensor product with the induced representation $\Ind_H^G (\ukk)$. 
\end{lemma}

\proof
Define the map  
 \begin{equation}
 \varphi: \kk G\otimes_{\kk H} M \rightarrow (\kk G\otimes_{\kk H} \ukk) \otimes_{\kk} M, \quad 
 g\otimes m\mapsto (g\otimes \underline{1})\otimes gm.
 \end{equation}
 The module on the left is $\Ind_H^G \circ \Res_G^H (M)$, the one on the right is $\Ind_H^G (\ukk) \otimes M$. The map is a module map, natural in $M$. 
 
\vspace{0.1in} 
 
 Conversely, let $\psi:(\kk G\otimes_{\kk H} \ukk) \otimes_{\kk} M\rightarrow \kk G\otimes_{\kk H} M$ be given by $(g\otimes \underline{1})\otimes n\mapsto g\otimes g^{-1}n$. One can easily check that $\varphi$ and $\psi$ are inverses.
\endproof

Thus $H\subseteq G$ being a subgroup of $G$, 
\begin{equation}
\Ind_H^G \circ \Res_G^H\simeq V_H^G\otimes -
\end{equation}
is an isomorphism of functors. 
We denote by $\V_H^G := \Ind_H^G(\ukk)$ the induced representation of $G$. 

\vspace{0.1in}

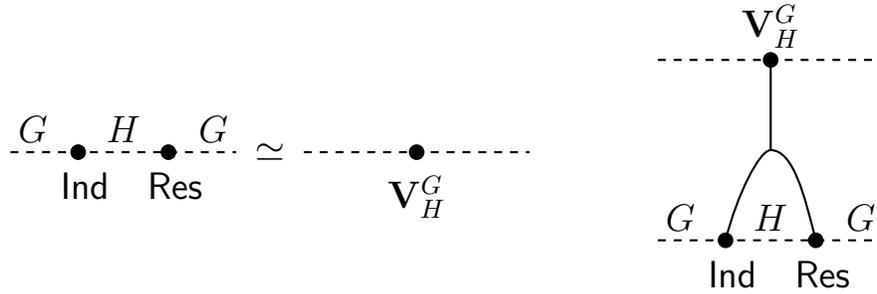
\begin{figure}[ht]
\begin{center}
\begin{tikzpicture}[scale=0.6]

\begin{scope}[shift={(0,0)}]


\draw[thick,dashed] (-1.5,0) -- (5.5,0);

\draw[thick,fill] (0.75,0) arc (0:360:0.25);
\draw[thick,fill] (3.75,0) arc (0:360:0.25);

\node at (0.85,-.75) {\Large $\Ind$};
\node at (3.40,-.75) {\Large $\Res$};

\node at (-0.75,0.65) {\Large $G$};
\node at ( 2,0.65) {\Large $H$};
\node at ( 4.75,0.65) {\Large $G$};

\node at (7.0,0) {\Large $\simeq$};

\begin{scope}[shift={(2.5,0)}]
\draw[thick,dashed] (6,0) -- (11,0);
\node at (8.5,-1) {\Large $\V_H^G$};
\draw[thick,fill] (8.75,0) arc (0:360:0.25);
\end{scope}
\end{scope}

\begin{scope}[shift={(19,-2)}]


\draw[thick,dashed] (-1.5,0) -- (5.5,0);

\draw[thick,fill] (0.75,0) arc (0:360:0.25);
\draw[thick,fill] (3.75,0) arc (0:360:0.25);

\node at (0.85,-.75) {\Large $\Ind$};
\node at (3.40,-.75) {\Large $\Res$};

\node at (-0.75,0.65) {\Large $G$};
\node at ( 2,0.65) {\Large $H$};
\node at ( 4.75,0.65) {\Large $G$};

\draw[thick] (2,2.5) -- (2,4);

\draw[thick,dashed] (-1.5,4) -- (5.5,4);
\draw[thick,fill] (2.25,4) arc (0:360:0.25);
\node at (2,4.9) {\Large $\V_H^G$};

\draw[thick] (0.5,0) .. controls (0.75,2) and (1.75,2.5) .. (2,2.5);
\draw[thick] (3.5,0) .. controls (3.25,2) and (2.25,2.5) .. (2,2.5);
\end{scope}

\end{tikzpicture}
\end{center}

    \caption{Left: diagrammatic notations for the two functors. Right: a vertex to denote their isomorphism.   }
    \label{fig4_2}
\end{figure}

Isomorphism $\varphi$ of functors can be represented by an invertible trivalent vertex in Figure~\ref{fig4_2} going two marks on a dashed line, representing composition of induction and restriction functors, to 
a single  mark, labeling the tensor product functor. The inverse isomorphism can be represented by a reflected diagram. 

\vspace{0.1in}

Although $\mchar(\kk)\not=2$ is sufficient, we assume that $\mchar(\kk)=0$. So representations of finite groups over $\kk$ are completely reducible. 
Given a subrepresentation $V\subseteq \V_H^G$, choose an idempotent endomorphism $e_V\in \End(\V_H^G)$ of projection on $V$. Although the diagrammatic calculus is not rich enough to have these idempotents built-in,
it can be described by a box labeled $e_V$ on the vertical line depicting the identity natural transformation of the functor $\V_H^G\otimes -$, see Figure~\ref{fig4_1}.  

\vspace{0.1in}

\begin{figure}
    \centering
\begin{tikzpicture}[scale=0.6]


\draw[thick,dashed] (-0.5,0) -- (4.5,0);
\draw[thick,dashed] (-0.5,4) -- (4.5,4);

\draw[thick] (2,0) -- (2,1);

\draw[thick] (2,3) -- (2,4);

\draw[thick] (1,1) rectangle (3,3);

\node at (4,1.5) {\Large $G$};

\node at (0,2.5) {\Large $G$};

\node at (2,-.85) {\Large $\V_H^G$};
\node at (2,4.85) {\Large $\V_H^G$};

\node at (2,2) {\Large $e_V$};
\end{tikzpicture}
    \caption{Idempotent $e_V$ on the endomorphism of $\V_H^G$.}
    \label{fig4_1}
\end{figure}
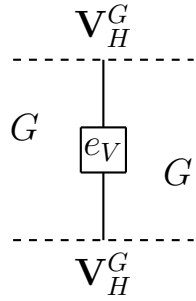

The quotient group $G_n/G_{n-1}^{(1)}$ is the symmetric group $S_2$, and its two-dimensional regular representation, viewed as a representation of $G_n$, will be denoted $V_1$. The latter representation is the induced from the trivial representation 
of $G_{n-1}^{(1)}$, 
\[ V_1 \cong  \Ind_{G_{n-1}^{(1)}}^{G_n} (\ukk).
\]
We depict the corresponding isomorphism of functors in Figure~\ref{fig4_3}. 

\vspace{0.1in}

\begin{figure}
    \centering
\begin{tikzpicture}[scale=0.6]

\begin{scope}[shift={(0,0)}]


\draw[thick] (-1,-.1) -- (0.5,-.1);
\draw[thick] (-1,.1) -- (0.5,.1);

\draw[thick] (3.5,-.1) -- (5,-.1);
\draw[thick] (3.5, .1) -- (5, .1);

\draw[thick] (0.5,-0.1) -- (0.5,0.1);
\draw[thick] (3.5,-0.1) -- (3.5,0.1);

\node at ( 4.5,0.65) {\Large $n$};
\node at (-0.5,0.65) {\Large $n$};

\draw[thick] (0.5, .1) .. controls (1,2) and (3,2) .. (3.5, .1);
\draw[thick] (0.5,-.1) .. controls (1,-2) and (3,-2) .. (3.5,-.1);

\node at (2,2) {\Large $n-1$};
\node at (2,-0.5) {\Large $n-1$};

\node at (-0.05,-1.35) {\Large $\Ind$};

\node at (4.05,-1.35) {\Large $\Res$};

\node at (7,0) {\Large $\simeq$};
\end{scope}

\begin{scope}[shift={(3,0)}]

\draw[thick] (6,0) -- (12,0);

\node at (06.5,-.5) {\Large $n$};
\node at (11.5,-.5) {\Large $n$};

\draw[thick,fill] (9.25,0) arc (0:360:0.25);
\node at (9,-1) {\Large $V_1$};
\end{scope}

\end{tikzpicture}
    \caption{Functor isomorphism.}
    \label{fig4_3}
\end{figure}
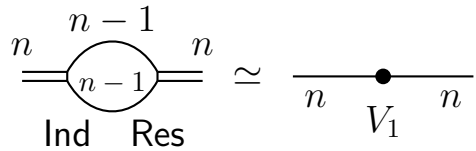

Recall the involution $\beta_n = (1,2^{n-1}+1)(2,2^{n-1}+2)\cdots (2^{n-1},2^n) \in G_n$ at the end of  Section~\ref{subsection:iterated-wreath-products}.
Under the quotient map, $\beta_n$ becomes the nontrivial element of $S_2$, which we may also denote $\underline{\beta}_n$. Multiplication by $\beta_n$ is an involutive 
endomorphism of $V_1$, see Figures~\ref{fig4_4} left and Figure~\ref{fig4_5}. 
Figure~\ref{fig4_4} right describes the foam that represents the corresponding endomorphism of $\Ind\circ\Res$, under its isomorphism with the tensor product functor. The foam consists of a flip between two $(n-1)$ facets, with the intersection interval shown in red. 

\vspace{0.1in}

\begin{figure}
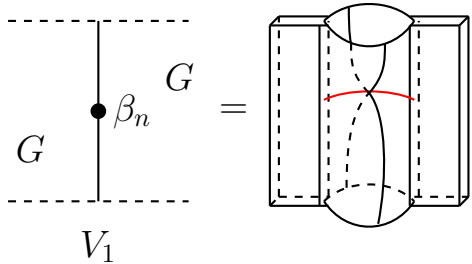

    \centering

    \caption{Foam representation of the endomorphism of the functor $V_1\otimes - \cong \Ind\circ \Res$ given by multiplication by $\beta_n$. Two lines on thin facets are used to better depict these facets.   }
    \label{fig4_4}
\end{figure}

\begin{figure}
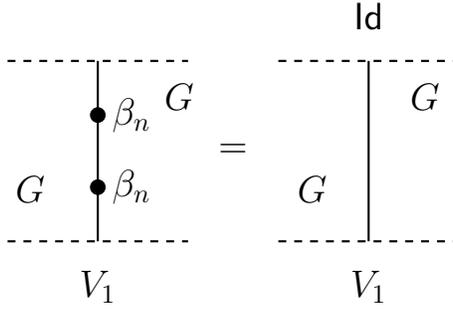

    \centering
%
    \caption{$\beta_n^2=1$, and endomorphism of $V_1$ it induces squares to identity.}
    \label{fig4_5}
\end{figure}

Relation $\beta_n^2=1$ translates into the foam identity in Figure~\ref{fig4_6} that can be obtained as a composition of Figure~\ref{fig4_7} and~\ref{fig2_5} relations.

\begin{figure}
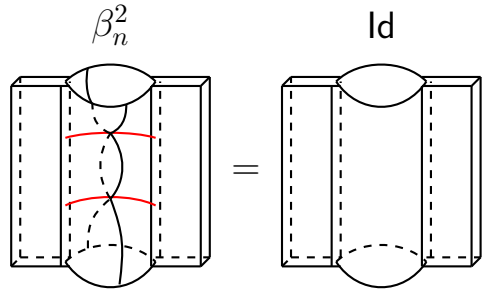

    \centering
%
    \caption{Equality of foams corresponding to the relation $\beta_n^2=1$ (as endomorphisms of $\Ind\circ \Res$ functor).}
    \label{fig4_6}
\end{figure}


\subsection{Foams for idempotents and basic relations on them}

Idempotents 
$e_+=\frac{1+{\beta_n}}{2}$ and $e_-=\frac{1-{\beta_n}}{2}$ in the group algebra $\kk G_n$ give corresponding idempotents, also denoted $e_+,e_-$, in the quotient algebra $\kk S_2\cong \End_{G_n}(V_1)$. These idempotents produce direct summands of representation $V_1$, the trivial and the sign representations,  that we denote $V_+$ and $V_-$, so that 
\begin{equation*}
    V_1 \cong V_+ \oplus V_-. 
\end{equation*} 
Note that $V_+\cong \ukk$, which is our two notations for the trivial representation. 

\vspace{0.1in} 

Under functor isomorphism $V_1\otimes - \cong \Ind\circ\Res$ these idempotents become idempotents in the endomorphism algebra of the latter functor,  also denoted $e_+$ and $e_-$. In the foam notation, we represent these idempotents in $\End(\Ind\circ \Res)$ by disks, green and blue, respectively, that intersect two opposite seam lines, with labels $+$ and $-$, respectively, see Figures~\ref{fig4_10} and~\ref{fig4_9}. 

\vspace{0.1in}

\begin{figure}
    \centering

    \caption{Idempotent $e_+=  \displaystyle{\frac{1+\beta_n}{2}}$.}
    \label{fig4_10}
\end{figure}

\begin{figure}
    \centering
%
    \caption{Idempotent  $e_-=\displaystyle{\frac{1-\beta_n}{2}}$. }
    \label{fig4_9}
\end{figure}

Some of the obvious relations 
\begin{equation*}
    1 = e_+ + e_-, \hspace{4mm} e_+e_-=e_-e_+ =0, \hspace{4mm}  e_+^2=e_+, \hspace{4mm} 
    e_-^2=e_-
\end{equation*}
are shown in Figures~\ref{fig4_9id}, and~\ref{fig4_16}.
Figure~\ref{fig4_26} shows how to convert from a planar to a foam representation of the identity endomorphism of $V_-$, also see Section~\ref{subsection:simplified-planar-notation}.

\vspace{0.1in}

\begin{figure}
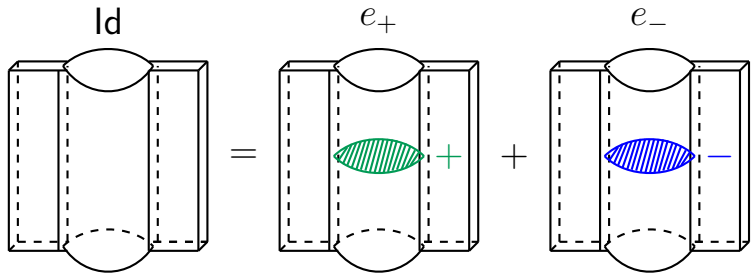

    \centering

    \caption{The sum of two idempotents $e_+$ and $e_-$ gives  the identity foam.}
    \label{fig4_9id}
\end{figure}

\begin{figure}
    \centering
%
    \caption{Top left: idempotency relation $e_-^2=e_-$ via foams. 
    Top right: idempotency relation $e_+^2=e_+$ via foams.  
    Bottom: orthogonality relation $e_+e_-=0$ via foams.}
    \label{fig4_16}
\end{figure}

\begin{figure}
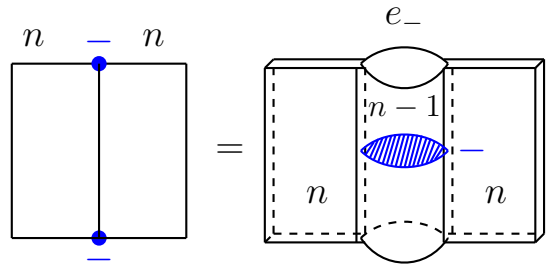

    \centering
%
    \caption{Converting from the planar to the foam presentation of the identity endomorphism of $V_-$.}
    \label{fig4_26}
\end{figure}

Figure~\ref{fig4_27} shows two more immediate foam relations or simplifications for these idempotent disks.
Figure~\ref{fig4_29} and~\ref{fig4_30} relations now follow.

\vspace{0.1in}

\begin{figure}
    \centering
%
    \caption{Left equality: symmetrizer $e_+$ is the projection onto the trivial representation; the foam interpretation is shown.  The second equality follows from the Figure~\ref{fig2_5} relation. Note  also the absence of homs between the trivial and the sign representations.}
    \label{fig4_27}
\end{figure}

\begin{figure}
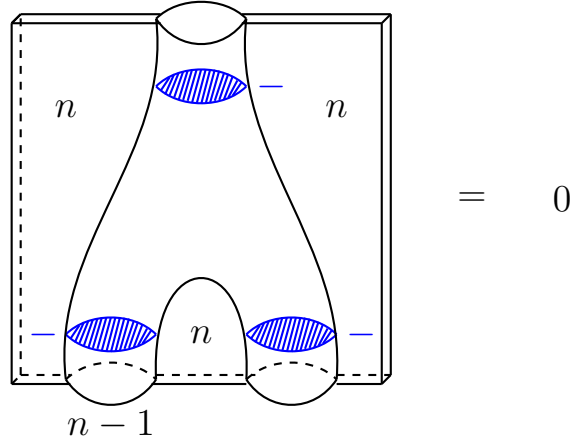

    \centering
%
    \caption{The only hom between irreducible representations  $V_-\otimes V_-\cong V_+$ and $V_-$ is $0$. This equality can also be checked by expanding three blue disks and canceling the terms.}
    \label{fig4_29}
\end{figure}

\begin{figure}
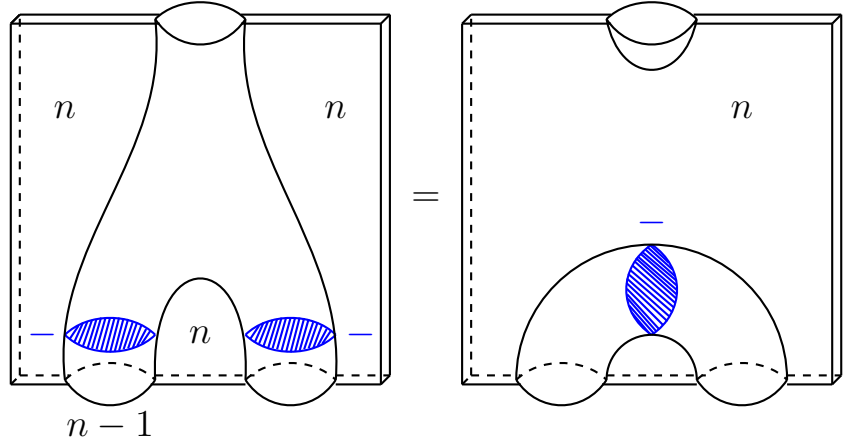

    \centering
%
    \caption{This relation follows by expanding the ``neck" on top left as in Figure~\ref{fig4_9id} and applying relations in Figures~\ref{fig4_27} and~\ref{fig4_29}.}
    \label{fig4_30}
\end{figure}

The last relation implies the relation in Figure~\ref{fig4_18}. 
Converting into the language of tensoring with representations, we can interpret it as saying that functor isomorphisms between tensoring with $V_-\otimes V_-$ and $V_+$ given by the two tubes at the top and bottom halves of Figure~\ref{fig4_18} left are mutually-inverse on one side. Consequently, they are mutually-inverse on the other side as well, as shown in Figure~\ref{fig4_17}, which can also be derived directly. Note that multiple blue disks along the tube in that figure can be reduced to a single one, via Figure~\ref{fig4_16} top row. Similarly, multiple green disks along a tube can be reduced to a single one, see Figure~\ref{fig4_16} top row.

\begin{figure}
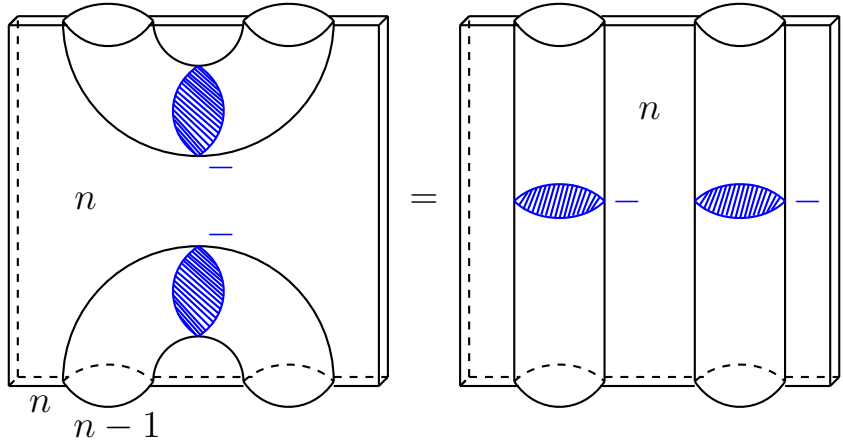

    \centering

    \caption{Foam equivalent of Figure~\ref{fig4_13} relation.}
    \label{fig4_18}
\end{figure}

\begin{figure}
    \centering
%
    \caption{ Horizontal circles indicate that the two $(n-1)$-facets on the left picture constitute a 2-torus inside the foam. Figure~\ref{fig4_16} relation allows to duplicate the $e_-$-disk, if desired. Compare with Figure~\ref{fig4_14} below.}
    \label{fig4_17}
\end{figure}

\subsection{Simplified (planar) notation}
\label{subsection:simplified-planar-notation}

Much simpler (and conventional) diagrammatics for representations of $S_2\simeq G_n/G_{n-1}^{(1)}$ are shown in Figures~\ref{fig4_38}, \ref{fig4_11}, \ref{fig4_12}, \ref{fig4_13}, and \ref{fig4_14}. Lines for the identity endomorphism of the trivial representation can be erased, see Figure~\ref{fig4_11}. 
We are essentially left with the sign representation $V_-$ and isomorphisms given by a cup and a cap between its tensor square and the trivial representation. 

\vspace{0.1in}

\begin{figure}
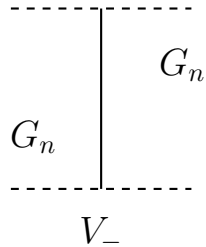

    \centering

    \caption{The sign representation $V_-$ corresponds to the idempotent $e_-=(1-\beta_n)/2$. The group $G_{n-1}^{(1)}$ acts trivially on $V_-$ and $\beta_n$ acts by $-1$. }
    \label{fig4_38}
\end{figure}

\begin{figure}
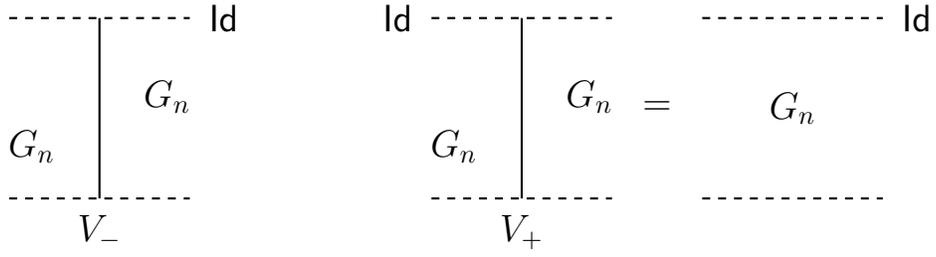

    \centering
%
    \caption{$V_-$ is the sign and $V_+$ is the trivial representation. Left: the identity endomorphism of the sign representation (and of the corresponding functor $V_-\otimes \bullet$). Right: the identity endomorphism of the trivial representation. Lines representing the identity map of the trivial representation can be erased, simultaneously with removing dots denoting $V_+$ on dashed lines.}
    \label{fig4_11}
\end{figure}

\begin{figure}
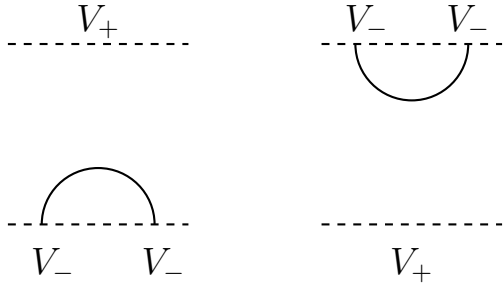

    \centering
%
    \caption{Mutually-inverse isomorphisms between $V_-\otimes V_-$ and the trivial representation $V_+$ since $V_-\otimes V_-\simeq V_+$.}
    \label{fig4_12}
\end{figure}

\begin{figure}
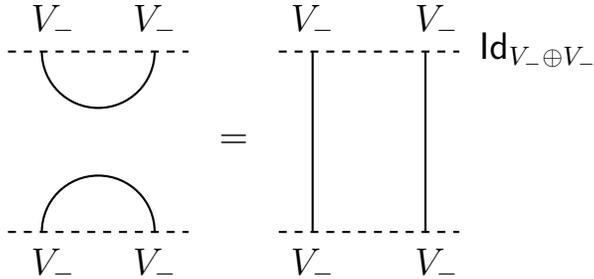

    \centering
%
    \caption{Composition of the two isomorphisms is the identity.}
    \label{fig4_13}
\end{figure}

\begin{figure}
    \centering
%
    \caption{Composition of isomorphisms is the identity.}
    \label{fig4_14}
\end{figure}

To convert between the two presentations, we need to replace $V_-$ lines in the second diagrammatics by tubes spanned by one or more blue ``minus" disks, see  
Figure~\ref{fig4_15} for the conversion of the ``cap" morphism. 

\vspace{0.1in}

\begin{figure}
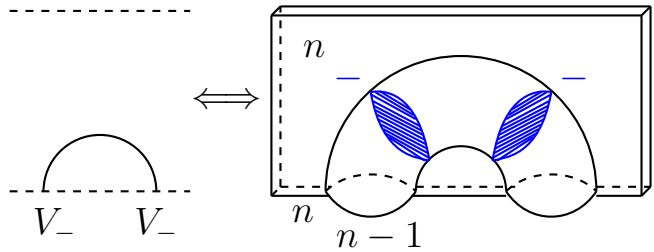

    \centering
%
    \caption{This correspondence represents the isomorphism $V_-^{\otimes 2}\cong V_+$ given by a foam. Two $e_-$ disks may be reduced to one, see~Figure~\ref{fig4_16}.}
    \label{fig4_15}
\end{figure}

The correspondence on the level of objects is further clarified in Figure~\ref{fig4_19}, with blue dot denoting the sign representation and the green dot the trivial representation (when it is convenient to keep track of the latter).

\begin{figure}
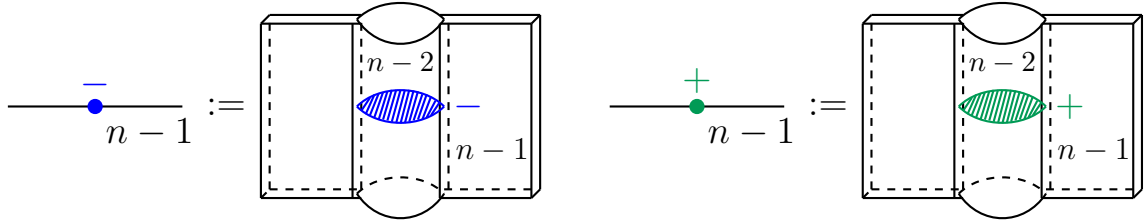

    \centering
%
    \caption{Compact notations for idempotents $e_-$ and $e_+$. For the equality on the right, the right hand side is isomorphic to the identity functor.}
    \label{fig4_19}
\end{figure}

\subsection{Dihedral groups}
\label{subsection:dihedral}

\begin{figure}
    \centering
%
    \caption{A functor isomorphism.}
    \label{fig4_20}
\end{figure}

Consider the two diagrams in Figure~\ref{fig4_20}. Each of them describes a summand of a composition  of restriction and induction functors. In the diagram on the left, we first restrict from $G_n$ to $G_{n-1}\times G_{n-1}$, then further restrict to $G_{n-2}\times G_{n-2}\times G_{n-1}$. After that we induce back to $G_n$. The ``minus" idempotent is applied for the composition of restriction and induction between $G_{n-1}$ and $G_{n-2}\times G_{n-2}$. In the diagram on the right, a similar functor  is described, but the inner induction and restriction is for the other factor of the product $G_{n-1}\times G_{n-1}$. 

\begin{proposition}\label{prop:rotate-minus}
The isomorphism in Figure~\ref{fig4_20} holds.
\end{proposition}

\proof
This is easily proven algebraically. The diagrammatic interpretation of mutually-inverse isomorphisms between these functors are given by the foam in Figure~\ref{fig4_21} and its reflection about the $xy$-plane. 
\endproof

Denote by $\mc{V}$ the functor given by the diagram on the left of Figure~\ref{fig4_20}. 

\vspace{0.1in}

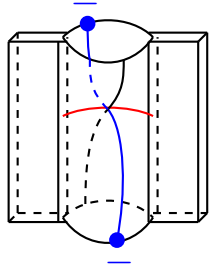
\begin{figure}
\centering 
\begin{tikzpicture}[scale=0.6]


\draw[thick,dashed] (17,.1) -- (17.2,.1);
\draw[thick,dashed] (13,.1) -- (15.1,.1);
\draw[thick] (16.9,-.1) -- (19,-.1);
\draw[thick] (12.8,-.1) -- (15.1,-.1);
\draw[thick,dashed] (12.8,-.1) -- (13,.1);

\draw[thick,dashed] (15,0) .. controls (15.5,.5) and (16.5,.5) .. (17,0);
\draw[thick] (15,0) .. controls (15.5,-.75) and (16.5,-.75) .. (17,0);

\draw[thick,dashed] (15.1,.1) -- (15.1,4);
\draw[thick,dashed] (17.1,.1) -- (17.1,4);
\draw[thick] (14.9,-.1) -- (14.9,3.9);
\draw[thick] (16.9,-.1) -- (16.9,3.9);
\draw[thick,dashed] (13,.1) -- (13,4.1);

\draw[thick] (12.8,-.1) -- (12.8,3.9);

\draw[thick] (19,-.1) -- (19,3.9);
\draw[thick] (19.2,.1) -- (19.2,4.1);
\draw[thick] (19,-.1) -- (19.2,.1);

\draw[thick] (15,4) .. controls (15.5,4.5) and (16.5,4.5) .. (17,4);
\draw[thick] (15,4) .. controls (15.5,3.25) and (16.5,3.25) .. (17,4);

\draw[thick] (12.8,3.9) -- (15.1,3.9);
\draw[thick] (16.9,3.9) -- (19,3.9);
\draw[thick] (13,4.1) -- (15.1,4.1); 
\draw[thick] (12.8,3.9) -- (13,4.1);
\draw[thick] (16.9,4.1) -- (19.2,4.1);
\draw[thick] (19,3.9) -- (19.2,4.1);

\draw[thick,fill,blue] (15.80, 4.3) arc (0:360:0.25);
\draw[thick,fill,blue] (16.48,-0.5) arc (0:360:0.25);

\draw[line width=1.60, red] (15,2.25) .. controls (15.5,2.5) and (16.5,2.5) .. (17,2.25);
\draw[line width=1.60,blue] (16.2,-.5) .. controls (16.4,.5) and (16.4,1.9) .. (16,2.4);

\draw[thick,dashed] (15.5,.3) .. controls (15.5,.8) and (15.5,1.9) .. (16,2.4);
\draw[thick] (16,2.4) .. controls (16.35,2.8) and (16.35,3.1) .. (16.35,3.5);
\draw[line width=1.60,dashed,blue] (16,2.4) .. controls (15.6,2.8) and (15.6,3.1) .. (15.6,3.5);
\draw[line width=1.60,blue] (15.6,3.5) .. controls (15.55,3.7) and (15.55,4.1) .. (15.55,4.3);

\node[blue] at (15.5,4.85) {\Large $-$};
\node[blue] at (16.25,-1) {\Large $-$};
\draw[thick,dashed] (17,.1) -- (19.2,.1);
\end{tikzpicture}
    \caption{Foam for  the functor isomorphism in Figure~\ref{fig4_20}. The blue line depicts the identity endomorphism of the ``blue point" functor (direct summand of the induction-restriction functor isomorphic to $V_-\otimes -$). The black line on the other thin facet is drawn to help see the facet. The two thin facets intersect along the red interval.}
    \label{fig4_21}
\end{figure}

The quotient group $G_n/G_{n-2}^{(2)}$ is naturally isomorphic to the dihedral group $D_4$ of symmetries of the square,
\begin{equation} \label{eq_dihed}
 G_n/G_{n-2}^{(2)} \cong D_4.
\end{equation}
The quotient map 
\begin{equation}
    G_n/G_{n-2}^{(2)} \lra G_n/G_{n-1}^{(1)} \cong S_2 
\end{equation}
corresponds to the homomorphism $D_4\lra S_2$ where to a symmetry of $D_4$ one associates the induced permutation of the two diagonals. Thinking of the quotient group 
$G_n/G_{n-2}^{(2)}$  as all symmetries of a full binary tree of depth 2, to get the homomorphism we map the tree to the square so that the two depth one branches correspond to the diagonals of the square, see Figure~\ref{fig5_21_a}. 

\vspace{0.1in}

\begin{figure}
    \centering
\begin{tikzpicture}[scale=0.6]


\draw[thick] (0,0) -- (1.5,-1);
\draw[thick] (0,0) -- (-1.5,-1);

\draw[thick] (1.5,-1) -- (2.5,-2);
\draw[thick] (1.5,-1) -- (0.5,-2);

\draw[thick] (-1.5,-1) -- (-0.5,-2);
\draw[thick] (-1.5,-1) -- (-2.5,-2);

\draw[thick,fill] (1.7,-1) arc (0:360:2mm);
\draw[thick,fill] (-1.3,-1) arc (0:360:2mm);

\draw[thick,fill] (2.7,-2) arc (0:360:2mm);
\draw[thick,fill] (0.7,-2) arc (0:360:2mm);
\draw[thick,fill] (-0.3,-2) arc (0:360:2mm);
\draw[thick,fill] (-2.3,-2) arc (0:360:2mm);

\node at (0.02,0) {\large $\blacksquare$};

\node at (-2.5,-2.85) {\Large $1$};
\node at (-0.6,-2.85) {\Large $2$};
\node at ( 0.6,-2.85) {\Large $3$};
\node at ( 2.5,-2.85) {\Large $4$};

\node at (5.25,-1) {\Large $\Rightarrow$};


\begin{scope}[shift={(3,0)}]


\draw[line width=0.65mm,Plum] (5,-3) -- (9,-3);
\draw[line width=0.65mm,Plum] (5,-3) -- (5,1);
\draw[line width=0.65mm,Plum] (5,1) -- (9,1);
\draw[line width=0.65mm,Plum] (9,1) -- (9,-3);

\draw[thick] (5,-3) -- (9,1);
\draw[thick] (5,1) -- (9,-3);

\draw[thick,fill] (5.2,1) arc (0:360:2mm);
\draw[thick,fill] (5.2,-3) arc (0:360:2mm);
\draw[thick,fill] (9.2,-3) arc (0:360:2mm);
\draw[thick,fill] (9.2,1) arc (0:360:2mm);

\node at (4.5,-3.5) {\Large $1$};
\node at (9.5, 1.5) {\Large $2$};
\node at (4.5, 1.5) {\Large $3$};
\node at (9.5, -3.5) {\Large $4$};

\draw[thick] (6.1,-1.9) -- (7,-2.4);
\draw[thick] (7.9,-1.9) -- (7,-2.4);

\node at (7.02,-2.4) {\large $\blacksquare$};

\draw[thick,fill] (6.3,-1.9) arc (0:360:2mm);
\draw[thick,fill] (8.1,-1.9) arc (0:360:2mm);

\end{scope}

\end{tikzpicture}
    \caption{An identification of symmetries of a depth 2 full binary tree with those of a square. Nodes $1,2,3,4$ of the tree are mapped to vertices of the square.}
    \label{fig5_21_a}
\end{figure}
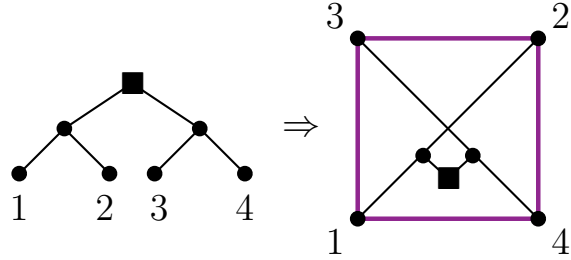


Denote by $V$ the unique (up to isomorphism) two-dimensional irreducible representation of $D_4$, given by its action by isometries on $\R^2$. 

\begin{prop}
 Under the above group isomorphism, the functor of tensor product with $V$ is isomorphic to the functor $\mc{V}$ given in Figure~\ref{fig4_20}. 
\end{prop}

\proof
  Functor $\mc{V}$ is a direct summand of the composition $\Ind\circ \Res$ for restricting from $G_n$ to the subgroup $G_{n-2}\times G_{n-2}\times G_{n-1}$ and inducing back. In $D_4$, the corresponding subgroup is $H-\{1,(34)\}$. 
  The complement to $\mc{V}$ in the above functor is $\Ind\circ \Res$ for the subgroup $G_{n-1}^{(1)}$ (since the complement is given by putting the ``plus" label on the dot in Figure~\ref{fig4_20} left and ``+" labels may be erased). Thus, the complementary functor is isomorphic to the direct sum of tensoring with the trivial $V_+$ and the sign representation $V_-$.
  
  \vspace{0.1in} 
  
  The composition functor of restriction then induction for the subgroup $H$ in $D_4$ is isomorphic to the functor of the tensor product with the four-dimensional representation $\kk[D_4/H]$. It is easy to decompose this representation into the direct sum 
  \begin{equation}
       \kk[D_4/H] \cong V \oplus V_+ \oplus V_-,
  \end{equation}
  using characters (in characteristic 0, see the table below) or directly (as long as $\mchar(\kk)\not=2$). 
\endproof

So far we have accounted for the fundamental representation $V$ of $D_4$ and two one-dimensional representations: the trivial $V_+$ and the sign representation $V_-$, on which the normal subgroup $\{1,(12),(34),(12)(34)\}$ acts trivially. 

\vspace{0.1in} 

Denote the remaining two one-dimensional representations of $D_4$ 
by $V_{-+}$ and $V_{--}$. On $V_{-+}$ generators $(12)$ and $(1324)$ act by $-1$, and on $V_{--}$ generators $(12)$ and $(13)(24)$ act by $-1$. Also see Figure~\ref{fig4_39}.

\vspace{0.1in} 

The table below lists the characters of the five irreducible 
representations of $D_4$ and of representation $\kk[D_4/H]$. 

\vspace{0.1in} 

\begin{center}
\begin{tabular}{ | c || c | c | c | c |c |} 
\hline 
 & $1$ &  $(12)$ & $(12)(34)$ & $(1324)$ &  $(13)(24)$ \\
 \hline \hline 
  $V\:\:\:\:$ & $2$  & $\:\:\: 0$  & $-2$ & $\:\:\: 0$ & $\:\:\: 0$ \\ 
  \hline 
  $V_+\:\:$ & $1$  & $\:\:\: 1$  &  $\:\:\:\: 1$ & $\:\:\: 1$ & $\:\:\: 1$ \\ 
  \hline 
  $V_-\:\:$ & $1$  & $\:\:\: 1$ & $\:\:\:\: 1$ & $-1$ & $-1$ \\ 
  \hline 
  ${V_{-+}}$ & $1$ & $-1$ & $\:\:\:\: 1$ & $-1$ & $\:\:\: 1$ \\ 
  \hline 
  ${V_{--}}$ & $1$  & $-1$ & $\:\:\:\: 1$  & $\:\:\: 1$ & $-1$ \\ 
  \hline 
  ${\kk[D_4/H]}$ & $4$ & $\:\:\: 2$  & $\:\:\:\: 0$ & $\:\:\: 0$ & $\:\:\: 0$\\
 \hline
\end{tabular}
\end{center}

\vspace{0.1in} 

Consider the endofunctor $\mc{V}'$ in the category of $G_n$-modules given by the diagram in Figure~\ref{fig4_21b} left. This functor is a direct summand of the composition of restriction to $G_{n-2}^{(2)}$ then induction back to $G_n$ functor. 
The latter composition is isomorphic to the tensor product with the 8-dimensional representation $\kk[G_n/G_{n-2}^{(2)}].$ The minus idempotents on both thin edges pick out a direct summand functor given by the tensor product with a two-dimensional representation. 

\vspace{0.1in} 

One way to understand functor $\mc{V}'$ is by computing the composition $\mc{V}\circ \mc{V}$, see Figure~\ref{fig4_23}. The computation uses the relations in Figure~\ref{fig4_24}. The square $\mc{V}^{2}$ decomposes as the sum of two functors, 
\begin{equation}
    \mc{V}^2 \ \cong \ (\Ind_{n-1}^n\circ \Res^{n-1}_n) \oplus \mc{V}' \cong (V_+\otimes \ast)\oplus (V_-\otimes \ast) \oplus \mc{V}',
\end{equation}
since $\Ind_{n-1}^n\circ \Res^{n-1}_n$ is isomorphic to the functor of tensoring with $V_+\oplus V_-$. 
At the same time, we have decomposition of tensor product of representations
\begin{equation}\label{eq_tens_prod}
    V\otimes V \ \cong \ V_+ \oplus V_- \oplus V_{-+}\oplus V_{--},
\end{equation}
(tensor square of the fundamental $D_4$ representation $V$ is the sum of the four irreducible one-dimensional $D_4$ representations). 

\vspace{0.1in} 

Thus, functor $\mc{V}'$ is isomorphic to the functor of tensoring with $V_{-+}\oplus V_{--}$, 
\begin{equation}
    \mc{V}' \ \cong \ (V_{-+}\oplus V_{--})\otimes \ast . 
\end{equation}

The foam that transposes the two thin edges of this diagram, together with the minus dots on them, see Figure~\ref{fig4_21b}, is an endomorphism of the diagram of order two. The two idempotents (symmetrizer and antisymmetrired) made off this endomorphism give functors isomorphic to functors of tensor product with $V_{-+}$ and $V_{--}$, respectively. 
 
\vspace{0.1in}

\begin{figure}
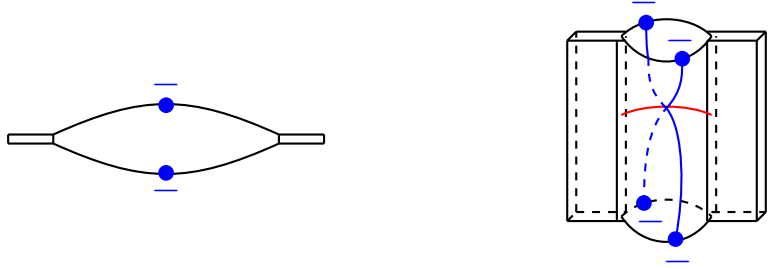

\centering 

    \caption{Left: An endofunctor $\mc{V}'$ of $G_n\mathrm{-mod}$. Right: an order two endomorphism $\beta'$ of $\mc{V}'$ (the flip).  Symmetrizing or antisymmetrizing via $\beta'$ decomposes $\mc{V}'$ into a direct sum of two functors.     }
    \label{fig4_21b}
\end{figure}

\begin{figure}
    \centering
%
    \caption{Direct sum decomposition of $\Res\circ\Ind$ into the identity and the  transposition functors, see Proposition~\ref{prop_canonical_decomp_functors} and Figure~\ref{fig2_1}. }
    \label{fig4_22}
\end{figure}

Figure~\ref{fig4_22} gives a direct sum decomposition of $\Res\circ\Ind$ as the identity and the transposition functors, and Figure~\ref{fig4_25} shows that functors $\mc{V}'$ and tensoring with $V_{-+}\oplus V_{--}$ are isomorphic.

\vspace{0.1in}

\begin{figure}
    \centering

    \caption{We apply relations in Figures~\ref{fig4_22} and~\ref{fig4_24} to decompose the square of the functor $\mc{V}$. Figure~\ref{fig4_22} relation is applied inside the dotted red rectangle. }
    \label{fig4_23}
\end{figure}

\begin{figure}
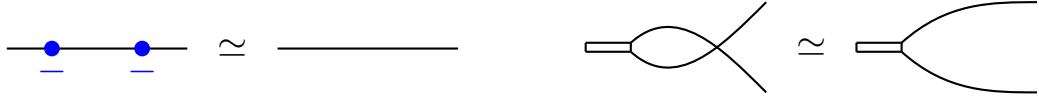

    \centering
%
    \caption{Left: tensor square of $V_-$ is the trivial representation, $V_-^{\otimes 2}\cong V_+$.   Right: functor isomorphism $\Ind \circ T_{12} = \Ind$, where $T_{12}$ is the transposition, given by Figure~\ref{fig2_4} foams reflected in the vertical plane. }
    \label{fig4_24}
\end{figure}

\begin{figure}
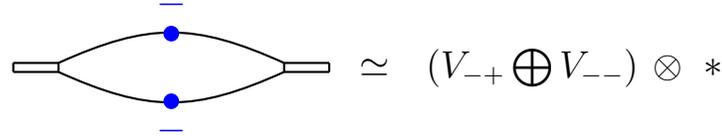

    \centering
    
%
    \caption{Functor $\mc{V}'$ is isomorphic to the functor of  
     tensoring with $V_{-+}\oplus V_{--}$. }
    
    \label{fig4_25}
\end{figure}


Each of the five irreducible $D_4$-representations: 
trivial rep $V_+$, sign rep $V_-$, fundamental two-dimensional representation $V$ and the two other one-dimensional representations $V_{-+}, V_{--}$ can be described by a suitable graph, together with an idempotent linear combination of foams assigned to it. The graphs together with the idempotents, in our notations, are shown in Figure~\ref{fig4_41}. 

\vspace{0.1in}

The computation in Figure~\ref{fig4_23}, together with direct  decompositions for the two terms at the bottom line of the figure, can be translated into the direct sum decomposition (\ref{eq_tens_prod}) for the tensor square $V^{\otimes 2}$. Decompositions of tensor products of other pairs of irreducible representations of $D_4$ can be derived in a similar way. For instance, an isomorphism $V\otimes V_-\cong V$ can be related to the identity in Figure~\ref{fig4_40} and a similar identity obtained by reversing the order of the two halves of the left picture and changing the right  hand side to the  identity natural transformation of the functor $\Res\circ\Ind\circ \Res$. This provides a foam interpretation and lifting of decompositions of tensor products of irreducible $D_4$-representations.

\begin{figure}
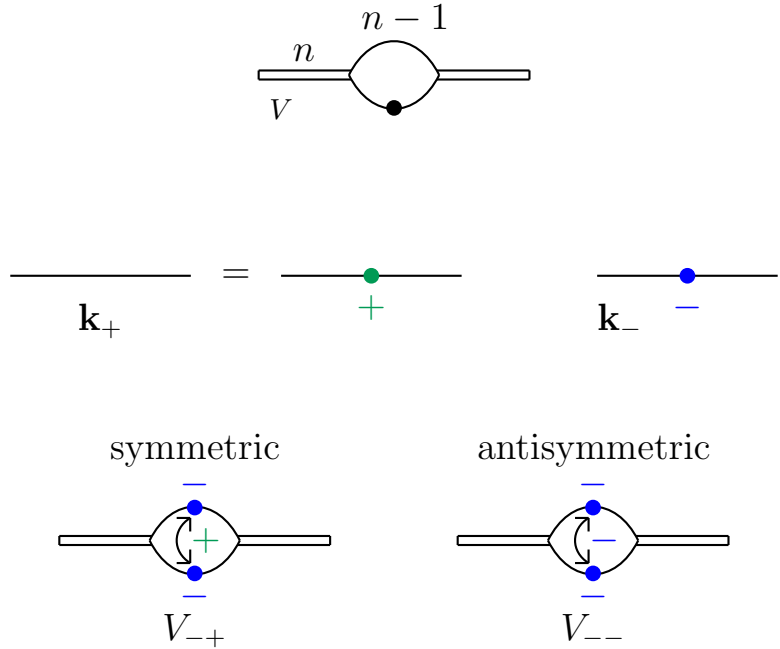

    \centering

    \caption{This figure lists foam idempotents for all $5$ irreducible representations of $D_4\simeq G_n/G_{n-2}^{(2)} $. Top to bottom and left to right, these correspond to representations $V$ (fundamental), $V_+$ (trivial), $V_-$ (sign), $V_{-+}$ and $V_{--}$, respectively. 
    }
    \label{fig4_41}
\end{figure}

\begin{figure}
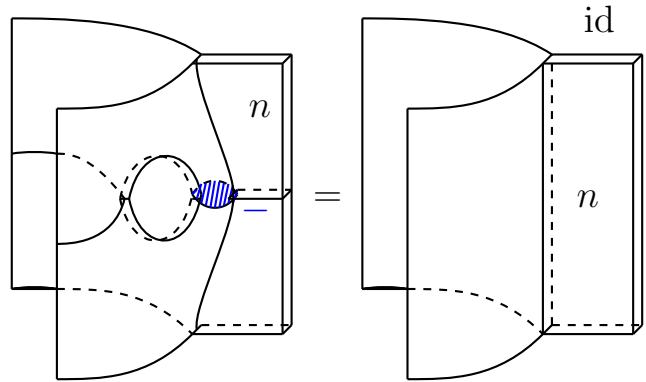

    \centering
%
    \caption{The two vertical halves of the image on the left are mutually inverse isomorphisms.  }
    \label{fig4_40}
\end{figure}


\subsection{Rooted trees and higher depth representations}
\label{subsection_rooted_trees}

Consider the chain of inclusions 
\begin{equation*}
G_n\supset G_{n-1}^{(1)}\supset G_{n-2}^{(2)}\supset\ldots \supset G_{0}^{(n)}=\{1\}.
\end{equation*}
In particular,  
$G_n/G_{n-1}^{(1)}\simeq S_2$, the symmetric group of order $2$, and 
$G_n/G_{n-2}^{(2)} \simeq D_4$, the dihedral group of order $4$. To describe irreducible representations of $D_4$ via foams we had to use foams that go between $n$-facets but in the middle may have facets of thickness $n-2$ (a dot labeled $-$ on an $(n-1)$-facet, as in Figure~\ref{fig4_41}, requires descending to $(n-2)$-facets to define it).  

\vspace{0.1in}

Let us say that a representation of $G_n$ has \emph{depth $k$} if the subgroup $G_{n-k}^{(k)}$ acts nontrivially  on it, while $G_{n-k-1}^{(k+1)}$ acts trivially. To create functors of tensoring with depth $k$ representations using foams, apply enough restriction functors to get from a line of thickness $n$ to a line of thickness $n-k-1$ in at least one location of the diagram, and then go back to a single line of thickness $n$. Denote the resulting graph  by $\Gamma$. There is a composition $F(\Gamma)$ of restriction and induction functors associated with $\Gamma$, and $F(\Gamma)$ is isomorphic to tensoring  with a suitable induced representation $W$ of $G_n$. One can then introduce some idempotent $e\in \End(F(\Gamma)) $ given by a linear combination of foams with boundary $\Gamma$ on bottom and top. Idempotent $e$ defines a direct summand of $F(\Gamma)$ as well as a direct summand of $W$ once an isomorphism between $F(\Gamma)$ and $W\otimes -$ is fixed.

\vspace{0.1in} 


Irreducible representations of the wreath product $G\wr S_n$ over an algebraically closed field, where $G$ is a finite group, was studied by Kerber~\cite[Chapter 2]{Ke} and the representations of the wreath product $G\wr H$ of two permutation groups $G$ and $H$ are discussed in Meldrum \cite{Me}. The irreducible representations of the (iterated) $n$-th wreath product $G_n$ over a field $\kk$ of characteristic different from $2$ were classified by Orellana, Orrison, and Rockmore~\cite[Proposition 3.1]{OOR}, as a special case of their classification of iterated wreath products of the cyclic group $C_m$, for $m=2$. Their classification gives a bijection between isomorphism classes of irreducible representations and isomorphism classes of complete binary trees of depth $n-1$ with vertices labeled by  signs $+,-$ and an additional assumptions that at each vertex $v$ labeled by the minus sign $-$ the standard symmetry $\beta_k$, see Section~\ref{subsection:iterated-wreath-products}, applied to the subtree at the vertex $v$,  preserves signs of vertices. 

\vspace{0.1in}

\begin{figure}
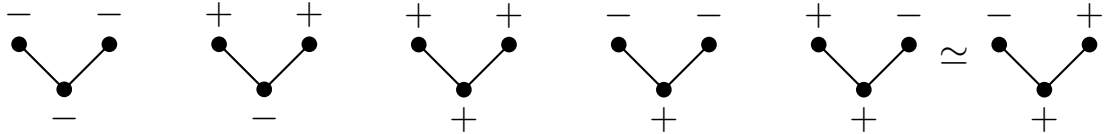

    \centering

    \caption{Above are the labeled trees corresponding to the irreducible representations of the dihedral group $D_4$, with the labeling as in~\cite{OOR}. Representations corresponding to these trees, going from left to right, are $V_{--},V_-,V_+,V_{-+},V$, respectively. The rightmost two labeled trees are isomorphic via the swap at the root of the tree.}
    \label{fig4_39}
\end{figure}

When depth $n=2$, the five labeled trees corresponding to irreducible representations of the dihedral group  $D_4\cong G_2$ are shown in Figure~\ref{fig4_39}. 

\vspace{0.1in} 

Let us write an irreducible representation of $D_4$ in the notation of~\cite{OOR} as 
$V({}^\beta {\alpha}^\gamma)$, where $\alpha, \beta, \gamma \in \{+,-\}$.
The two-dimensional irreducible representation is $V({}^+ +^-) $
which corresponds to the tree on the left in Figure~\ref{fig4_35}, 
By ~\cite{OOR}, notation $V({}^+ -^-)$ corresponding to the labeled tree on the right in Figure~\ref{fig4_35} does not correspond to any representation since there is a minus sign at the root; if there is a minus sign at the root, then the two subtrees must be identical, via a swap of the subtrees that preserves the order of lowest nodes (from left to right). If the sign at the root is $+$, then the two subtrees do not need to be identical. 
So the irreducible representations of the dihedral group $D_4$ are given in Figure~\ref{fig4_39}. 
Note that $V({}^+ +^-)\simeq V({}^- +^+)$ since the two subtrees are canonically isomorphic (via a swap at the root of the tree).


\begin{figure}
    \centering
\begin{tikzpicture}[scale=0.6]
\draw[thick,fill] (0.25,0) arc (0:360:0.25);  
\draw[thick] (0,0) -- (1,1); 
\draw[thick] (0,0) -- (-1,1);
\draw[thick,fill] (1.25,1) arc (0:360:0.25);
\draw[thick,fill] (-.75,1) arc (0:360:0.25);
\node at (0,-.75) {\Large $\alpha$};
\node at (-1.25,1.75) {\Large $\beta$};
\node at (1.25,1.75) {\Large $\gamma$};
\end{tikzpicture}
\qquad \qquad 
\begin{tikzpicture}[scale=0.6]
\draw[thick,fill] (0.25,0) arc (0:360:0.25);  
\draw[thick] (0,0) -- (1,1); 
\draw[thick] (0,0) -- (-1,1);
\draw[thick,fill] (1.25,1) arc (0:360:0.25);
\draw[thick,fill] (-.75,1) arc (0:360:0.25);
\node at (0,-.65) {\Large $-$};
\node at (-1.25,1.75) {\Large $+$};
\node at (1.25,1.65) {\Large $-$};
\end{tikzpicture}
    \caption{A labeled tree of height $1$, where $\alpha,\beta,\gamma\in\{+,-\}$. The labeled tree on the right corresponds to the notation $+^--$.}
    \label{fig4_35}
\end{figure}
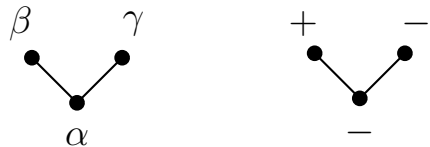







\begin{remark}
Consider the profinite limit 
\begin{equation}
    \widehat{G}=\lim_{n\to\infty} G_n. 
\end{equation}
The profinite limit has an open subgroup isomorphic to $(\widehat{G})^{\times 2^n}$,  with the quotient $G_n$, for each $n\ge 1$. 
One can then consider foams as above without any restrictions on the number of times a facet can be split into a pair of ``thinner" facets. Such foams will encode natural transformations between induction and restriction functors for suitable inclusions between direct products of groups $\widehat{G}$. 
\end{remark}


%
%

\section{Patched surfaces, separable extensions, and foams}  \label{galois_extensions}

\subsection{Defect circles and Frobenius algebra automorphisms}

\subsubsection
{Commutative Frobenius algebras and 2D TQFTs.}
A commutative Frobenius algebra $A$ over a field $\kk$ is a commutative $\kk$-algebra together with a nondegenerate linear functional (trace map) $\varepsilon: A\lra \kk$. Algebra $A$ is necessarily finite dimensional.
Such algebra gives rise to a two-dimensional TQFT, a tensor functor $\mcF$ from the category of oriented two-dimensional cobordisms to the category of $\kk$-vector spaces, see~\cite{Ab,Kc1,Kc2,LP}. This functor $\mcF$ associates $A^{\otimes k}$ to the 1-manifold which is the union of $k$ circles. To the generating morphisms \emph{cup, cap, pants, copants, and transposition}, it associates the unit, counit (trace), multiplication, comultiplication maps and transposition of factors in $A^{\otimes 2}$, respectively, see Figure~\ref{fig5_1}. 

\vspace{0.1in}

\begin{figure}[ht]
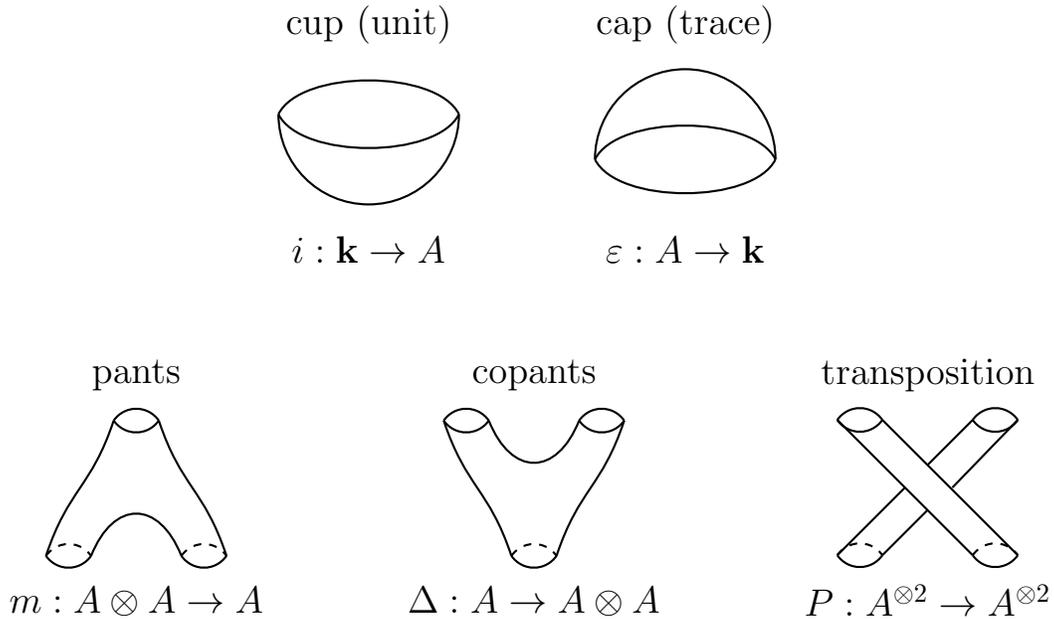

    \centering


    \caption{Generating cobordisms are taken by $\mcF$ to the structure maps of $A$: identity  $\iota: \kk \lra A$, trace $\varepsilon:A\lra \kk$, multiplication $m:A^{\otimes 2}\lra A$, comultiplication $\Delta:A\lra A^{\otimes 2}$ and the transposition of factors in the tensor product $P:A^{\otimes 2}\lra A^{\otimes 2}$.}
    \label{fig5_1}
\end{figure}

Given $A$, two-dimensional cobordisms can be refined by allowing elements of $A$, which are represented by dots, to float on surfaces. Functor $\mcF$ is extended to such cobordisms by associating  to  a tube with a dot labeled by $a\in A$ the multiplication map $m_a: A\lra A, m_a(b)=ab.$ Dots $a,b$ floating on a component may be merged into a single dot $ab$. Dots can also be called \emph{0-dimensional defects}. 

\vspace{0.1in} 

A closed surface of genus $g$ (possibly with elements of $A$ floating on it) evaluates to an element of $\kk$. One way to compute the evaluation and, more generally, simplify the topology of the cobordism (at the cost of working with linear combinations of cobordisms) is via the \emph{neck-cutting relation}. That is, pick a basis $x_1,\dots, x_n$ of $A$ and let $y_1,\dots, y_n$ be the dual basis, with $\varepsilon(x_iy_j)=\delta_{i,j}$ for $1\le i,j\le n$. Then a tube can be ``cut" to a sum of decorated cups and caps, see Figure~\ref{fig5_2}, right. 

\vspace{0.1in}

\begin{figure}
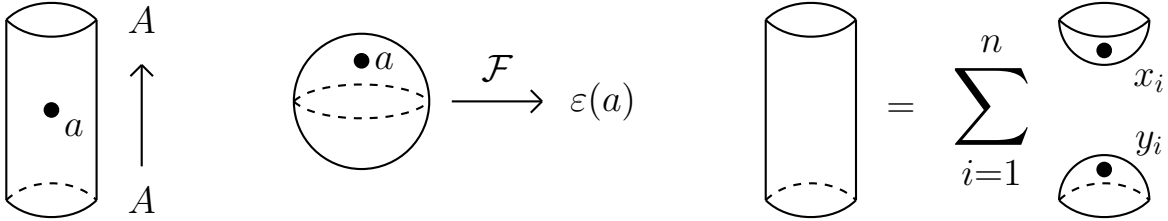

    \centering

    \caption{Left: a dot labeled $a$ on a tube goes to multiplication by $a$ map $m_a:A\lra A$. Middle: a 2-sphere dotted by $a$ evaluates to $\varepsilon(a)$. Right: the neck-cutting relation, where $\{x_i\}_i$ and $\{y_i\}_i$ are dual bases of $A$ relative to $\varepsilon$.}
    \label{fig5_2}
\end{figure}

The neck-cutting relation can be written algebraically: 
\begin{equation}
     \Id_A \ = \ \sum_{i=1}^n x_i  \varepsilon(y_i \ast), \ \ \ \ \mathrm{or} \ \ \ \  a = \sum_{i=1}^n x_i  \varepsilon(y_i a), 
     \quad  
     a\in A. 
\end{equation}
This decomposition of the identity map for a commutative Frobenius algebra can be found in \cite[Chapter 2, page 16]{KQ}. Its analogue for noncommutative Frobenius algebras has a similar form but requires cobordisms with inner boundary and corners, see~\cite[Section 3.1, Figure 3.1.8]{IK}. 


\subsubsection{Defect lines and Frobenius automorphisms}
An automorphism $\sigma$ of $A$ is called a \emph{Frobenius automorphism} or an \emph{$\varepsilon$-automorphism} if $\varepsilon\circ \sigma= \varepsilon$ as maps $A\lra\kk$. The second way of referring to $\sigma$ may be preferable to avoid possible confusion with the Frobenius  endomorphism of commutative rings in finite characteristic. The group of $\varepsilon$-automorphisms may be denoted $G(A)$ or $G(A,\varepsilon)$, to emphasize dependence on $\varepsilon$. 

\vspace{0.1in} 

Two-dimensional TQFT $\mcF$ may be further refined by adding one-dimensional defects to surfaces. These defects are co-oriented circles labeled by $\varepsilon$-automorphisms of $A$. An example is worked out in~\cite[Section 2.2]{KR2}. 

\vspace{0.1in}

Functor $\mcF$ is extended to such cobordisms with defects.  It takes a circle labeled $\sigma$ on a tube with upward coorientation to the map $\sigma: A\lra A$, see Figure~\ref{fig5_3}. Coorientation of a circle may be reversed simultaneously with replacing $\sigma$ by $\sigma^{-1}$. Given that the underlying surface is oriented, co-orientation of a circle on it induces an orientation on the circle and vice versa, so it is also possible to describe this setup via oriented rather than co-oriented circles. 

\vspace{0.1in}

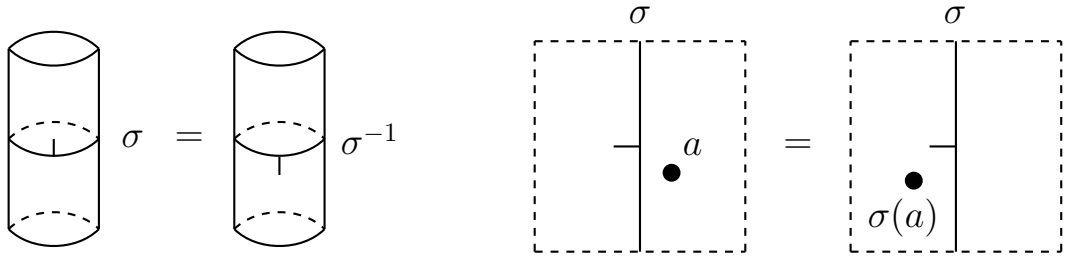
\begin{figure}
    \centering
\begin{tikzpicture}[scale=0.7]
\draw[thick,dashed] (0,0) .. controls (0.5,.5) and (1.5,.5) .. (2,0);
\draw[thick] (0,0) .. controls (0.5,-.5) and (1.5,-.5) .. (2,0);

\draw[thick,dashed] (0,2) .. controls (0.5,2.5) and (1.5,2.5) .. (2,2);
\draw[thick] (0,2) .. controls (0.5,1.5) and (1.5,1.5) .. (2,2);

\draw[thick] (0,4) .. controls (0.5,4.5) and (1.5,4.5) .. (2,4);
\draw[thick] (0,4) .. controls (0.5,3.5) and (1.5,3.5) .. (2,4);

\draw[thick] (0,0) -- (0,4);
\draw[thick] (2,0) -- (2,4);

\draw[thick] (1,1.6) -- (1,2.0);

\node at (2.6,2) {\Large $\sigma$};

\node at (4,2) {\Large $=$};

\draw[thick,dashed] (5,0) .. controls (5.5,.5) and (6.5,.5) .. (7,0);
\draw[thick] (5,0) .. controls (5.5,-.5) and (6.5,-.5) .. (7,0);

\draw[thick,dashed] (5,2) .. controls (5.5,2.5) and (6.5,2.5) .. (7,2);
\draw[thick] (5,2) .. controls (5.5,1.5) and (6.5,1.5) .. (7,2);

\draw[thick] (5,4) .. controls (5.5,4.5) and (6.5,4.5) .. (7,4);
\draw[thick] (5,4) .. controls (5.5,3.5) and (6.5,3.5) .. (7,4);

\draw[thick] (5,0) -- (5,4);
\draw[thick] (7,0) -- (7,4);

\draw[thick] (6,1.6) -- (6,1.2);
\node at (8,2) {\Large $\sigma^{-1}$}; 

\begin{scope}[shift={(11,0)}]

\draw[thick,dashed] (0,0) -- (4,0); 
\draw[thick,dashed] (0,0) -- (0,4); 
\draw[thick,dashed] (4,0) -- (4,4); 
\draw[thick,dashed] (0,4) -- (4,4); 
\draw[thick] (2,0) -- (2,4);
\node at (2,4.5) {\Large $\sigma$};

\draw[thick,fill] (3.00,1.5) arc (0:360:0.25);
\node at (3.25,2) {\Large $a$};

\draw[thick] (1.5,2) -- (2,2);
\node at (5,2) {\Large $=$};

\draw[thick,dashed] (6,0) -- (10,0); 
\draw[thick,dashed] (6,0) -- (6,4); 
\draw[thick,dashed] (10,0) -- (10,4); 
\draw[thick,dashed] (6,4) -- (10,4); 
\draw[thick] (8,0) -- (8,4);
\node at (8,4.5) {\Large $\sigma$};

\draw[thick,fill] (7.25,1.5) arc (0:360:0.25);

\node at (7,0.65) {\Large $\sigma(a)$};
\draw[thick] (7.5,2) -- (8,2);
\end{scope}
\end{tikzpicture}
    \caption{Dot crossing a defect circle. If crossing in the opposite direction, $b$ will become $\sigma^{-1}(b)$.}
    \label{fig5_3}
\end{figure}

A dot labeled $a$ may cross over a defect line $\sigma$ simultaneously with changing its label to $\sigma(a)$, see Figure~\ref{fig5_3}. An innermost circle around a dot $a$ reduces to the dot $\sigma^{\pm 1}(a)$ depending on its coorientation, see Figure~\ref{fig5_4}. An innermost circle not containing any dots can be removed, since $\sigma(1)=1$, see Figure~\ref{fig5_4}.

\vspace{0.1in}

\begin{figure}
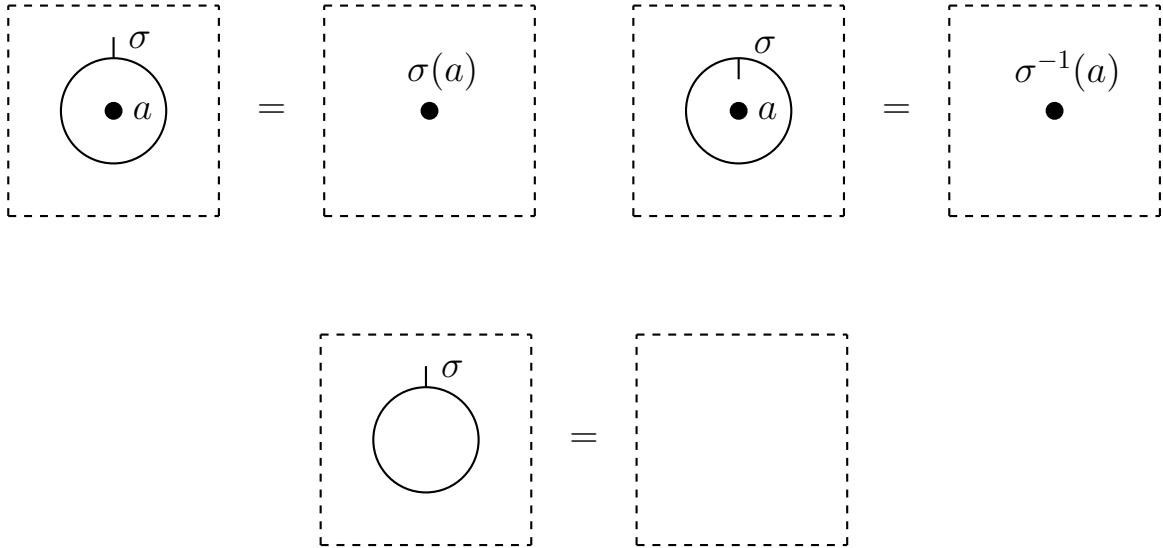

    \centering
 

    \caption{Defect circle $\sigma$ around dot $a$ without coorientation evaluates to dot $\sigma(a)$.   }
    \label{fig5_4}
\end{figure}

Two parallel defect lines labeled $\sigma_1,\sigma_2$, co-oriented in the same direction can be converted to a single line labeled $\sigma_1\sigma_2$, see Figure~\ref{fig5_5}.  

\vspace{0.1in}

\begin{figure}
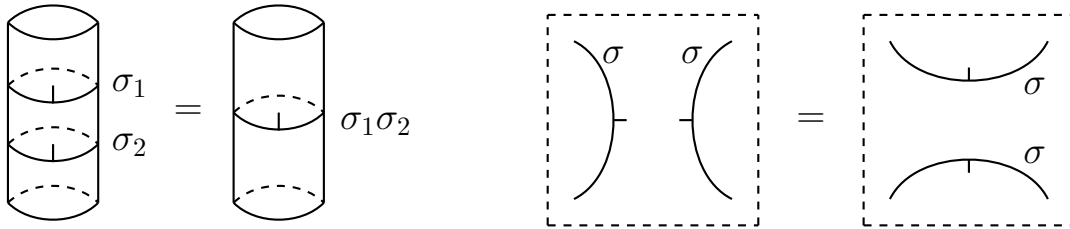

    \centering
%
    \caption{Left: merging two parallel circles into one. Right: Merging and splitting circles with the same label.}
    \label{fig5_5}
\end{figure}

Suppose we are given an oriented closed surface $S$ with $A$-labeled dots and $G(A,\varepsilon)$-labeled defect circles. Evaluation $\mcF(S)\in \kk$ is multiplicative under disjoint union of surfaces so we can assume $S$ is connected. To evaluate $S$, we do two surgeries (neck-cutting) on each side of each defect line in $S$ to reduce $S$ to a linear combination of products of dotted spheres with a single defect line and dotted surfaces, see Figure~\ref{fig5_6} and Figure~\ref{fig5_7}. 
Each connected component of genus $g>0$ can be further simplified via neck-cutting into a linear combination of dotted 2-spheres, see Figure~\ref{fig5_7} right. 

\vspace{0.1in}

\begin{figure}
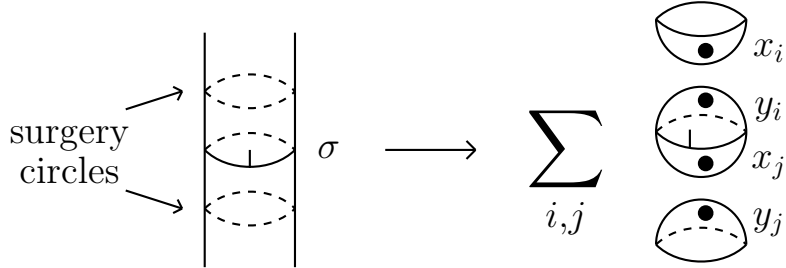

    \centering

    \caption{Separating defect lines into different connected components via neck-cutting. }
    \label{fig5_6}
\end{figure}

\begin{figure}
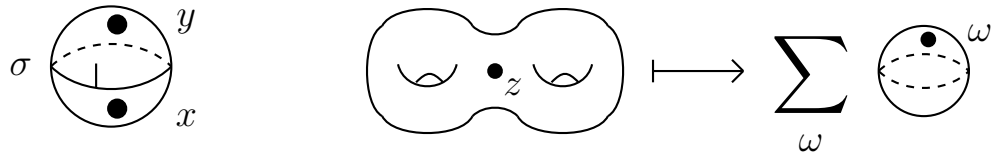

    \centering
%
    \caption{Left: a 2-sphere with a single defect circle and two dots. Right: reducing higher genus components via neck-cutting. 
    }
    \label{fig5_7}
\end{figure}

To evaluate a sphere with a $\sigma$-defect and dots $x,y$  on it, as in Figure~\ref{fig5_8} center, we can push one of the dots across $\sigma$-circle into the region with the other dot, remove the circle (since it now circles an empty region), multiply the dots and apply the trace, see Figure~\ref{fig5_8}. Since $\sigma$ respects $\varepsilon$, the two ways  of doing it result in the same answer. 
 
 \vspace{0.1in}

\begin{figure}
    \centering
\begin{tikzpicture}[scale=0.8]

\begin{scope}[shift={(-1,0)}]
\draw[thick] (6.25,3) arc (0:360:1.25);
\draw[thick,dashed] (3.75,3) .. controls (4.5,3.5) and (5.5,3.5) .. (6.25,3);
\draw[thick] (3.75,3) .. controls (4.5,2.5) and (5.5,2.5) .. (6.25,3);

\draw[thick, fill] (4.75,2.25) arc (0:360:0.25);
\draw[thick, fill] (5.75,2.25) arc (0:360:0.25);

\node at (2.9,2) {\Large $\sigma^{-1}(y)$};

\node at (6.35,2) {\Large $x$};

\draw[thick,|->] (5,1.25) -- (5,-0.75);

\node at (5.5,0.25) {\Large $\mathcal{F}$};
\node at (5,-1.5) {\Large $\varepsilon(\sigma^{-1}(y)x)$};
\end{scope}

\draw[thick,<-|] (6.5,3) -- (8.5,3);

\begin{scope}[shift={(0,0)}]


\draw[thick] (12.25,3) arc (0:360:1.25);

\draw[thick,dashed] (9.75,3) .. controls (10.5,3.5) and (11.5,3.5) .. (12.25,3);
\draw[thick] (9.75,3) .. controls (10.5,2.5) and (11.5,2.5) .. (12.25,3);
\draw[thick, fill] (11.25,3.85) arc (0:360:0.25);
\node at (12.35,3.95) {\Large $y$};
\draw[thick, fill] (11.25,2.2) arc (0:360:0.25);
\node at (12.35,2.05) {\Large $x$};
\draw[thick] (10.75,2.65) -- (10.75,3.05);
\node at (9.4,3.10) {\Large $\sigma$};
\end{scope}

\draw[thick,|->] (13.5,3) -- (15.5,3);

\begin{scope}[shift={(1,0)}]


\draw[thick] (18.25,3) arc (0:360:1.25);

\draw[thick,dashed] (15.75,3) .. controls (16.5,3.5) and (17.5,3.5) .. (18.25,3);
\draw[thick] (15.75,3) .. controls (16.5,2.5) and (17.5,2.5) .. (18.25,3);

\draw[thick, fill] (16.75,3.70) arc (0:360:0.25);
\node at (15.75,4) {\Large $y$};

\draw[thick, fill] (17.75,3.70) arc (0:360:0.25);
\node at (18.78,4) {\Large $\sigma(x)$};

\draw[thick,|->] (17,1.25) -- (17,-0.75);
 
\node at (17.5,0.25) {\Large $\mathcal{F}$};
\node at (17,-1.5) {\Large $\varepsilon(y\sigma(x))$};
\end{scope}

\end{tikzpicture}
    \caption{Two ways to evaluate a sphere with a $\sigma$-defect circle give the same answer since $\sigma$ is an $\varepsilon$-automorphism, $\varepsilon(\sigma^{-1}(y)x)=\varepsilon(y\sigma(x))$. }
    \label{fig5_8}
\end{figure}
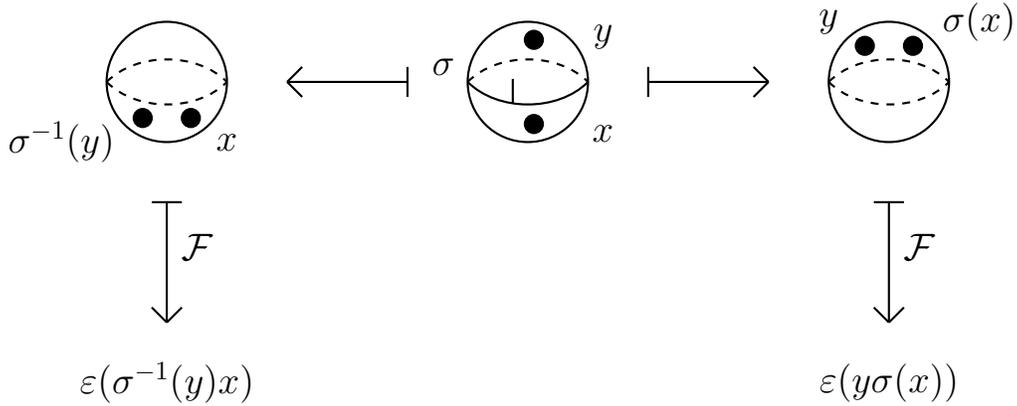

We record this as a proposition and denote the resulting evaluation of a decorated surface $S$ as $\brak{F}$ of $\mcF(S)$.  

\begin{prop}
  A closed oriented surface $S$ with floating $A$-dots and co-oriented disjoint $\sigma$-circles for $\sigma\in G(A,\varepsilon)$ has a well-defined evaluation $\mcF(S)$. 
\end{prop}

\proof
The evaluation of $S$ is outlined above, via 
surgeries on both sides of each $\sigma$-circle, and then evaluating surfaces decorated by dots and 2-spheres decorated by a $\sigma$-circle and $x,y$, as in Figure~\ref{fig5_7}. The only invariance to check, modulo commutative Frobenius algebra axioms, is that for pushing a dot labeled $y$ across a $\sigma$-circle, which is done in Figure~\ref{fig5_8}. 
\endproof
 
 \begin{example}
 \label{ex:dotless-2-torus-non-contractible-sigma-defect-circle}
 A dotless 2-torus $T$ with a non-contractible $\sigma$-defect circle evaluates to the trace of $\sigma$ on $A$, see Figure~\ref{fig5_9}.
For example, $\tr(\sigma)=\lambda+2+\lambda^{-1}$ for the automorphism $\sigma$ on $A$ given by \eqref{eq_A_sigma} and  \eqref{eq_var_eps1} below since $1,a,b,ab$ have eigenvalues $1, \lambda,\lambda^{-1}, 1$, respectively. 
 \end{example}

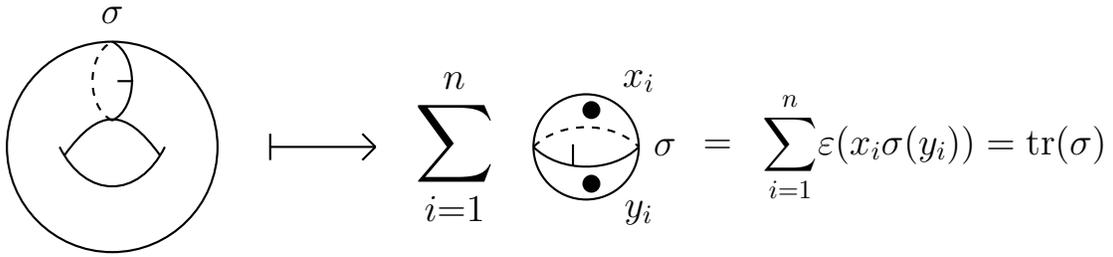
\begin{figure}
    \centering
\begin{tikzpicture}[scale=0.7]
\draw[thick] (4,0) arc (0:360:2); 
\draw[thick] (1,0) .. controls (1.5,-1) and (2.5,-1) .. (3,0);
\draw[thick] (1.1,-.15) .. controls (1.75,.75) and (2.25,.75) .. (2.9,-.15);

\node at (2,2.5) {\Large $\sigma$};
\draw[thick] (2,.5) .. controls (2.5,.75) and (2.5,1.75) .. (2,2);
\draw[thick,dashed] (2,.5) .. controls (1.5,.75) and (1.5,1.75) .. (2,2);
\draw[thick] (2.1,1.25) -- (2.4,1.25);

\draw[thick,|->] (5,0) -- (7,0);

\node at (8.5,0) {\huge $\displaystyle{\sum_{i=1}^n}$};


\draw[thick] (12.25,0) arc (0:360:1.25);
\draw[thick,dashed] (9.75,0) .. controls (10.5,0.5) and (11.5,0.5) .. (12.25,0);
\draw[thick] (9.75,0) .. controls (10.5,-.5) and (11.5,-.5) .. (12.25,0);

\draw[thick, fill] (11.25, 0.80) arc (0:360:0.25);
\node at (12.2,1.25) {\Large $x_i$};
\draw[thick, fill] (11.25,-0.80) arc (0:360:0.25);
\node at (12.2,-1.25) {\Large $y_i$};

\draw[thick] (10.75,-.35) -- (10.75,0.05);

\node at (12.75,0) {\Large $\sigma$};

\node at (13.75,0) {\Large $=$};

\node at (18,0) {\Large
$\ \displaystyle{\sum_{i=1}^{n}}\ \varepsilon(x_i\sigma(y_i))=\tr(\sigma)$};

\end{tikzpicture}
    \caption{Torus with an essential $\sigma$-defect circle evaluates to $\tr(\sigma)$. Note that $\tr(\sigma)=\tr(\sigma^{-1})$ in view of Corollary~\ref{cor_nondeg}.   }
    \label{fig5_9}
\end{figure}

More generally, given a decorated cobordism $C$ between one-manifolds, neck-cutting and consequent evaluation reduces its image under $\mcF$ to a linear combination of dotted cup and cap cobordisms, see Figure~\ref{fig5_10}, where we assume that one-manifolds are not decorated. Such a decorated cobordism between unions of circles induces a linear map $A^{\otimes k_0}\lra A^{\otimes k_1}$, where $k_0,k_1$ is the number of bottom and top boundary circles of $C$. In this way, the original two-dimensional TQFT associated to $(A,\varepsilon)$ allows an extension with these decorations. This TQFT associated $A^{\otimes k}$ to a union of $k$ undecorated circles.

\begin{figure}
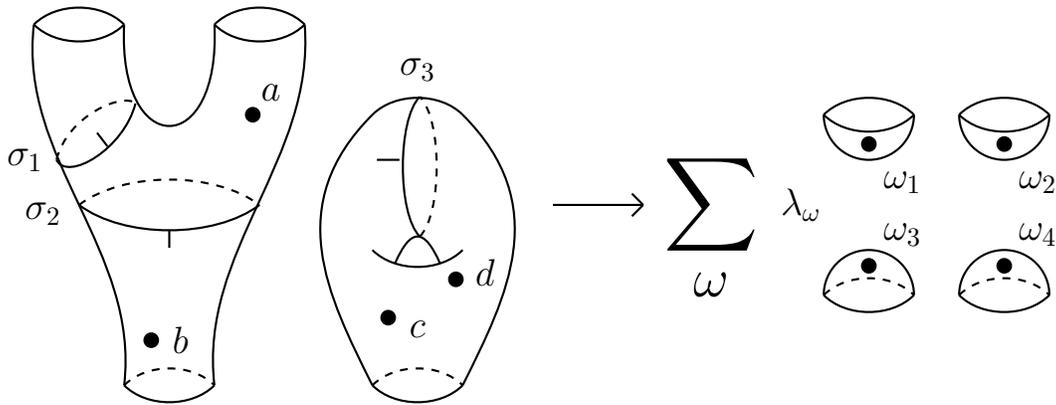

\begin{center}

\end{center}
    \caption{Reducing a decorated cobordism to a linear 
    combination of cups and caps, by neck-cutting near each boundary circle and evaluating closed components. }
    \label{fig5_10}
\end{figure}

\subsubsection{State spaces of decorated 1-manifolds}\label{subsubsec_state}

One can  use the language of universal constructions, see~\cite{Kh4,BHMV} and references there, to extend the evaluation $\mcF(S)$ for closed decorated surfaces $S$ to state spaces of one-manifolds that inherit decorations from surfaces. Namely, a generic codimension one submanifold of $S$ may intersect $\sigma$-circles in finitely-many points. Local intersection information at such point consists  of a co-orientation and choice of $\sigma$. 

\vspace{0.1in} 

Vice versa, to a union $L$ of circles with co-oriented $\sigma$-dots, one can assign the state space $\mcF(L)$ as follows. Start with a $\kk$-vector space $\Fr(L)$ with a basis of oriented decorated surfaces $S$ with $\partial(S)\cong L$, one for each equivalence class of rel boundary homeomorphisms, see Figure~\ref{fig5_11}. These surfaces contain co-oriented $\sigma$-intervals, $\sigma$-circles and $A$-dots. Denote by $[S]$ the basis element for the surface $S$. 

\vspace{0.1in}

\begin{figure}[ht]
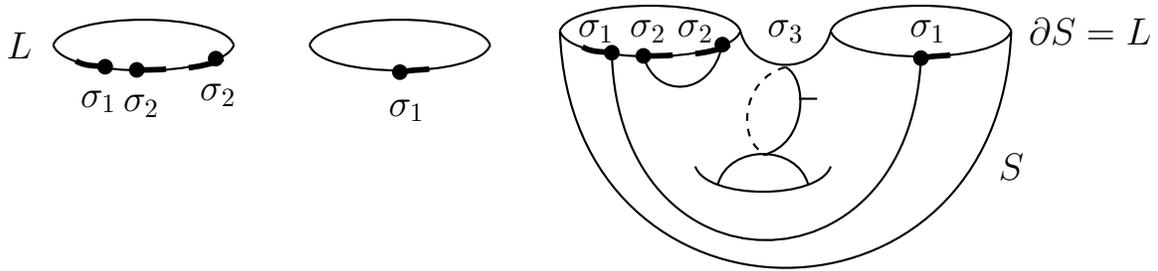

\begin{center}

\end{center}
    \caption{Decorated 1-manifold $L$ and decorated surface $S$ with $\partial(S)=L$.}
    \label{fig5_11}
\end{figure}

Two decorated surfaces $S_1,S_2$ with $\partial S_1\cong \partial S_2 \cong L$ can be glued together along the common boundary resulting in a closed decorated surface denoted $\overline{S}_2 S_1$, see Figure~\ref{fig5_12}.

\vspace{0.1in}

\begin{figure}
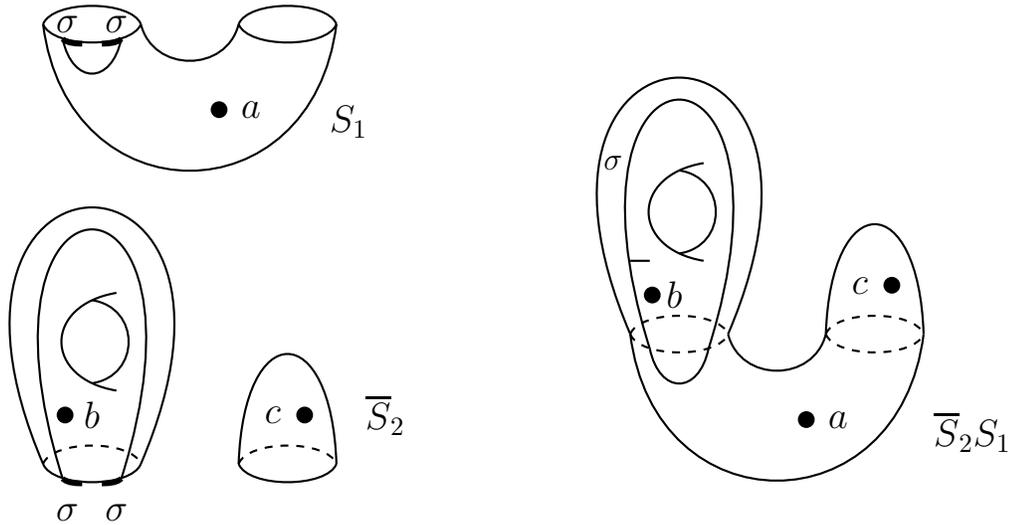

    \centering
%
    \caption{Gluing surfaces $S_1$ and $S_2$ along the common boundary.  }
    \label{fig5_12}
\end{figure}

Define a bilinear form $(\:\:,\:\:)$ on $\Fr(L)$ by 
\begin{equation}
    ([S_1],[S_2]) \ =  \ \mcF(\overline{S}_2 S_1).
\end{equation}
This bilinear form is symmetric. Define the state space of $L$ as the quotient of $\Fr(L)$ by the kernel of this bilinear form: 
\begin{equation}
    \mcF(L) \ := \ \Fr(L)/\mathrm{ker}((\:\:,\:\:)). 
\end{equation}
This is an example of universal construction of topological theories~\cite{Kh4,BHMV}, and this construction strategy has been applied in many different situations. Any decorated oriented two-dimensional cobordism $S$ induces a map 
\begin{equation}
    \mcF(S)\ : \ \mcF(\partial_0 S) \lra \mcF(\partial_1 S)
\end{equation}
given by composing a cobordism representing an element in $\mcF(\partial_0 S)$ with $S$. 

\vspace{0.1in}

In general, the state spaces $\mcF(L)$ are not multiplicative under disjoint union, and there are only inclusions
\begin{equation}
    \mcF(L_1)\otimes \mcF(L_2) \subset \mcF(L_1\sqcup L_2).
\end{equation}
Due to the neck-cutting formula, which can be applied only if the cutting circle is disjoint from $\sigma$-circles,  there is a restrictive case of multiplicativity, 
\begin{equation}
    \mcF(L\sqcup\SS^1) \cong \mcF(L)\otimes \mcF(\SS^1),
\end{equation}
where $\SS^1$ denotes an undecorated circle. State spaces $\mcF(L)$ are trivial for many decorated one-manifolds $L$, for instance if the endpoint labels $\sigma$ cannot be matched in pairs, keeping track of co-orientations. Here, it is convenient to at least allow co-orientation reversal together with changing $\sigma$ to $\sigma^{-1}$. Such a reversal (or flip) may happen anywhere along a defect circle or line. Along a defect circle, the total number of reversals must be even, so that locally  along a circle there is a well-defined co-orientation, with flips along reversal points, see Figure~\ref{fig5_13}.

\vspace{0.1in}

\begin{figure}
    \centering
\begin{tikzpicture}[scale=0.70]


\draw[thick,dashed] (0,0) -- (4,0);
\draw[thick,dashed] (0,5) -- (4,5);
\draw[thick] (2,0) -- (2,5);

\draw[thick,fill] (2.25,2.5) arc (0:360:0.25);

\draw[thick] (2,4.375) -- (2.5,4.375);
\draw[thick] (2,3.75) -- (2.5,3.75);
\draw[thick] (2,3.125) -- (2.5,3.125);

\draw[thick] (2,1.875) -- (1.5,1.875);
\draw[thick] (2,1.25) -- (1.5,1.25);
\draw[thick] (2,0.625) -- (1.5,0.625);

\node at (2.40,-0.5) {\Large $\sigma^{-1}$};

\node at (2,5.5) {\Large $\sigma$};

\begin{scope}[shift={(6.5,4.5)}]
\draw[thick] (0,0) .. controls (0.25,.6) and (1.75,.6) .. (2,0);
\draw[thick] (0,0) .. controls (0.25,-.5) and (1.75,-.5) .. (2,0);

\draw[line width=1.00mm] (1,-.4) .. controls (1.1,-.4) and (1.4,-.35) .. (1.7,-.25);

\draw[line width=1.00mm] (5,-.4) .. controls (5.1,-.4) and (5.4,-.35) .. (5.7,-.25);

\draw[thick] (1.9,-3.1) -- (2.3,-2.8);
\draw[thick] (4.2,-3) -- (4.55,-3.3);

\draw[thick,fill] (3.25,-3.65) arc (0:360:0.25);

\draw[thick] (4,0) .. controls (4.25,.6) and (5.75,.6) .. (6,0);
\draw[thick] (4,0) .. controls (4.25,-.5) and (5.75,-.5) .. (6,0);

\draw[thick] (0,0) .. controls (0.25,-6) and (5.75,-6) .. (6,0);
\draw[thick] (1,-.4) .. controls (1.25,-4.75) and (4.75,-4.75) .. (5,-.4);
\draw[thick] (2,0) .. controls (2.25,-3.5) and (3.75,-3.5) .. (4,0);

\node at (0.65,.05) {\large $\sigma$};
\node at (4.92,.05) {\large $\sigma^{-1}$};

\draw[thick,fill] (1.25,-.35) arc (0:360:0.25);
\draw[thick,fill] (5.25,-.35) arc (0:360:0.25);

\end{scope}

\begin{scope}[shift={(18,2.5)}]

\draw[thick,dashed] (1,0) arc (0:360:1);
\draw[thick] (2,0) arc (0:360:2);
\draw[thick,dashed] (3,0) arc (0:360:3);

\draw[thick,fill] (2.25,0) arc (0:360:0.25);
\draw[thick,fill] (-1.75,0) arc (0:360:0.25);

\draw[thick] (0,2) -- (0,2.5);
\draw[thick] (0,-1.5) -- (0,-2);

\node at (1.65,1.75) {\Large $\sigma$};
\node at (.85,-1.20) {\Large $\sigma^{-1}$};

\end{scope}
\end{tikzpicture}
    \caption{Left and center: co-orientation and $\sigma \leftrightarrow \sigma^{-1}$ flip along a seam. Right: There is an even number of flips along a $\sigma$-circle, even if $\sigma=\sigma^{-1}$.}
    \label{fig5_13}
\end{figure}
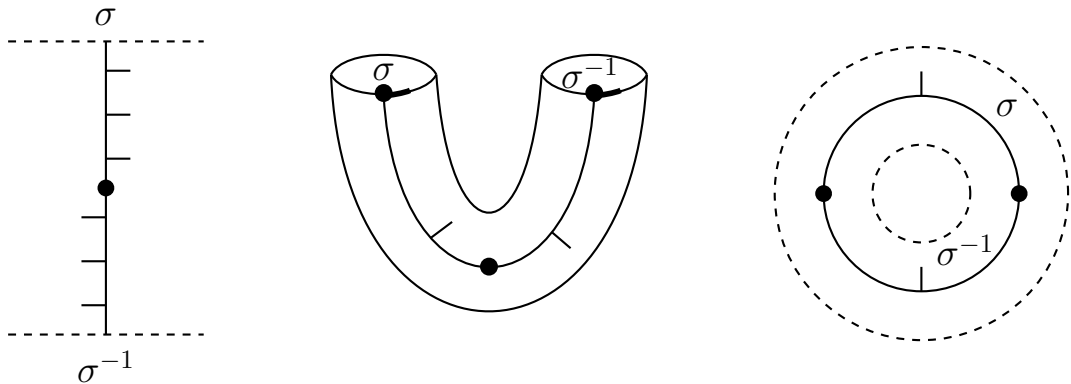

In general, we do not know much about the state spaces $\mcF(L)$ for collections of decorated circles as above. Furthermore, it would be natural to look for extensions of these theories to networks, where lines labeled $\sigma$ and $\tau$ can merge into a line labeled $\sigma\tau$, as we  now explain. 

\subsubsection{Turaev's homotopy TQFTs and universal theories.}
\label{subsubsec_turaev}
One can think of a $\sigma$-circle on $S$ as describing a sort of monodromy. Choose a topological space $X$ with a base point $x_0$ such that $\pi_1(X,x_0)\cong G=G(A,\varepsilon)$ and $\pi_2(X,x_0)=0$. To a $\sigma$-surface $S$, associate a homotopy class of maps $S\lra X$ as follows. $A$-dots floating on $S$ are ignored. Points away from neighborhoods of $\sigma$-circles are mapped to the basepoint $x_0$. An interval transverse to a $\sigma$-circle is mapped to the loop at $x_0$ representing element $\sigma\in\pi_1(X,x_0)$, using the co-orientation to choose between a map representing $\sigma$ or $\sigma^{-1}$, see Figure~\ref{fig5_14}.

\vspace{0.1in}

\begin{figure}[ht]
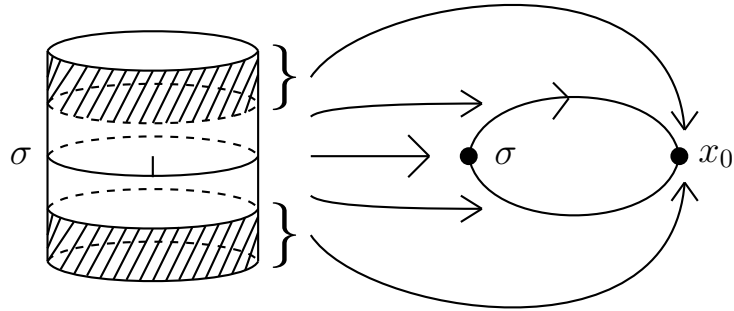

    \centering

    \caption{Map into $X$ near a seam circle of $S$.  }
    \label{fig5_14}
\end{figure}

In this way, our construction is reminiscent of Turaev's homotopy TQFTs in dimension two~\cite{Tu1,Tu2,MS}, Landau--Ginzburg (LG) orbifolds \cite{IV,BH,BR,LS,KW} and orbifolded Frobenius algebras 
\cite{Ka}. In the setting of Landau--Ginzburg models, monodromy transformation refers to a deformation of LG orbifolds via BPS spectrum (stable particles) or other (geometric) invariants as one moves around the moduli construction in order to understand their mirror symmetry.

\vspace{0.1in}

Furthermore, consider a circle $L$ with 3 defect points, co-oriented in the same direction, such that their labels multiply to $1\in G$, see Figure~\ref{fig5_15}.
In the $\sigma$-circles setup, this decorated circle cannot bound a decorated surface, so its space is zero. 

\vspace{0.1in}

\begin{figure}
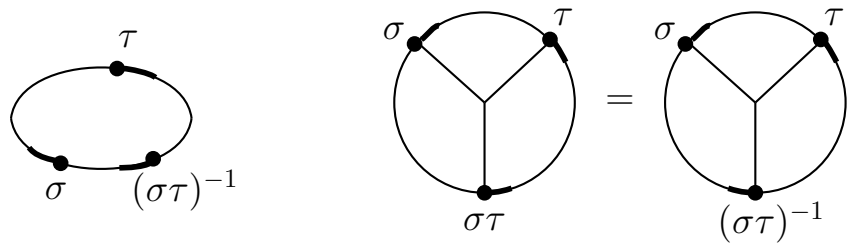

    \centering

    \caption{``Monodromy" along the circle on the left is trivial and motivates the introduction of a trivalent vertex.}
    \label{fig5_15}
\end{figure}

However, the ``monodromies" along the circle multiply to the trivial element of $G$, and it is natural to introduce a trivalent vertex, as shown in Figure~\ref{fig5_15}. 

\vspace{0.1in} 

Seam lines and trivalent vertices can be arranged into ``networks" on a surface $S$. It is natural to require that moves shown in Figure~\ref{fig5_16a} and Figure~\ref{fig5_16b} should preserve the evaluation of the network. 

\vspace{0.1in}

\begin{figure}
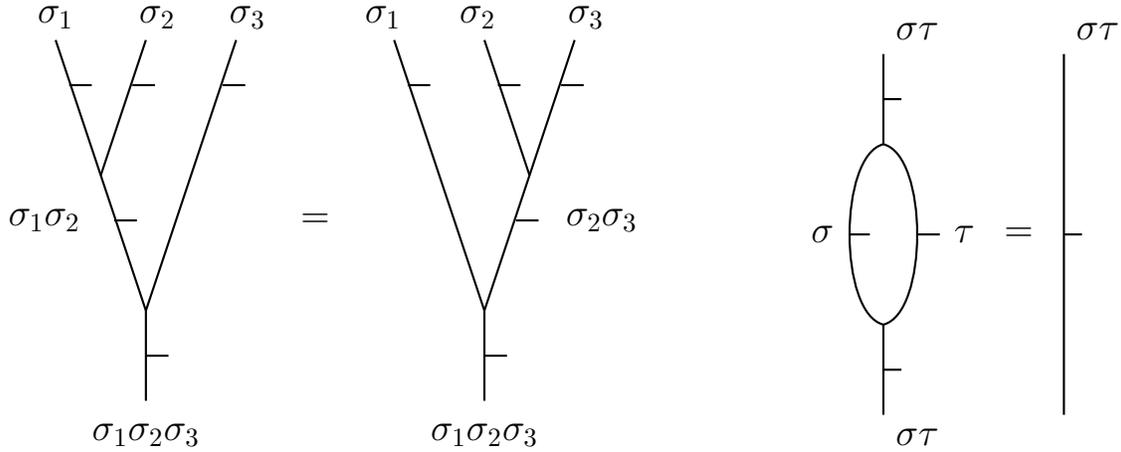

    \centering

    \caption{Associativity and digon simplification moves of $G$-labeled networks.}
    \label{fig5_16a}
\end{figure}

\begin{figure}
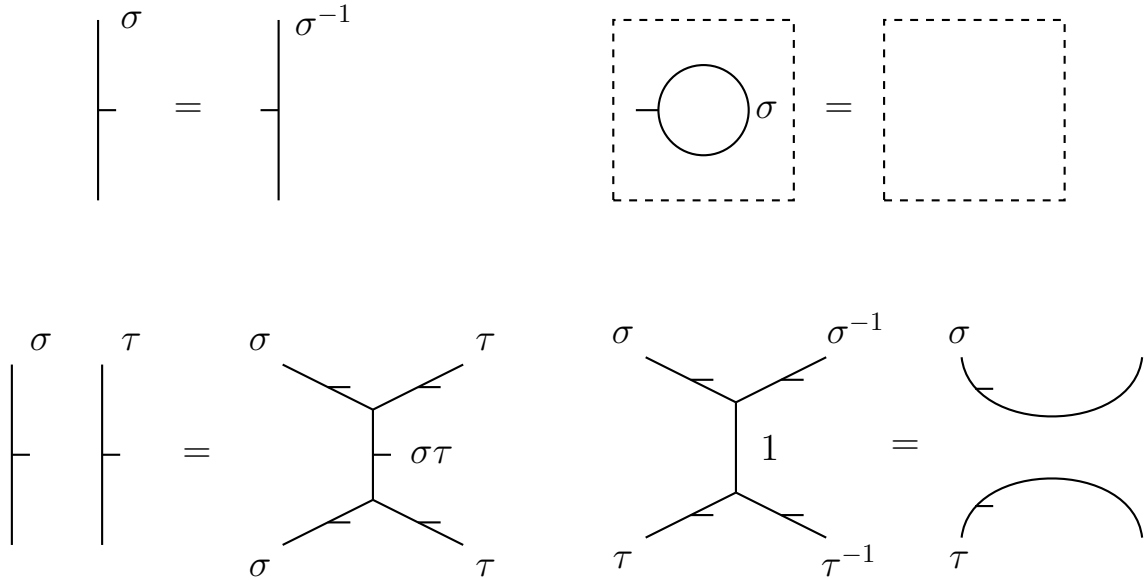

    \centering
%

    \caption{Some skein relations for $G$-networks. Either co-orientation is fine for top right relation. Bottom right relation says that intervals labeled $1$ may be erased. This relation can be clarified by  removing the interval on the left labeled $1$ and keeping the flip points on the two arcs that reverse co-orientation and send $\sigma$ to $\sigma^{-1}$, as in  Figure~\ref{fig5_13}. Same refinement can be applied to the top left relation. For careful treatment of flip points one should also choose types of allowed triples of coorientations allowed at networks' vertices and add suitable relations, see Figures~\ref{fig5_16c}, \ref{fig5_16cd} and~\ref{fig5_16d}.  }
    \label{fig5_16b}
\end{figure}

\begin{figure}[ht]
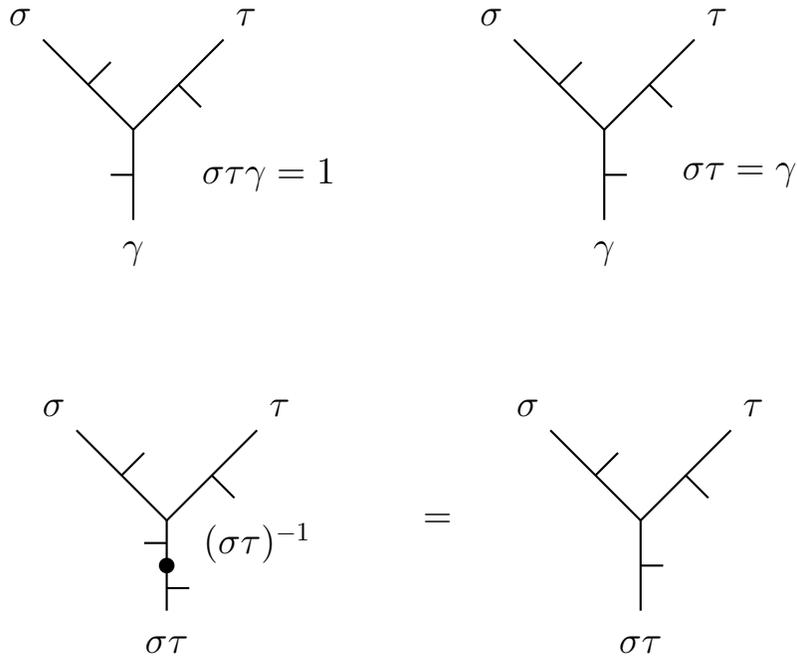

    \centering

    \caption{Two types of co-orientation triples around a vertex of the network are shown on the top. If both types are allowed, it is natural to add the relation that a vertex can absorb a co-orientation flip point, see bottom equality.  }
    \label{fig5_16c}
\end{figure}

\begin{figure}
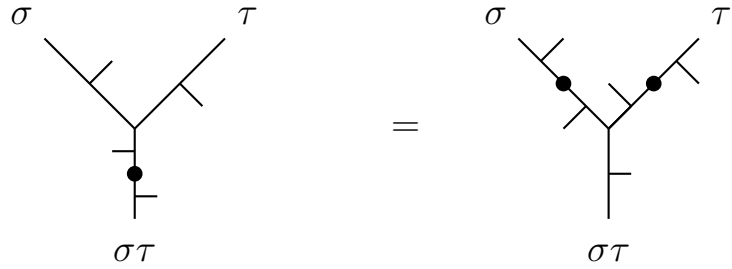

    \centering
%
    \caption{Moving a flip point through a vertex.}
    \label{fig5_16cd}
\end{figure}

\begin{figure}
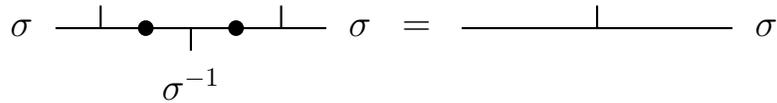

    \centering
%
    \caption{Canceling a pair of adjacent flip points on a seam.   }
    \label{fig5_16d}
\end{figure}

One motivation for these moves is that a representation of $\pi_1(S)$ into $G$, up to conjugation in $\pi_1(S)$, gives rise to an equivalence class of networks. Denote the set of 
equivalence classes by 
\begin{equation}
    \pi(S,G) \ := \Hom(\pi_1(S),G)/\pi_1(S), \quad   
    s(\rho)(t) = \rho(s^{-1}t s), \  \ s,t\in \pi_1(S), \ \rho \in \Hom(\pi_1(S),G). 
\end{equation}
To construct a network representing 
$\rho: \pi_1(S)\lra G$, viewed as an element of $\pi(S,G)$, decompose  a connected surface $S$ in the usual way as given by gluing a $4g$-gon along the sides. The sides represent generators $a_i,b_i$ of $\pi_1(S)$. Draw an interval crossing each the side, place the label $\rho(a_i), \rho(b_i)\in G$ on it, and suitably co-orient  the interval as well. Inside the $4g$-gon these $4g$ intervals naturally extend to a connected network, uniquely defined, since the relation 
\begin{equation}
    \prod_{i=1}^n \rho(a_i)\rho(b_i)\rho(a_i)^{-1}\rho(b_i)^{-1}=1 
\end{equation}
holds in $G$. The network has $4g-2$ trivalent vertices and its complement in $S$ is an open disk (if some edges of the network are labeled $1\in G$, they can then be erased, making the complement not simply connected). An example for $g=1$ is shown in Figure~\ref{fig5_17}, when necessarily $\rho(a),\rho(b)$ commute. Since the fundamental group of the two-torus is abelian, there is only one element in the conjugacy class of a homomorphism. Note that we are not conjugating by elements of $G$, only  by elements of $\pi_1(S)$ (equivalently, by inner automorphisms of the latter group). 

\vspace{0.1in}

\begin{figure}
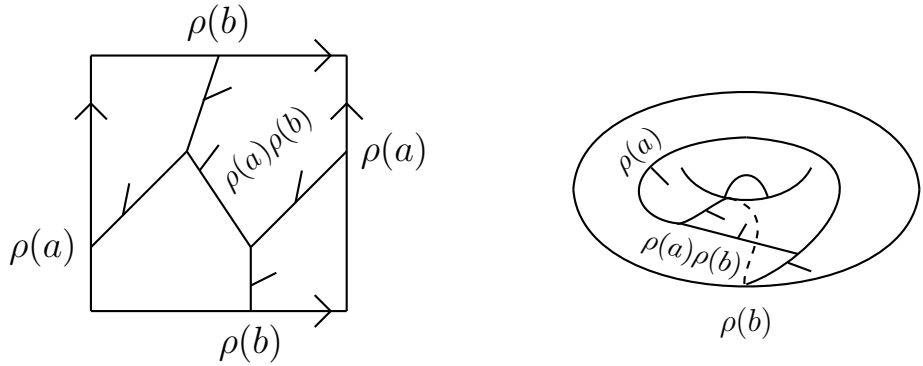

    \centering

    \caption{A network on the torus describing a homomorphism $\pi_1(T^2)\lra G$. Commutativity $\rho(a)\rho(b)=\rho(b)\rho(a)$ is needed for both trivalent vertices to make sense.   }
    \label{fig5_17}
\end{figure}

Vice versa, to a $G$-network $w$ on $S$ we can assign an element of $\pi(S,G)$. Choose a $K(G,1)$ space  $X$ which is a CW-complex with a single vertex $v_0$ and 1-cells $c(g)$  in a bijection with elements of $G$ such that that the loop along $c(g)$ represents $g$ in $\pi_1(X,v_0)\cong G$. 

\vspace{0.1in} 

View network $w$ as a trivalent graph on $S$, possibly with loops, and form a standard open neighborhood $U$ of $w$. Construct a map $\phi_w: S\lra X$ as follows. All points in $S\setminus U$ map to the base points $v_0$ of $X$. Neighborhood $U$ can be partitioned into a union of intervals, each one intersecting $w$ at a single point, and triangles, one for each vertex of $w$, see Figure~\ref{fig5_15_1}. 

\vspace{0.1in}

\begin{figure}
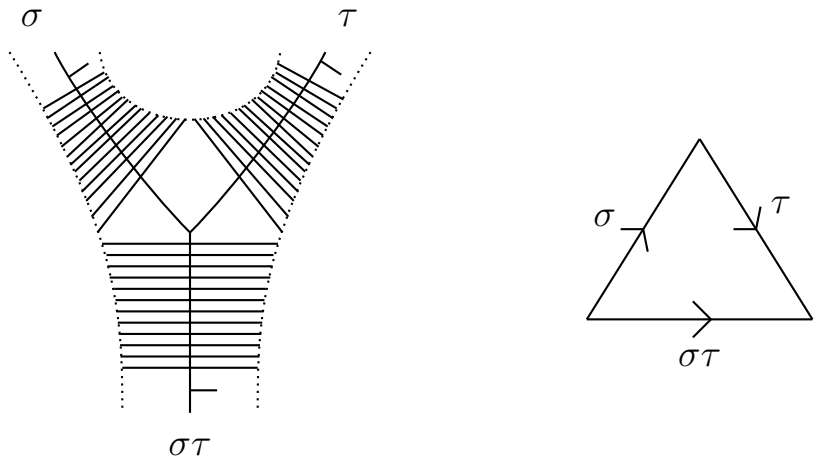

    \centering

    \caption{Left: Decomposing $U$ near a vertex of $w$ into a triangle and unions of parallel intervals, one for each leg of the tripod at the vertex. Right: the triangle is mapped to $X$ in a unique way, up to homotopy, extending the map of its boundary. }
    \label{fig5_15_1}
\end{figure}


Each interval intersecting $w$ at a point of a line labeled $\sigma$ is mapped bijectively to the 1-cell $c(\sigma)$ in the direction of co-orientation. Around each vertex of $w$ there is a triangle, with its sides mapped to $c(\sigma), c(\tau), c(\sigma\tau)$, respectively. There is a unique, up to homotopy, way to map this triangle to $X$ given the map on its sides.

\vspace{0.1in} 

Thus, to a network $w$  we assign a map $\phi_w: S\lra X$. Fixing a base point $s_0$ on $S$ away from $w$ induces a map $\pi_1(S,s_0)\lra \pi_1(X,v_0)$. Network transformations shown in Figure~\ref{fig5_16a} and~\ref{fig5_16b} away from the basepoint correspond to basepoint-preserving homotopies of maps $S\lra X$ and induce the same homomorphism of fundamental groups. Moving a base-point across a line labeled $\sigma$ conjugates the homomorphism by $\sigma.$

\begin{prop}\label{prop_hom_classes}
 The above correspondence gives a bijection between elements of $\pi(S,G)$ and isotopy classes of $G$-networks modulo relations in Figures~\ref{fig5_16a} and~\ref{fig5_16b}. 
\end{prop}

\proof Let us sketch a proof of this proposition.
  A network as above describes a map of $S$ into the classifying space of $BG$. The latter has the standard cell decomposition with $n$-dimensional cells given by $n$-tuples of elements of $G$. A map of $S$ to $BG$ can be made simplicial, with the image of $S$ lying in the 2-skeleton $BG^2$ of $BG$, via a map $\psi:S\lra BG^2$. Take  the Poincar\'e dual $P_2$ of the cell decomposition of $BG^2$. The inverse image $\psi^{-1}(P^1_2)$ of the 1-skeleton of $P_2$ gives a network on $S$ as described above. Vice versa, any network comes from such a simplicial map $\psi:S\lra BG^2$. Two homotopic maps of $S$ to $BG$ can both be made simplicial, giving maps $\psi_1,\psi_2:S\lra BG^2$. These maps are homotopic through a simplicial map $\psi:S\times [0,1]\lra BG^3$, where now the image lies in the 3-skeleton of $BG$, for some simplicial decomposition of $S\times [0,1]$. One can now connect $S\times \{0\}$ and $S\times \{1\}$ in $S\times [0,1]$ through a collection of surfaces $S_t$, for a finite subset of $t's$ in $[0,1]$, where two consecutive surfaces $S_t, S_t'$, that come with maps to $S\times [0,1]$, differ in an elementary way, through one of the Pachner moves for triangulations of surfaces~\cite{Pa} (and one additionally keeps track of $G$-labels of all edges on surfaces). These moves can be translated to the corresponding transformations of our networks. 
\endproof


The description of representations of the fundamental group via networks is Poincar\'e dual to the one commonly used in the  literature~\cite{Tu2,MS,Ka}.

\vspace{0.1in} 


We do not expect that $G$-valued networks on a surface (equivalently, elements of $\pi(S,G)$) can be evaluated consistently given the data $(A,\varepsilon)$ of a commutative Frobenius algebra and taking $G=G(A,\varepsilon)$ the group of trace-respecting automorphisms of $A$. Clearly, one needs much more structure to have a natural evaluation. 

\vspace{0.1in}

If $G$ is fixed, there is the notion of $G$-equivariant two-dimensional TQFT and corresponding $G$-equivariant commutative Frobenius algebra, see~\cite{Tu1,Tu2,MS,Ka}, much more sophisticated than that of a commutative Frobenius algebra. These structures do allow evaluations of surfaces with $G$-networks. Additionally, they define tensor functors on the corresponding categories of two-dimensional $G$-cobordisms, thus assigning vector spaces to $G$-labeled one-manifold, in a multiplicative way (disjoint union corresponds to tensor product of vector spaces).  

\vspace{0.1in}

Universal theories approach~\cite{BHMV,Kh4,KS,KKO,KL}  provides a different way to construct a topological theory, given evaluation function on networks on closed surfaces. Fix a  group $G$. Choose an \emph{evaluation} function, that is, a map of sets 
\begin{equation}\label{eq_alpha}
    \alpha \ : \ \pi(S,G)\lra \kk. 
\end{equation}

Given a closed oriented surface $S$ with a $G$-network $w$, define $\alpha(S,w):=\alpha(\rho)$, where $\rho$ is the equivalence class of homomorphisms defined by $w$ (equivalence under \emph{source conjugations}, that is, conjugations in $\pi_1(S)$). 

\vspace{0.1in}

With evaluation $\alpha$ for closed surfaces with a $G$-network at hand, we can define state spaces $\alpha(L)$ of decorated oriented one-manifolds $L$ in Section~\ref{subsubsec_state}, with $\alpha$ in place of $\mcF$ and $G$-networks on $S$ in place of collections of $G$-circles. First interesting question is funding families of evaluations $\alpha$ such that the state spaces $\alpha(L)$ are finite-dimensional for all $L$,by analogy with a study in~\cite{Kh4} and follow-up papers. Such evaluations may be called \emph{rational} or \emph{recognizable}.  

\vspace{0.1in} 


The state spaces $\alpha(L)$ may be zero for some $G$-decorated one-manifolds no matter what $\alpha$ is. For instance if $\sigma\in G\setminus [G,G]$ is not in the commutator subgroup, the  state space of a single circle $\SS^1(\sigma)$ with a mark $\sigma$ on it is trivial, since such circle cannot bound  any $G$-network $w$ on a surface $S$ with $\partial(S,w)\cong \SS^1(\sigma)$. We leave studying these state spaces and associated categories (as in~\cite{Kh4,KS,KKO}) for another paper.   

\vspace{0.1in} 

\begin{remark}
Following Turaev's homotopy TQFT, one can consider the case of maps of surfaces into a  path-connected topological space $X$ with $\pi_2(X)\not= 0$. The group $G:=\pi_1(X,x_0)$ acts linearly on the abelian group $B:=\pi_2(X,x_0)$, i.e., see \cite[Section 8.2]{FF}, and \cite[Section 7.1.ii]{BHS} for its generalization.  
 Consider oriented closed surfaces $S$ decorated by a $G$-network together with floating dots labeled by elements of $B$ and disjoint from the graph of the $G$-network. To relations in Figures~\ref{fig5_16a}, \ref{fig5_16b} one can add the rules in Figure~\ref{fig5_20} below. 

The relation between equivalence classes of these networks and homotopy classes of maps from $S$ to $X$ is discussed in~\cite[Remark 2.27]{IK2}.  
 
Universal theories can be further considered for such pairs $(G,B)$, and we hope to treat examples elsewhere. When $G=\{1\}$ is the trivial group, the network with each edge labeled $1$ may be erased, and $S$ is decorated only by dots that are elements of an abelian group $B$. Universal theories for this case are discussed in~\cite[Section 8]{KKO}. 
\end{remark}

\begin{figure}
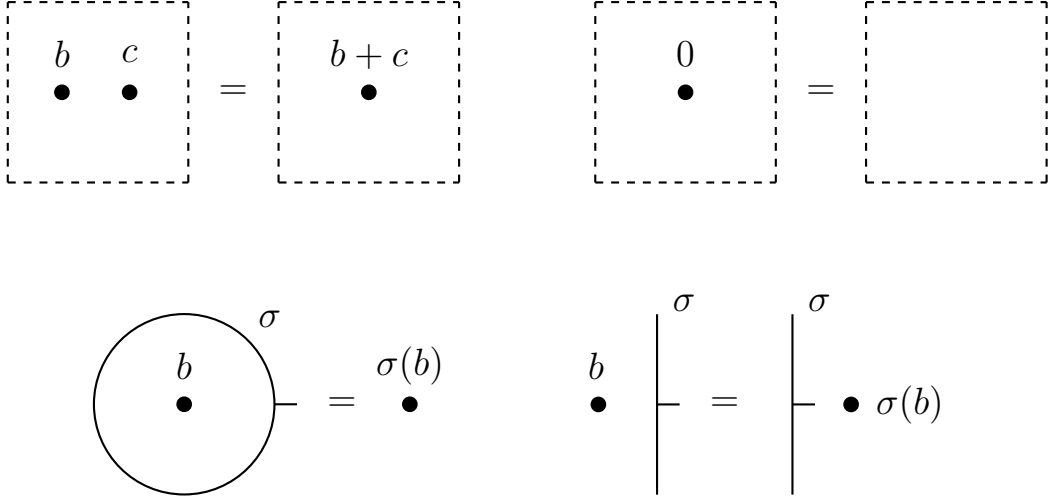

    \centering


    \caption{Top row relations allow to merge dots via addition in $B$ and to remove a dot labeled $0\in B$. The two relations in the bottom row are easily shown to be equivalent using an isotopy and the relation in Figure~\ref{fig5_5} on the right. The latter relation follows from the relations in the bottom row of Figure~\ref{fig5_16b}.  }
    \label{fig5_20}
\end{figure}


\vspace{0.1in}

One can further assume that $B$ is an abelian monoid with an action of $G$ on it rather than an abelian group. The notion of a  $(G,B)$-decoration of $S$ modulo Figure~\ref{fig5_16a}, \ref{fig5_16b}, \ref{fig5_20} relations makes sense, and one can consider universal theories and state spaces for such pairs as well, although there is no underlying topological space $X$  to interpret equivalences classes of $(G,B)$-decorations as homotopy classes of maps into $X$.


\begin{remark}
In this subsection, we describe the Poincar\'e dual diagrammatics (to the usual diagrammatics) for specifying representations of the fundamental groups of a surface, as well as propose to study universal theories for such representations, which should generalize Turaev's homotopy TQFTs in two dimensions. 
\end{remark}

\subsubsection{Basic structure of Frobenius automorphisms.} Fix an $\varepsilon$-automorphism $\sigma$ and assume that $\kk$ is algebraically closed (if $\kk$ is not closed, this can easily be achieved by passing to the algebraic closure via scalar extension, i.e., pass to $\overline{A}:=A\otimes_{\kk}\overline{\kk}$). Then $A$ decomposes into the direct sum of generalized weight spaces for $\sigma$, 
\begin{equation}
    A \ = \ \oplusop{\lambda} A_{\lambda}, \ \ \  (\sigma-\lambda)^N\big|_{A_{\lambda}} = 0, \quad   
    N \gg 0, 
    \quad  
    \lambda\in \kk^{\ast}.
\end{equation}
Note that $\lambda\not=0$ for a nonzero weight space $A_{\lambda}$, since $\sigma$ is an automorphism. 
We have $A_{\lambda}A_{\mu}\subset A_{\lambda\mu}$, making $A$ into a graded algebra, and $\sigma(A_{\lambda})=A_{\lambda}$. 

\vspace{0.1in} 

Let $\Lambda=\{\lambda\in\kk^{\ast}|A_{\lambda}\not=0\}$ 
be the subset of 
weights $\lambda$ such that $A_{\lambda}\not=0$. 
Let $\Lambda^{\ast}$ be the subgroup of $\kk^{\ast}$ generated by $\Lambda$. Algebra $A$ is naturally graded by the abelian group $\Lambda^{\ast}$. 

\vspace{0.1in} 

Note that, in general, $\Lambda\not= \Lambda^{\ast}$. As an example, consider a four-dimensional algebra with an automorphism $\sigma$ given by
\begin{equation}\label{eq_A_sigma}
    A = \kk[a,b]/(a^2,b^2), 
    \quad \sigma(a)=\lambda a, 
    \quad 
    \sigma(b)=\lambda^{-1}b,
\end{equation}
where $\lambda$ is any element of $\kk^{\ast}$, 
and the trace map 
\begin{equation} \label{eq_var_eps1}
    \varepsilon(ab)=1, 
    \quad  \varepsilon(1)=\varepsilon(a)=\varepsilon(b)=0.
\end{equation}
Then $\sigma$ is an $\varepsilon$-automorphism, $\Lambda=\{1,\lambda^{\pm 1}\}$, and $\Lambda^{\ast}$ is the subgroup generated by $\lambda$, infinite if $\lambda$ is not a root of unity in $\kk$.  

\begin{lemma}
The trace map $\varepsilon$ is zero on $A_{\lambda}$ for $\lambda\not= 1$.
\end{lemma}

\proof Let $v\in A_{\lambda}$ and assume first $\sigma(v)=\lambda v$ and $\varepsilon(v)\not=0$. Then $\varepsilon(\sigma(v))=\varepsilon(v)$, forcing $\lambda=1$. By induction on nilpotence degree of $v$ relative to $\sigma-\lambda$ we can assume $\varepsilon(\sigma(v)-\lambda v)=0$. Then $\varepsilon(\sigma(v))=\lambda \varepsilon(v)$ showing that either $\lambda=1$ or $\varepsilon(\sigma(v))=\lambda=0$. Applying this to $\sigma^{-1}(v)$ in place of $v$ and observing that $A_0=0$ completes the proof. 
\endproof

\begin{corollary} \label{cor_nondeg} $\varepsilon$ restricts to a nondegenerate pairing $A_{\lambda}\otimes A_{\lambda^{-1}}\lra \kk$ for each $\lambda\in\Lambda$. In particular, $A_1$ is a commutative Frobenius algebra. 
\end{corollary}

\begin{example}\label{ex_nonsemi}
 Consider Frobenius algebra $A$ in \eqref{eq_A_sigma} with the trace given by \eqref{eq_var_eps1} but a different $\sigma$:
 \begin{equation}
     \sigma(a)=a+b, 
     \quad 
     \sigma(b) =b, 
 \end{equation}
 which also requires $\mathrm{char}(\kk)=2$ to define $\sigma$. Then $A$ is the generalized $1$-eigenspace of $\sigma$ and the action of $\sigma$ is not semisimple. 
\end{example}

The direct sum of eigenspaces $A_{\lambda}'\subset A_{\lambda}$, over $\lambda\in \Lambda$, is a subalgebra of $A$ which is not necessarily Frobenius (see Example~\ref{ex_nonsemi}). 

\vspace{0.1in} 
    
There does not seem to be a substantial literature about Frobenius automorphisms; they are discussed in Wang~\cite{Wa} and several other papers.


\subsection{Field extensions and patched surfaces}

\subsubsection{Traces of field extensions} 
\label{subsubsec_traces}

Let $\kk\subset F$ be a field extension of finite degree $n$. The trace  map $\varepsilon: F\lra \kk,$ $\varepsilon(x)=\tr_{F/\kk}(m_x)$ assigns to $x\in F$ the trace of multiplication by $x$ map $m_x:F\lra F$, $m_x(a)=xa$, viewed as a
$\kk$-linear endomorphism of vector space $F$. A basic result on field extensions says that $\varepsilon$ is nondegenerate if and only if the extension is separable, {\it e.g.}, see \cite[Theorem 5.2]{Ja} and lecture notes by Conrad \cite{Co}. 

\begin{prop} The pair $(F,\varepsilon)$, for $\varepsilon=\tr_{F/\kk}$ and a finite separable extension $F/\kk$, is a commutative Frobenius $\kk$-algebra. Any element $\sigma\in \Gal(F/\kk)$ of the Galois group is an $\varepsilon$-automorphism.
\end{prop} 

\proof 
The first part of the proposition is the nondegeneracy statement  right before it. The second part is trivial, since an automorphism of $
\Gal(F/\kk)$ preserves traces of multiplication operators by elements of $F$.  
\endproof

Consequently, each extension $\kk\subset F$ as above defines a two-dimensional TQFT $\mcF$ with defect lines and Galois group $\Gal(F/\kk)$ being a subgroup of $G(F,\varepsilon)$. To $m$ circles, this TQFT associates $F^{\otimes m}$, with tensor product taken over $\kk$. The defect lines in this TQFT are labeled by elements of $\Gal(F/\kk)$. 

\vspace{0.1in} 

Closed surfaces with defect lines in $(\kk,F)$ TQFT admit a straightforward computation. 
That is, the action of $\Gal(F/\kk)$ on $F$ extends to $F\otimes_{\kk}\okk$ via the trivial action on the second factor. The resulting  action is $\okk$-linear. 

\vspace{0.1in} 

Writing $F$ as a simple extension, $F\cong \kk[x]/(f(x))$, a factorization in the algebraic closure
\begin{equation}
    f(x)=(x-\lambda_1)\cdots (x-\lambda_n),
\end{equation}
with distinct $\lambda_1,\dots, \lambda_n$, gives minimal idempotents for the direct product decomposition, 
\begin{equation}
e_i \ = \ \prod_{j\not= k} \frac{x-\lambda_j}{\lambda_k-\lambda_j}. 
\end{equation}
Although we cannot explicitly write down an action of $\sigma$ on $x$, we know that the action permutes minimal  idempotents $e_i$ in the same way $\sigma$ permutes roots $\lambda_i$. 

\vspace{0.1in} 

We therefore see that a finite extension $F/\kk$ is separable if and only if the tensor product $F\otimes_{\kk} \okk$ with the algebraic closure of $\kk$ is a direct product of copies of $\okk$, 
\begin{equation}
     F\otimes_{\kk} \okk \ \cong \okk \times \okk \times \dots \times \okk, 
\end{equation}
(necessarily of $[F:\kk]$ copies). Another equivalent condition is that $F\otimes_{\kk}\okk$ is a semisimple algebra (equivalently, $F\otimes_{\kk}\okk$ does not contain nilpotent elements). The trace of $m_a: F\lra F$ can be computed in the $\okk$-vector space $F\otimes_{\kk}\okk$.

\vspace{0.1in} 

This passage to the algebraic closure results in a 2D TQFT over the ground field  $\okk$ with the Frobenius algebra 
$F\otimes_{\kk}\okk$, trace map given by 
\begin{equation}
    \varepsilon (e_i) = 1, \ \  i=1,\dots, n
\end{equation}
and comultiplication 
\begin{equation}
    \Delta(e_i) = e_i \otimes e_i , \ \  i=1,\dots, n.
\end{equation}
Action of $\sigma\in \Gal(F/\kk)$ is given by permutation of $e_1,\dots, e_n$ corresponding to the action of $\sigma$ on the roots of $f(x)$. Consequently, each defect circle map can be computed in the basis of minimal idempotents as their permutation.

\vspace{0.1in} 

The map associated to any surface with defect circles can now be computed explicitly. For instance, consider a one-holed torus with a defect circle $\sigma$, see Figure~\ref{fig5_21} left.

\vspace{0.1in}

\begin{figure}
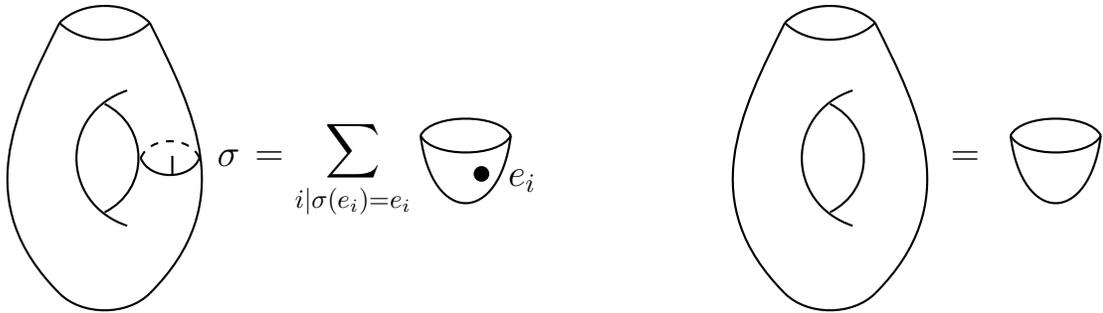

    \centering

    \caption{Left: simplification of a one-holed torus with a defect circle. Right: one-holed torus equals a disk.}
    \label{fig5_21}
\end{figure}

We compute the induced map 
\begin{equation}
    1=\sum_i e_i \stackrel{\Delta}{\lra} \sum_i e_i\otimes e_i \stackrel{1\otimes \sigma}{\lra} \sum_i e_i\otimes \sigma(e_i) \stackrel{m}{\lra}
    \sum_{i|\sigma(e_i)=e_i}e_i  \in F\otimes_{\kk}\okk.
\end{equation}
In particular, one-holed torus simplifies to a disk, see Figure~\ref{fig5_21} right. 

\vspace{0.1in}

Capping off the boundary by a disk, a 2-torus with an essential $\sigma$-defect 
circle evaluates to $[F^{\sigma}:\kk]$, the degree of the fixed field of $\sigma$ over $\kk.$  An undecorated 2-torus evaluates to $\dim_{\kk}(F)=[F:\kk]$, seen as an element of $\kk$. Over a field of finite characteristics, these evaluations may be equal to $0$.

\vspace{0.1in} 

More complicated cobordisms with circle defects can be computed analogously. For example, Figure~\ref{fig5_22} shows the evaluation of a genus two surface with 3 defect circles labeled $\sigma_1,\sigma_2,\sigma_3\in \Gal(F/\kk)$. 

\vspace{0.1in}

\begin{figure}
    \centering
\begin{tikzpicture}[scale=0.6]


\draw[thick] (0,0) .. controls (.25,3) and (7.75,3) .. (8,0);
\draw[thick] (0,0) .. controls (.25,-3) and (7.75,-3) .. (8,0);

\draw[thick] (3,1) .. controls (2,.5) and (2,-.5) .. (3,-1);
\draw[thick] (2.5,.65) .. controls (3,.5) and (3,-.5) .. (2.5,-.65);

\draw[thick] (5.75,1) .. controls (4.75,.5) and (4.75,-.5) .. (5.77,-1);
\draw[thick] (5.25,.65) .. controls (5.75,.5) and (5.75,-.5) .. (5.25,-.65);

\draw[thick] (0,0) .. controls (0.25,-.5) and (2.00,-.5) .. (2.25,0);
\draw[thick,dashed] (0,0) .. controls (0.25,.5) and (2.00,.5) .. (2.25,0);

\draw[thick] (2.9,0) .. controls (3.15,-.5) and (4.75,-.5) .. (5,0);
\draw[thick,dashed] (2.9,0) .. controls (3.15,.5) and (4.75,.5) .. (5,0);

\draw[thick] (5.6,0) .. controls (5.85,-.5) and (7.75,-.5) .. (8,0);
\draw[thick,dashed] (5.6,0) .. controls (5.85,.5) and (7.75,.5) .. (8,0);

\draw[thick] (1.15,-.4) -- (1.15,.1);
\draw[thick] (4,-.4) -- (4,.1);
\draw[thick] (6.85,-.4) -- (6.85,.1);

\node at (-.50,-0.4) {\Large $\sigma_1$};
\node at (4,-0.85) {\Large $\sigma_2$};
\node at (8.60,-0.4) {\Large $\sigma_3$};

\node at (9.75,0) {\Large $=$};

\node at (18,0) {\large $\#\left\{ i: 1\leq i\leq n, \sigma_1(e_i)=\sigma_2(e_i)=\sigma_3(e_i)\right\}$};

\end{tikzpicture}
    \caption{Evaluating a genus two surface with three circle defects. }
    \label{fig5_22}
\end{figure}
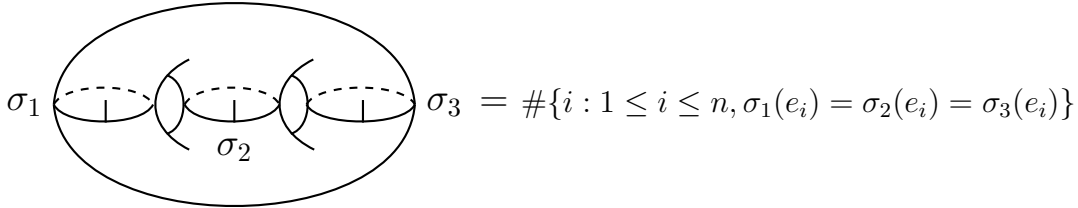

\subsubsection{A chain of extensions and evaluation of patched surfaces}
\label{Chain_ext_patched_surfaces}
Recall that for a finite Galois extension $\kk\subset F$, the trace map is given by 
\begin{equation}
    \tr_{F/\kk}(a) \ = \ \sum_{\sigma\in \Gal(F/\kk) }\sigma(a). 
\end{equation}

Consider now a chain of finite separable field extensions $\kk\subset F \subset K$ with $[K:F]=n$, $[F:\kk]=m$. The trace maps 
\begin{equation*}
    \tr_{F/\kk}: F\lra \kk, \quad  
    \tr_{K/F}: K \lra F, \quad 
    \tr_{K/\kk}:K \lra \kk 
\end{equation*}
are non-degenerate and satisfy 
\begin{equation}
    \tr_{K/\kk} \ = \tr_{F/\kk}\circ \tr_{K/F},
\end{equation}
see \cite[Chapter 1, Theorem 5.2]{Ja}, 
and turn $F$ into a commutative Frobenius $\kk$-algebra and $K$ into a commutative Frobenius algebra over $F$ and over $\kk$~\cite[Sections 2.2.13 and 2.2.17]{Kc1}.

\vspace{0.1in}

Three commutative Frobenius algebras with traces
\begin{equation}
    (F,\kk,\tr_{F/\kk}), 
    \qquad 
    (K,F,\tr_{K/F}), 
    \qquad  
    (K,\kk,\tr_{K/\kk})
\end{equation}
give rise to three two-dimensional TQFTs that we denote by
\begin{equation}
    \mcF_F =\mcF_{F/\kk}, 
    \qquad 
    \mcF_K=\mcF_{K/F}, 
    \qquad 
    \mcF_{K/F},
\end{equation}
respectively. Defect lines in these TQFTs are labeled by elements of the corresponding Galois groups 
\begin{equation*}
    \Gal(F/\kk), \qquad  
    \Gal(K/F), \qquad 
    \Gal(K/\kk). 
\end{equation*}
It is natural to ask whether these three TQFTs can be combined into a single structure, and we now suggest one possible approach, first without the defect lines. 

\vspace{0.1in}

Consider a ``patched" or ``seamed" closed oriented surface $S$ which consists of regions labeled $F$ and $K$. Elements of $F$ and $K$ may float in the regions labeled by the corresponding field. 
Seam circles separate regions labeled $F$ and $K$, see Figure~\ref{fig5_2_1} for an example.

\vspace{0.1in}

\begin{figure}
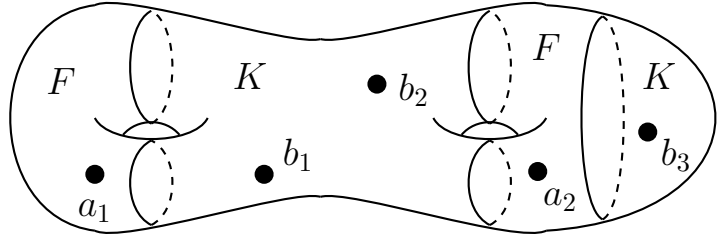

    \centering

    \caption{Seamed surface with facets checkerboard colored by fields $F,K$. Elements $a_i\in F$ and $b_j\in K$ float in the corresponding regions.   }
    \label{fig5_2_1}
\end{figure}

To evaluate such a surface to an element $\mcF(S)\in \kk$ let us use neck-cutting relations to separate each seam circle from the rest of the diagram. We do surgery on both sides of a seam circle using neck-cutting relations in extensions $F/\kk$ and $K/\kk$, correspondingly. Choose dual bases: 
\begin{itemize}
    \item $\{x_i\},\{y_i\}, 1\le i \le m$ for the Frobenius pair $(F,\kk,\tr_{F/\kk})$, 
    \item $\{x_j'\},\{y_j'\}, 1\le j \le n$ for the Frobenius pair $(K,F,\tr_{K/F})$, 
    \item $\{x_k''\},\{y_k''\}, 1\le k \le mn$ for the Frobenius pair $(K,\kk,\tr_{K/\kk})$. 
\end{itemize}
Note that we may choose $\{x_ix_j'\},\{y_iy_j'\}$ as dual bases for the third extension. 

\vspace{0.1in} 

Each seam circle $C$ bounds one component (facet) labeled $F$ and one labeled $K$. Choose circles parallel to $C$ in each of these components and apply neck-cutting along these circles as shown in Figure~\ref{fig5_2_2}. 

\vspace{0.1in}

\begin{figure}
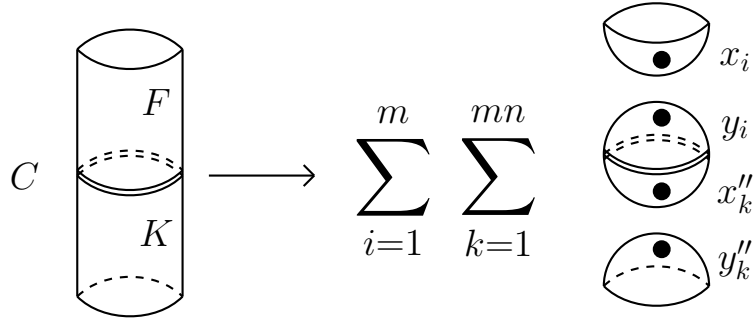

    \centering

    \caption{ Surgery on a seam circle. Dotted seamed spheres are then evaluated using trace maps.  }
    \label{fig5_2_2}
\end{figure}

Doing this neck-cutting along each seam circle converts $S$ into a sum of terms which are disjoint unions of connected components of three types:
\begin{itemize}
    \item closed connected surfaces labeled $F$ with elements of $F$ floating on them, 
    \item closed connected surfaces labeled $K$ with element of $K $ floating on them, 
    \item spheres with a seam circle and an element of $F$, respectively $K$,  in a disk labeled $F$, respectively $K$, see Figure~\ref{fig5_2_3}. 
\end{itemize}

\begin{figure}
    \centering
\begin{tikzpicture}[scale=0.6]
\draw[thick] (4,0) arc (0:360:2);

\draw[thick,dashed] (0,0) .. controls (.5,.5) and (3.5,.5) .. (4,0);
\draw[thick] (0,0) .. controls (.5,-.5) and (3.5,-.5) .. (4,0);

\draw[thick,dashed] (0,0.1) .. controls (.5,.6) and (3.5,.6) .. (4,0.1);
\draw[thick] (0,0.1) .. controls (.5,-.4) and (3.5,-.4) .. (4,0.1);

\draw[thick,fill] (3.25,-1.1) arc (0:360:0.25);
\node at (3.75,-1.75) {\Large $b$};

\draw[thick,fill] (3.25,1.1) arc (0:360:0.25);
\node at (3.75,1.75) {\Large $a$};

\node at (-.25, 1.75) {\Large $F$};
\node at (-.25,-1.75) {\Large $K$};
\end{tikzpicture}













    \caption{Seamed 2-sphere, denoted $\SS^2(a,b)$, with dots $a\in F$ and $b\in K$ floating in the $F$-disk and $K$-disk, respectively. }
    \label{fig5_2_3}
\end{figure}

Components of the first and second kind are evaluated via TQFTs for $(F,\kk,\tr_{F/\kk})$ and $(K,\kk,\tr_{K/\kk})$, respectively, to yield elements of $\kk$. We consider the following evaluation of the seamed 2-sphere $\SS^2(a,b)$: 
\begin{equation}\label{eq_eval}
    \mcF(\SS^2(a,b)) =  \tr_{F/\kk}( a\, \tr_{K/F}(b)).
\end{equation}
We can interpret this evaluation, see Figure~\ref{fig5_2_4}, as first removing the $K$-disk with dot $b$ and inserting dot $\tr_{K/F}(b)\in F$ in its place, now floating on the 2-sphere labeled $F$ alongside the original dot $a$. Now multiply the two dots and evaluate using $\varepsilon_F=\tr_{F/\kk}$. Figure~\ref{fig5_2_4} shows the two steps in this evaluation. 

\vspace{0.1in}

\begin{figure}
    \centering
\begin{tikzpicture}[scale=0.6]
\draw[thick] (4,0) arc (0:360:2);

\draw[thick,dashed] (0,0) .. controls (.5,.5) and (3.5,.5) .. (4,0);
\draw[thick] (0,0) .. controls (.5,-.5) and (3.5,-.5) .. (4,0);

\draw[thick,dashed] (0,0.1) .. controls (.5,.6) and (3.5,.6) .. (4,0.1);
\draw[thick] (0,0.1) .. controls (.5,-.4) and (3.5,-.4) .. (4,0.1);

\draw[thick,fill] (3.25,-1.1) arc (0:360:0.25);
\node at (2.25,-1.20) {\Large $b$};

\draw[thick,fill] (3.25,1.1) arc (0:360:0.25);
\node at (2.25,1.30) {\Large $a$};

\node at (-.25, 1.75) {\Large $F$};
\node at (-.25,-1.75) {\Large $K$};

\draw[thick,->] (5,0) -- (7,0);

\begin{scope}[shift={(0.5,0)}]
 
\draw[thick] (11.5,0) arc (0:360:2);

\node at (7.2, 1.75) {\Large $F$};

\draw[thick,dashed] (7.5,0) .. controls (8,.4) and (11,.4) .. (11.5,0);
\draw[thick,dashed] (7.5,0) .. controls (8,-.4) and (11,-.4) .. (11.5,0);


\draw[thick,fill] (9.25,1.25) arc (0:360:0.25);
\node at (9.75,1.25) {\Large $a$};

\draw[thick,fill] (8.75,-.65) arc (0:360:0.25);
\node at (9.5,-1.28) {$\tr_{K/F}(b)$};

\draw[thick,->] (12.5,0) -- (14.5,0);

\node at (19,0) {\Large$\tr_{F/\kk}(a\,\tr_{K/F}(b))$};
\end{scope}

\end{tikzpicture}
    \caption{Evaluation of the seamed sphere $\SS^2(a,b)$, given by pushing $b$ via relative trace into the $F$-facet and evaluating via $\tr_{F/\kk}$. }
    \label{fig5_2_4}
\end{figure}
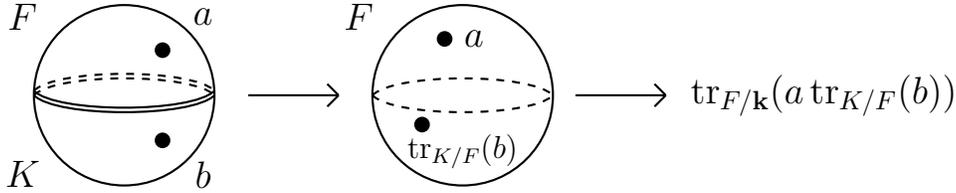

With evaluations for all three types of connected components at hand, we know how to evaluate an arbitrary seamed $(F,K)$-surface $S$ as above. In the sum resulting after neck-cutting, for each term we take the product of evaluations of all connected components and then sum these elements of $\kk$. Denote this evaluation by $\mcF(S)$ or $\langle S \rangle$. 

\vspace{0.1in}

It is easy to see that, if a seam circle $C$ bounds an $F$-disk or a $K$-disk on one side (or such disks on both side), possibly with some dots in them, then one can skip the neck-cutting procedure on the corresponding side of $C$ (or on both sides of $C$) without changing the evaluation. This observation implies  relations in the top row of Figure~\ref{fig5_2_5}.  

\vspace{0.1in}

\begin{figure}
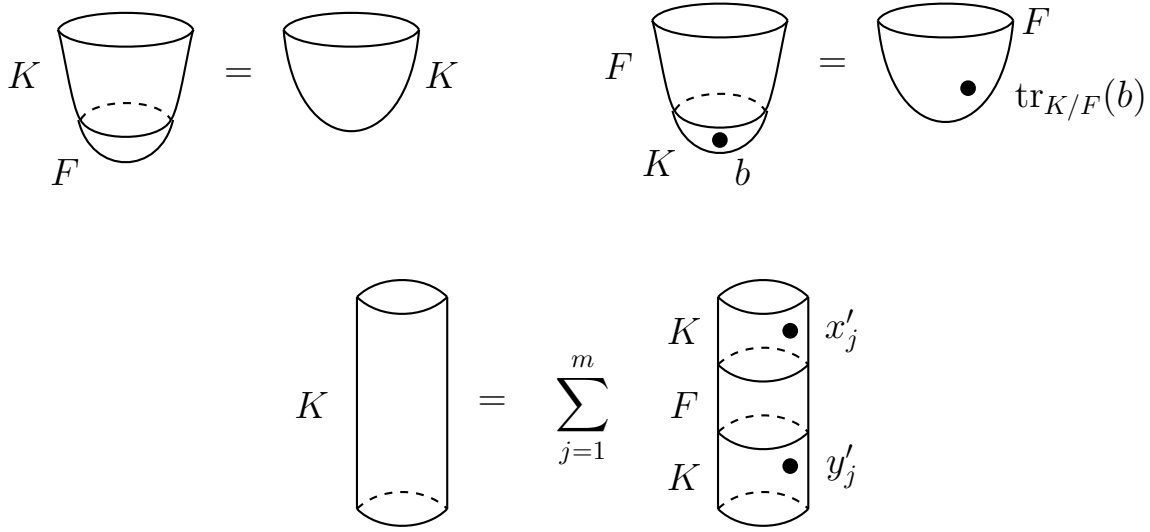

    \centering


    \caption{Some skein relations. Top left: a dotless $F$-disk may be removed. Top right: pushing a dot off a $K$-disk. Bottom: partial neck-cutting from $K$ to $F$.}
    \label{fig5_2_5}
\end{figure}

Specializing Figure~\ref{fig5_2_5} relation on the top right to $b\in F$, the dot on the right hand side has label $nb$, since $\tr_{K/F}(b)=nb$, where $n=[K:F]$. Consequently, if $\nchar{\kk}=p$ and $p|n$, the right hand side is $0$.

\vspace{0.1in}

 
 A dot on an $F$-component can be pushed across a seam into an adjacent $K$-component, since the trace $\tr_{K/F}$ is $F$-linear, see Figure~\ref{fig5_2_6}. If the $F$-component were a disk, it can then be removed. 
 
 \vspace{0.1in}

\begin{figure}
    \centering
\begin{tikzpicture}[scale=0.6]
\draw[thick,dashed] (0,0) -- (4,0);
\draw[thick,dashed] (0,0) -- (0,4);
\draw[thick,dashed] (4,0) -- (4,4);
\draw[thick,dashed] (0,4) -- (4,4);
\node at (5,2) {\Large $=$};
\draw[thick,dashed] (6,0) -- (10,0);
\draw[thick,dashed] (6,0) -- (6,4);
\draw[thick,dashed] (10,0) -- (10,4);
\draw[thick,dashed] (6,4) -- (10,4);

\draw[thick] (2,0) -- (2,4);
\draw[thick] (8,0) -- (8,4);

\node at (1.25,3.35) {\Large $F$};
\node at (2.75,3.35) {\Large $K$};

\draw[thick,fill] (0.85,1) arc (0:360:0.25);
\node at (1.25,.65) {\Large $a$};

\node at (7.25,3.35) {\Large $F$};
\node at (8.75,3.35) {\Large $K$};

\draw[thick,fill] (9.15,1.15) arc (0:360:0.25);
\node at (9.45,.65) {\Large $a$};

\end{tikzpicture}
\qquad \qquad 
\begin{tikzpicture}[scale=0.6]
\draw[thick,dashed] (0,0) .. controls (0.25,.5) and (1.75,.5) .. (2,0);
\draw[thick] (0,0) .. controls (0.25,-.5) and (1.75,-.5) .. (2,0);

\draw[thick] (-.5,2) .. controls (-.25,2.5) and (2.25,2.5) .. (2.5,2);
\draw[thick] (-.5,2) .. controls (-.25,1.5) and (2.25,1.5) .. (2.5,2);
\node at (0.40,1.05) {\Large $K$};

\draw[thick] (-.5,2) .. controls (-.4,1.75) and (-.2,.25) .. (0,0); 
\draw[thick] (2.5,2) .. controls (2.4,1.75) and (2.2,.25) .. (2,0); 

\draw[thick] (-0.05,0) .. controls (0.20,-1.75) and (1.80,-1.75) .. (2.05,0);
\node at (-.35,-1.15) {\Large $F$};

\draw[thick,fill] (1.40,-0.85) arc (0:360:0.25);
\node at (2.15,-1.25) {\Large $a$};

\node at (3.5,1) {\Large $=$};

\draw[thick] (4.5,2) .. controls (4.75,2.5) and (7.25,2.5) .. (7.5,2);
\draw[thick] (4.5,2) .. controls (4.75,1.5) and (7.25,1.5) .. (7.5,2);

\draw[thick] (4.5,2) .. controls (4.75,-1) and (7.25,-1) .. (7.5,2);
\node at (5.5,1.05) {\Large $K$};

\draw[thick,fill] (6.75,0.5) arc (0:360:0.25);
\node at (7.45,0.25) {\Large $a$};
\end{tikzpicture}

    \caption{Left: pushing an $F$-dot $a\in F$ across a seam into a $K$-component. Right: removing a dotted $F$-disk.  }
    \label{fig5_2_6}
\end{figure}
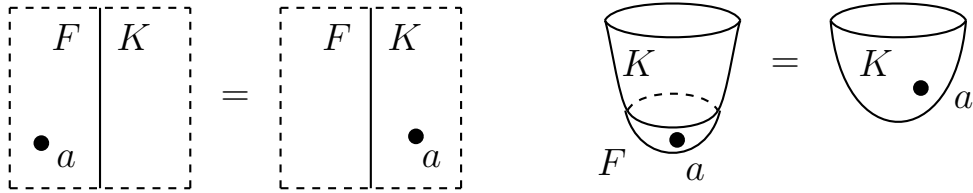

This allows us to move dots away from any $F$-component that  bounds a seam (otherwise it is a connected component of $S$).
Furthermore, Figure~\ref{fig5_2_7} relation holds. It allows to reduce $S$ to a surface where each $K$-facet has at most one boundary component. The relation can  be checked by doing surgeries on $F$-sides of the two circles  on the left hand  side, then using relations in Figure~\ref{fig5_2_6}.

\begin{figure}
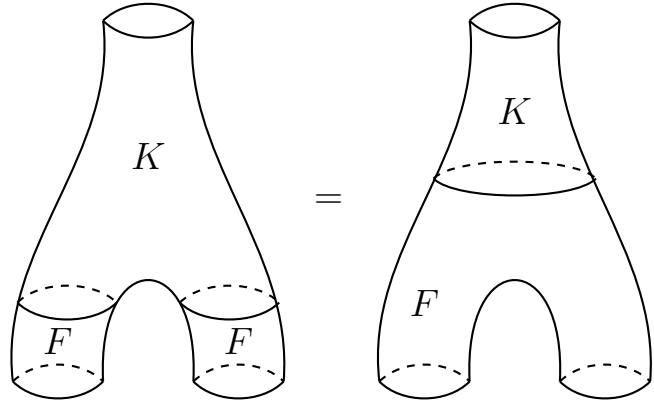

    \centering


    \caption{Different boundary  components of a $K$-facet can be merged into one.  }
    \label{fig5_2_7}
\end{figure}

\vspace{0.1in} 

Similar to the discussion in Section~\ref{subsubsec_traces}, such patched surfaces can be evaluated by tensoring the fields with the algebraic closure $\okk$ and looking at the chain of $\okk$-algebra inclusions and trace maps between them 
\begin{equation}
    \okk \subset F\otimes_{\kk}\okk \subset K\otimes_{\kk} \okk.
\end{equation}
Both rings $F\otimes_{\kk}\okk$, $K\otimes_{\kk}\okk$ are direct product of fields $\okk$, and under the inclusion
$F\otimes_{\kk}\okk \subset K\otimes_{\kk}\okk$
 minimal idempotents in 
$K\otimes_{\kk}\okk$ go to sums of distinct minimal idempotents of $K\otimes_{\kk}\okk$. Relative traces have a similar simple description. This allows to easily evaluate a patched surface to an element of $\kk$. In this way, the TQFT reduces to set-theoretic computations with roots of an irreducible polynomial describing the extension $K/\kk$, together with the Galois group action on the roots, and the  partition of roots corresponding to the subfield $F\subset K$.

\vspace{0.1in}

If one allows elements of $F$ and $K$ to float in the corresponding regions of the surface, the evaluation requires decomposing these elements in the bases of minimal idempotents of 
$F\otimes_{\kk}\okk$ and $K\otimes_{\kk}\okk$.

\subsubsection{Defect lines and Galois symmetries} 

Here we use notations from the previous subsection, including having a chain of finite separable field extensions $\kk\subset F\subset K$.

In this section, we can extend patched surfaces setup in Section~\ref{Chain_ext_patched_surfaces} by adding defect lines for Galois symmetries of extensions $F/\kk$ and $K/\kk$. Let $S$ be a patched $(F,K)$-surface. 
Choose a collection of disjoint circles with co-orientations on $F$-patches of $S$ and label each of  them by an element $\sigma\in\Gal(F/\kk)$ (not necessarily the same one). 
Likewise, choose a collection of disjoint circles with co-orientations on $K$-patches of $S$ and label each of them by an element $\tau\in \Gal(K/\kk)$. As before, dots labeled by elements of $F$ and $K$ may float in $F$- and $K$-regions of $S$, correspondingly. To evaluate such a decorated patched surface $S$, one applies neck-cutting around each of three types of seamed circles of $F$: 
\begin{itemize}
    \item $(F,K)$-circles, along which $K$- and $F$-regions of $S$ meet, 
    \item  $\Gal(F/\kk)$-circles in $F$-patches, \item $\Gal(K/\kk)$-circles in $K$-patches.
\end{itemize}
After that, evaluation reduces to the familiar cases that have already been discussed. We can denote the evaluation by $\mcF(S)$ or $\brak{S}\in \kk$.  

\vspace{0.1in} 

If $\sigma\in \Gal(K/\kk)$ preserves the subfield $F$, one can allow $\sigma$-circles to intersect seam lines separating $F$- and $K$-regions of $S$, see Figure~\ref{fig5_25}. 

\vspace{0.1in}

\begin{figure}
    \centering
\begin{tikzpicture}[scale=0.6]
\draw[thick] (0,0) -- (6,0);
\draw[thick,dashed] (0,-2) -- (0,2);
\draw[thick,dashed] (6,-2) -- (6,2);
\draw[thick] (4.5,-2) -- (4.5,2);
\draw[thick] (4.5,1) -- (5,1);

\node at (1, .75) {\Large $K$};
\node at (1,-.75) {\Large $F$};
\node at (5,-1.5) {\Large $\sigma$};

\draw[thick] (0.5,2) .. controls (0.75,1.5) and (2.75,1.5) .. (3,2.75);
\node at (3.4,2.45) {\Large $\tau$};

\draw[thick] (2,1.8) -- (1.8,2.4);

\end{tikzpicture}
    \caption{A crossing of a $\sigma$-circle and a seam line. }
    \label{fig5_25}
\end{figure}

Such a network can be evaluated as before, by tensoring all fields with $\okk$, representing $K$ as a simple extension, and working with the set of roots of the corresponding irreducible polynomial.

\vspace{0.1in} 

Finally, one can consider arbitrary finite Galois extensions $\kk\subset F$ and patched surfaces with regions labeled by finite extensions $F_i$. A seam circle separating regions labeled $F_i$ and $F_j$ is assumed to be co-oriented, and an inclusion $F_i \subset F_j$ is assigned to each such circle, see Figure~\ref{fig5_23} left. Elements of $F_i$ may float in $F_i$-regions. These regions may contain $\sigma$-circles, for $\sigma\in\Gal(F_i/\kk)$. Furthermore, these circles can be viewed as a special case of the seam circles of the first type, for the case when fields $F_i=F_j$ and the inclusion (isomorphism) $F_i\subset F_i$ is given by $\sigma.$

\vspace{0.1in}

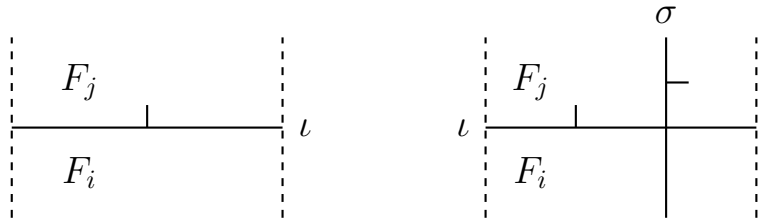
\begin{figure}
    \centering    
\begin{tikzpicture}[scale=0.6]
\draw[thick] (0,0) -- (6,0);
\draw[thick,dashed] (0,-2) -- (0,2);
\draw[thick,dashed] (6,-2) -- (6,2);
\draw[thick] (3,0) -- (3,0.5);

\node at (1.5,1) {\Large $F_j$};
\node at (1.5,-1) {\Large $F_i$};

\node at (6.5,0) {\Large $\iota$};
\end{tikzpicture}
\qquad 
\qquad 
\begin{tikzpicture}[scale=0.6]
\draw[thick] (0,0) -- (6,0);
\draw[thick,dashed] (0,-2) -- (0,2);
\draw[thick,dashed] (6,-2) -- (6,2);
\draw[thick] (2,0) -- (2,0.5);

\draw[thick] (4,-2) -- (4,2);
\draw[thick] (4,1) -- (4.5,1);

\node at (1,1) {\Large $F_j$};
\node at (1,-1) {\Large $F_i$};

\node at (4,2.5) {\Large $\sigma$};

\node at (-.5,0) {\Large $\iota$};
\end{tikzpicture}
    \caption{Left: $\iota:F_i\hookrightarrow F_j$ is a field inclusion. 
    Right: A seam line that can intersect $(F_i,F_j)$-line corresponds to an automorphism $\sigma:F_j\rightarrow F_j$ that preserves the subfield $F_i$, that is,   $\sigma(\iota(F_i))=\iota(F_i)$.}
    \label{fig5_23}
\end{figure}

The seamed circles labeled by Galois group elements may intersect seamed circles of the first type, as long as the Galois symmetry $\sigma$ preserves the fields $F_i$ for all regions along the $\sigma$-circle, see Figure~\ref{fig5_23}.  

\vspace{0.1in}

Such closed networks can then be evaluated by working with a splitting field $K$ that contains copies of fields $F_i$, over all patches of the surface $S$, tensoring with the algebraic closure $\okk$, and working with minimal idempotents in $K\otimes_{\kk}\okk$ to evaluate the network. 

\vspace{0.1in}

With these evaluations at hand, one can then define state spaces for collections of circles that are patched from intervals labeled by various fields $F_i$ separated by points labeled by inclusions $F_i\subset F_j$ and Galois symmetries $\sigma:F_i\lra F_i$.

\vspace{0.1in}

The same approach allows to evaluate even more general networks. Namely, beside seam circles for inclusions $F_i\subset F_j$ one can consider networks with co-oriented edges labeled by inclusions $F_i\subset F_j$ that may contain trivalent (or even more general) vertices where three regions meet, see Figure~\ref{fig5_24} left.

\vspace{0.1in}

\begin{figure}
    \centering
\begin{tikzpicture}[scale=0.6]

\begin{scope}[shift={(0,0)}]


\draw[thick] (0,0) -- (0,2);
\draw[thick] (0,2) -- (2,4);
\draw[thick] (0,2) -- (-2,4);

\draw[thick] (0,1) -- (0.6,1);

\node at (-1.75,1.5) {\Large $F_1$};
\node at (2.00,1.5) {\Large $F_2$};
\node at (0,4.5) {\Large $F_3$};

\node at (0,-0.75) {\Large $\iota$};
\node at (2.5,4.5) {\Large $\jmath$};
\node at (-2.5,4.5) {\Large $\jmath\iota$};

\draw[thick] (-1,3) -- (-0.5,3.5);
\draw[thick] (1.0,3) -- (0.5,3.5);
\end{scope}

\begin{scope}[shift={(11,0)}]

\draw[thick] (0,0) -- (0,2);
\draw[thick] (0,2) -- (2,4);
\draw[thick] (0,2) -- (-2,4);

\draw[thick] (0,1) -- (0.6,1);

\node at (-1.25,0.75) {\Large $F$};
\node at (1.25,0.75) {\Large $F$};
\node at (0,4.5) {\Large $F$};

\node at (0,-0.75) {\Large $\sigma\tau$};
\node at (2.5,4.5) {\Large $\tau$};
\node at (-2.5,4.5) {\Large $\sigma$};

\draw[thick] (-1,3) -- (-0.5,3.5);
\draw[thick] (1.0,3) -- (1.5,2.5);

\draw[thick,fill] (3.0,2.25) arc (0:360:0.25); 
\node at (3.5,1.75) {\Large $a$};

\draw[thick,fill] (-1.5,2.5) arc (0:360:0.25);
\node at (-2.40,2) {\Large $b$};
\end{scope}
\end{tikzpicture}
    \caption{Left: $F_1$, $F_2$, and an $F_3$-region meet at a vertex, with the inclusion $F_1\subset F_3$ given by the composition $\jmath\circ\iota$ of inclusions $\iota:F_1\xhookrightarrow{} F_2$, $\jmath: F_2\xhookrightarrow{} F_3$. Right: as a special case, when all 3 fields are $F$ and the inclusions are isomorphisms in $G=\Gal(F/\kk)$, the networks match those that appear in homotopy 2D TQFTs with the group $G$. Picking an abelian subgroup $A\subset F$ stable under $G$ and allowing dots labeled by elements of $A$ to float in the regions corresponds to working with a space $X$ with $\pi_1(X)\cong G$ and $\pi_2(X)\cong A$ with the matching action of $G$ on $A$.  }
    \label{fig5_24}
\end{figure}
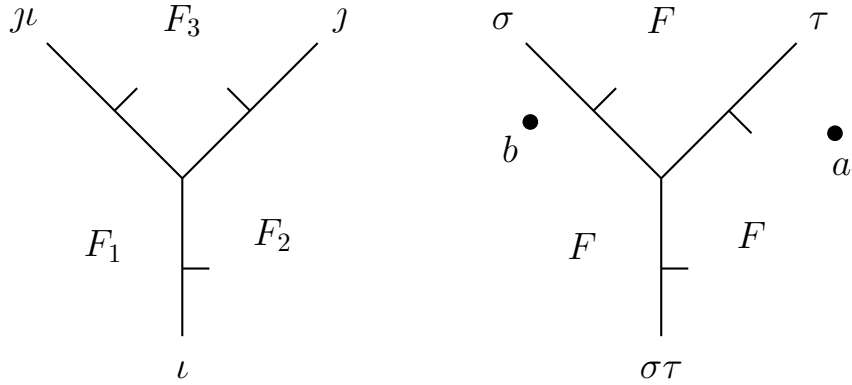

In the special case when all fields are the same field $F$ and seam edges are Galois symmetries $\sigma\in \Gal(F/\kk)$,
these networks are identical to those in Section~\ref{subsubsec_turaev}. One then obtains a special case of Turaev's homotopy 2D TQFTs, where networks describe conjugacy classes of homomorphisms $\pi_1(S)\lra \Gal(F/\kk)$, up to conjugation by elements of $\pi_1(S)$ or, more generally, homotopy classes of maps $S\lra X$ for $X$ as specified in the caption for Figure~\ref{fig5_24}. 

\vspace{0.1in}

Suitable state spaces for decorated patched circles for these theories can then be studied. 

\vspace{0.1in} 

It is possible to further refine the theory by introducing ``orbifold" points with a nontrivial ``monodromy" around them. These points may be located on facets of a network, along seam lines, along Galois ($\sigma$-defect) lines, and at vertices of the network, see Figures~\ref{fig5_26}, \ref{fig5_27} and \ref{fig5_28}. An orbifold point with a label $\sigma$ inside a facet (type (1) point, shown in Figure~\ref{fig5_26}) can be defined via a connected sum with a torus with a $\sigma$-defect circle. In the state sum, only idempotents $e_i$  with $\sigma(e_i)=e_i$ placed on that facet will contribute to the evaluation. 

\vspace{0.1in}

\begin{figure}
    \centering
\begin{tikzpicture}[scale=0.6]
\node at (5.5,-3) {\Large $(1)$};

\draw[thick,dashed] (7,-6) -- (13,-6);
\draw[thick,dashed] (7,-6) -- (7,0);
\draw[thick,dashed] (7,0) -- (13,0);
\draw[thick,dashed] (13,0) -- (13,-6);

\node at (8,-3) {\Large $F$};
\node at (10,-2.25) {\Large $\sigma$};
\node at (10,-3) {\Large $\times$};

\node at (14,-3) {\Large $=$};

\draw[thick,dashed] (16,0) .. controls (16.5,.5) and (17.5,.5) .. (18,0);
\draw[thick,dashed] (16,0) .. controls (16.5,-.5) and (17.5,-.5) .. (18,0);
\draw[thick] (16,0) .. controls (15,-2) and (14,-4) .. (16,-6);
\draw[thick] (18,0) .. controls (19,-2) and (20,-4) .. (18,-6);

\draw[thick] (16,-6) .. controls (16.5,-6.5) and (17.5,-6.5) .. (18,-6);

\draw[thick] (17.5,-1.5) .. controls (16,-2) and (16,-4) .. (17.5,-4.5);
\draw[thick] (17,-1.8) .. controls (18,-2.3) and (18,-3.7) .. (17,-4.2);

\draw[thick] (17.8,-3) .. controls (18,-3.5) and (18.9,-3.5) .. (19.1,-3);
\draw[thick,dashed] (17.8,-3) .. controls (18,-2.5) and (18.9,-2.5) .. (19.1,-3);

\draw[thick] (18.5,-3.35) -- (18.5,-2.95);

\node at (19.75,-3) {\Large $\sigma$};

\node at (10,-7) {};
\end{tikzpicture}
\qquad \qquad 
\begin{tikzpicture}[scale=0.6]
\draw[thick,dashed] (7,-6) -- (13,-6);
\draw[thick,dashed] (7,-6) -- (7,0);
\draw[thick,dashed] (7,0) -- (13,0);
\draw[thick,dashed] (13,0) -- (13,-6);

\node at (10,-2.25) {\Large $\sigma$};
\node at (10,-3) {\Large $\times$};
\node at (11,-4) {\Large $e_i$};
\node at (10,-7) {\Large $\sigma(e_i)=e_i$};

\end{tikzpicture}
    \caption{Type (1) orbifold point, on an $F$-patch of surface.}
    \label{fig5_26}
\end{figure}
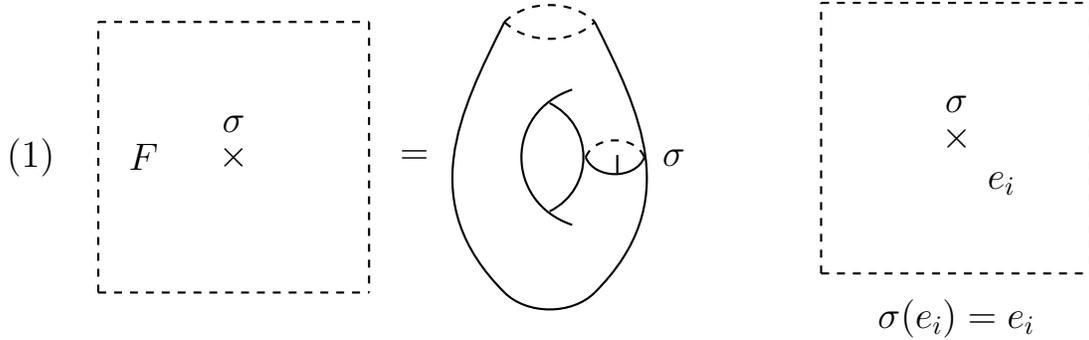

At an orbifold point on a $\sigma$-defect circle (type (2a) orbifold point) the automorphism label in $\Gal(F/\kk)$ may change from $\sigma_0$ to $\sigma_1$,  see Figure~\ref{fig5_27}. In the corresponding evaluation, only minimal idempotents $e_i$ with $\sigma_0(e_i)=\sigma_1(e_i)$ may contribute.  

\vspace{0.1in}

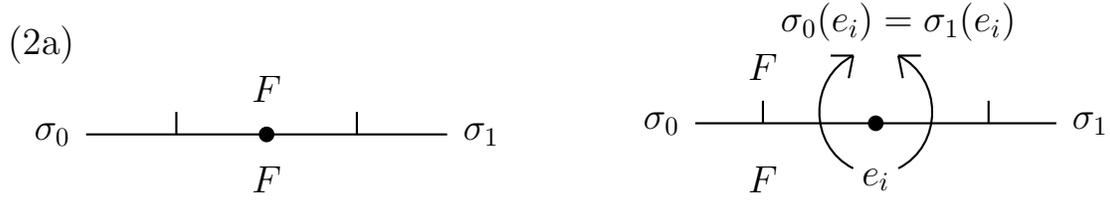
\begin{figure}
    \centering
\begin{tikzpicture}[scale=0.6]

\begin{scope}[shift={(0,0)}]
\draw[thick] (0,0) -- (8,0);
\draw[thick] (2,0) -- (2,0.5);
\draw[thick] (6,0) -- (6,0.5);

\node at (-0.75,0) {\Large $\sigma_0$};
\node at (8.75,0) {\Large $\sigma_1$};

\draw[thick,fill] (4.25,0) arc (0:360:0.25);

\node at (4, 1) {\Large $F$};
\node at (4,-1) {\Large $F$};

\node at (-1,2) {\Large (2a)};
\end{scope}

\begin{scope}[shift={(14,0)}]

\draw[thick] (0,0) -- (8,0);
\draw[thick] (1.5,0) -- (1.5,0.5);
\draw[thick] (6.5,0) -- (6.5,0.5);

\node at (-0.75,0) {\Large $\sigma_0$};
\node at (8.75,0) {\Large $\sigma_1$};

\draw[thick,fill] (4.25,0) arc (0:360:0.25);

\node at (1.00, 1.00) {\Large $F$};
\node at (1.00,-1.00) {\Large $F$};

\node at (4,-1.25) {\Large $e_i$};
\node at (4.5,2.25) {\Large $\sigma_0(e_i)=\sigma_1(e_i)$}; 

\draw[thick,->] (3.5,-1) .. controls (2.5,-0.5) and (3,1) .. (3.5,1.5);

\draw[thick,->] (4.5,-1) .. controls (5.5,-0.5) and (5,1) .. (4.5,1.5);
\end{scope}

\end{tikzpicture}
    \caption{Type (2a) orbifold point on an $(F,F)$-seam of a surface, with field automorphisms different along the seam on the two sides of the point.}
    \label{fig5_27}
\end{figure}

At more general type (2) orbifold point on an $(F,K)$-seam, an embedding $\iota_0: F\xhookrightarrow{} K$ may change to a different embedding $\iota_1: F \xhookrightarrow{} K$, see Figure~\ref{fig5_28} left.  At a type (3) orbifold point, at a vertex of the network, the embedding $\jmath: F_1\xhookrightarrow{} F_3$ corresponding to the Northwest seam may be different from the composition of embeddings $\jmath_0\circ \iota_0$ for the South seam  $\iota_0:F_0\xhookrightarrow{} F_1$ and the Northeast seam $\jmath_0: F_2 \xhookrightarrow{} F_3$, see Figure~\ref{fig5_28} right.

\begin{figure}
    \centering
\begin{tikzpicture}[scale=0.6]

\begin{scope}[shift={(0,2)}]
\draw[thick] (0,0) -- (8,0);
\draw[thick] (2,0) -- (2,0.5);
\draw[thick] (6,0) -- (6,0.5);

\node at (-0.75,0) {\Large $\iota_0$};
\node at (8.75,0) {\Large $\iota_1$};

\draw[thick,fill] (4.25,0) arc (0:360:0.25);

\node at (4, 1) {\Large $K$};
\node at (4,-1) {\Large $F$};

\node at (-1,2) {\Large $(2)$};
\end{scope}

\begin{scope}[shift={(19,0)}]

\node at (-5.0,4.5) {\Large $(3)$};

\draw[thick] (0,0) -- (0,2);
\draw[thick] (0,2) -- (2,4);
\draw[thick] (0,2) -- (-2,4);

\draw[thick] (0,1) -- (0.6,1);

\node at (-1.75,1.5) {\Large $F_1$};
\node at (2.00,1.5) {\Large $F_2$};
\node at (0,4.5) {\Large $F_3$};

\node at (0,-0.75) {\Large $\iota_0$};
\node at (2.5,4.5) {\Large $\jmath_0$};
\node at (-2.5,4.5) {\Large $\jmath$};

\draw[thick] (-1,3) -- (-0.5,3.5);
\draw[thick] (1.0,3) -- (0.5,3.5);

\end{scope}
\end{tikzpicture}

    \caption{Left: type (2) orbifold point, with different embeddings $\iota_0,\iota_1: F \xhookrightarrow{} K$ on the two sides of the seam. Right: type (3) orbifold point, with $\jmath \not= \jmath_0 \iota_0$.}
    \label{fig5_28}
\end{figure}
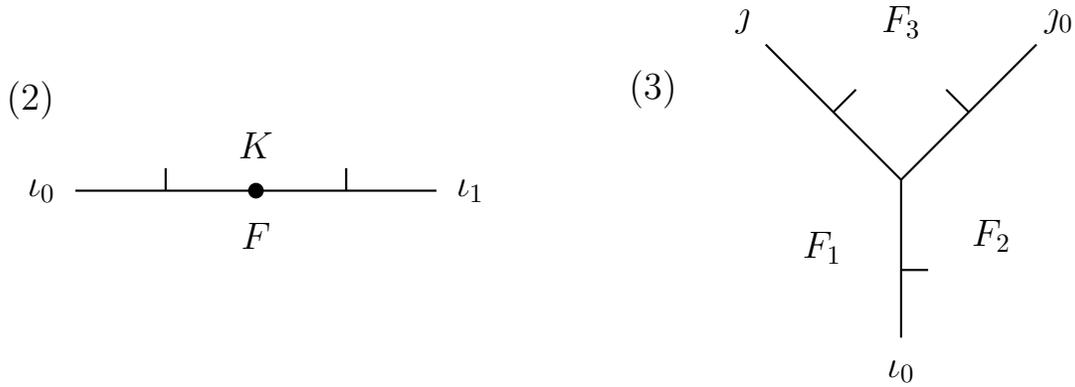


%
%
%

\section{Foams, Galois extensions, and Sylvester sums}
\label{sec-sylvester}

\subsection{Base change for \texorpdfstring{$\GL(N)$}{GLN} foams and field extensions} 
\label{subsec_base_change}

Consider the ring of polynomials 
\begin{equation}
    R'\ = \ \Z[\alpha_1, \dots, \alpha_N] 
\end{equation}
in variables $\alpha_1, \dots, \alpha_N$, and its 
 subring of symmetric polynomials 
\begin{eqnarray*}
    R &  = &  \Z[\alpha_1,\dots, \alpha_N]^{S_N} \subset R', \qquad R =\Z[E_1,\dots, E_N], \\
    E_k & = & \sum_{i_1<\ldots <i_k} \alpha_{i_1}\dots\alpha_{i_k}, 
\end{eqnarray*}
where $E_k$ is the $k$-th elementary  symmetric function in $\alpha_1,\dots, \alpha_N$. 

\vspace{0.1in}

Most constructions of equivariant $\GL(N)$ link homology, as an intermediate step, associate a free graded $R'$-module $\brak{\Gamma}$ to a planar trivalent graph $\Gamma$ with oriented edges labeled by weights in $\{1,2,\dots,N\}$ subject to the flow constraint that the sum of  weights of out edges equals the sum of weights of in edges at each vertex of $\Gamma$, see Figure~\ref{fig5_29}. There is an extensive literature on $\GL(N)$ homology. We refer to~\cite{KK} for a partial list of references and to~\cite{RW2} for a combinatorial way to define $\brak{\Gamma}$. 

\vspace{0.1in}

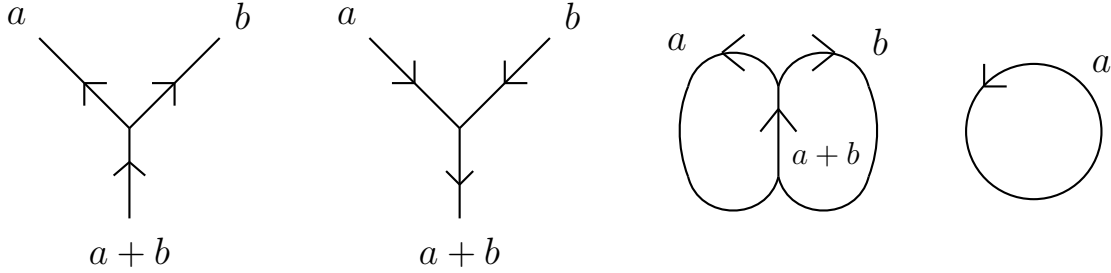
\begin{figure}
    \centering
\begin{tikzpicture}[scale=0.6,decoration={
    markings,
    mark=at position 0.60 with {\arrow{>}}}]

\begin{scope}[shift={(0,0)}]

\draw[thick,postaction={decorate}] (0,0) -- (0,2);
\draw[thick,postaction={decorate}] (0,2) -- (-1.5,4);
\draw[thick,postaction={decorate}] (0,2) -- (1.5,4);

\node at (0,-.75) {\Large $a+b$};

\node at (-2.0,4.5) {\Large $a$};
\node at ( 2.0,4.5) {\Large $b$};
\end{scope}

\begin{scope}[shift={(7.5,0)}]

\draw[thick,postaction={decorate}] ( 0,2) -- (0,0);
\draw[thick,postaction={decorate}] (-1.5,4) -- (0,2);
\draw[thick,postaction={decorate}] ( 1.5,4) -- (0,2);

\node at (0,-.75) {\Large $a+b$};
\node at (-2.0,4.5) {\Large $a$};
\node at (2.0,4.5) {\Large $b$};

\end{scope}

\begin{scope}[shift={(15.0,1)}]

\draw[thick,postaction={decorate}] (0,0) -- (0,2);

\draw[thick] (0,0) .. controls (0.25,-1) and (1.75,-1) .. (2,0);
\draw[thick,postaction={decorate}] (0,2) .. controls (0.25,3) and (1.75,3) .. (2,2);
\draw[thick] (2,0) .. controls (2.25,.65) and (2.25,1.4) .. (2,2);

\draw[thick] (0,0) .. controls (-0.25,-1) and (-1.75,-1) .. (-2,0);
\draw[thick,postaction={decorate}] (0,2) .. controls (-0.25,3) and (-1.75,3) .. (-2,2);
\draw[thick] (-2,0) .. controls (-2.25,.65) and (-2.25,1.4) .. (-2,2);

\node at (0.90,.5) {$a+b$};

\node at (2.25,3) {\Large $b$};

\node at (-2.25,3) {\Large $a$};

\end{scope}

\begin{scope}[shift={(21.5,2)}]

\draw[thick,postaction={decorate}] (1.5,0) arc (0:360:1.5);

\node at (1.5,1.5) {\Large $a$};

\end{scope}

\end{tikzpicture} 

    \caption{Left: split and merge vertices of an MOY graph. Right: simplest MOY graphs, the $(a,b)$-theta graph and thickness $a$ circle $\Gamma_a$.}
    \label{fig5_29}
\end{figure}

A planar graph as above is called a Murakami--Ohtsuki--Yamada (MOY) graph or a \emph{web}. 
The graded rank of $\brak{\Gamma}$ equals the quantum $\mathfrak{gl}(N)$ invariant $P(\Gamma)\in \Z_{+}[q,q^{-1}]$, also known as the Murakami--Ohtsuki--Yamada (MOY) invariant, i.e., see \cite{MOY}. Here $\Z_+:=\{0,1,2,\dots\}$. 

This invariant extend to a link invariant~\cite{MOY}, called the MOY invariant, taking values in $\Z[q,q^{-1}]$. It additionally depends on the labels of the link's components, which are in the range $\{1,\ldots, N\}$. The invariant extends to links by replacing each  crossing in a link's diagram by a suitable linear combination of MOY graphs. 
This link invariant is a special case of the Reshetikhin--Turaev link invariants constructed from quantum deformations of universal enveloping algebras of simple Lie algebras. 

Upon categorification, $P(\Gamma)$ is replaced by a free graded $R$-module $\brak{\Gamma}$ of graded rank $P(\Gamma)$. 
One can refer to $\brak{\Gamma}$ as the \emph{homology} or \emph{state space} of $\Gamma$.

\vspace{0.1in} 

The homology groups of a link are obtained as a complex built out of state spaces $\brak{\Gamma}$ for various MOY graphs $\Gamma$ given by taking a planar projection $D$ of a link  and substituting certain elementary subgraphs in place of crossings of $D$. 

\vspace{0.1in} 

For the empty web $\emptyset$ the associated module is $R$, $\brak{\emptyset}\cong R$, and the MOY invariant is $P(\emptyset)=1$. 

\vspace{0.1in}

Denote by $\Gamma_a$ the MOY graph which is a circle labeled $a$, $1\le a \le N$, see Figure~\ref{fig5_29} right. Then $\brak{\Gamma_a}$ can be canonically identified with the subring 
\begin{equation}
    R_{a,N-a} \ = \ \Z[\alpha_1,\dots, \alpha_N]^{S_a\times S_{N-a}}.
\end{equation}
Here, $S_a\times S_{N-a}\subset S_N$ is the parabolic subgroup for the decomposition $(a,N-a)$, separately permuting the first $a$ variables and the last $N-a$ variables.

\vspace{0.1in}

The web $\Gamma_N$, a circle of thickness $N$, has the state space isomorphic to $R$, so that $\brak{\Gamma_N}\cong \brak{\emptyset}$. In general, with a minimal amount of effort and little loss of information (there are subtleties, but these will not play any role for us), lines labeled $N$ can be hidden (erased) from MOY diagrams. This corresponds to passing from $\GL(N)$ to $\SL(N)$ link homology. However, it is often convenient to keep these lines.

\vspace{0.1in}

When $a=1$, we can also identify 
\begin{equation}
    \ang{\Gamma_1} \ \cong \ R[X]/(X^N-E_1X^{N-1} + \ldots +(-1)^N E_N). 
\end{equation}

More generally, choose a sequence 
$\unda=(a_1, \ldots, a_k)$ of positive integers that add up to $N$, with $a_1 + \ldots + a_k = N$ and $k\geq 1$, and consider the graph $\Gamma_{\unda}$ in Figure~\ref{fig5_30} left.

\vspace{0.1in}

\begin{figure}[ht]
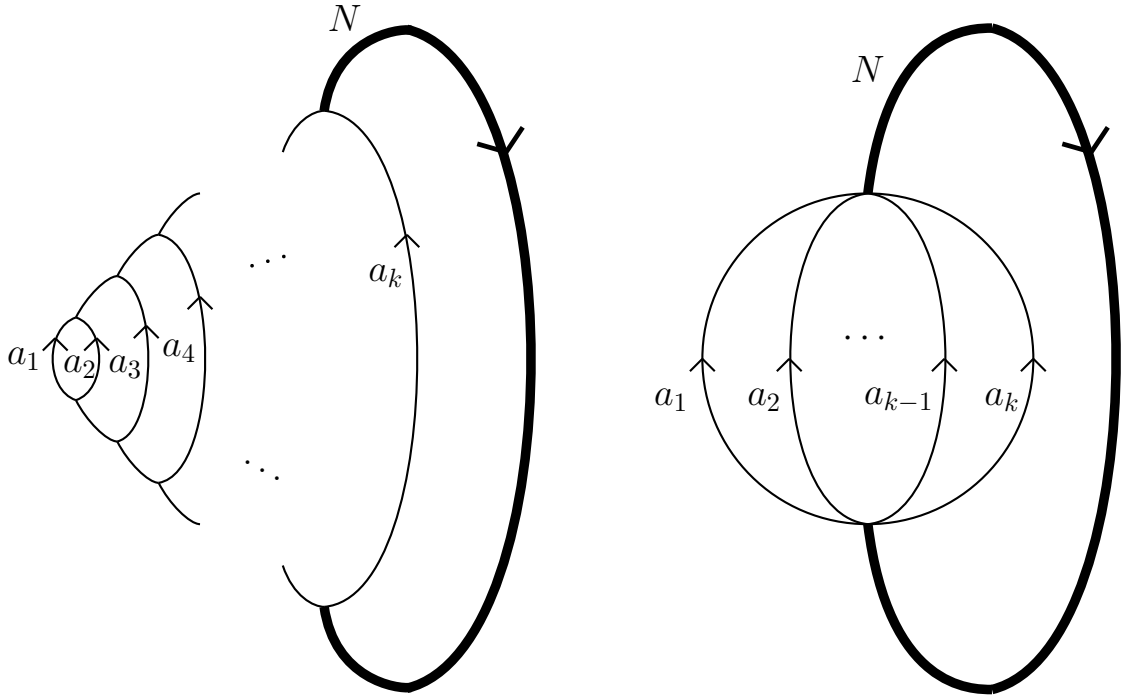

    \centering

    \caption{Left: $\GL(N)$ web $\Gamma_{\unda}$. Lines of thickness $a_1, \dots, a_k$ merge into a line of thickness $N$. Right: a schematic way to depict this web, with $k$ lines merging at once into the $N$-line.  Changing the order of merges of lines results in webs with canonically isomorphic state spaces.}
    \label{fig5_30}
\end{figure}

In this web lines of weight $a_1,a_2,\dots, a_k$ merge into thicker and thicker lines, eventually merging into a line of thickness $N$ that goes around and then splits off into the original lines. The state space $\brak{\Gamma_{\unda}}$ does not depend on the order in which the $k$ lines merge and the graph can be denoted as in Figure~\ref{fig5_30} right, where the order of merge is not specified. 

The value of the quantum MOY invariant on the graph $\Gamma_{\unda}$ is the $q$-multinomial coefficient 
\begin{equation}
     P(\Gamma_{\unda}) \ = \ 
     \left[ \begin{matrix}
     N \\ a_1,\dots, a_k \end{matrix} \right]_q \ := \ \dfrac{[N]!}{[a_1]!\dots [a_k]!},
\end{equation}
where 
\[
[m]!\ :=\ [m][m-1]\dots [1], \quad [m]=\frac{q^m-q^{-m}}{q-q^{-1}}.
\]
The state space $\brak{\Gamma_{\unda}}$ is a free $R$-module of graded rank $P(\Gamma_{\unda})$. 

\vspace{0.1in} 

When doing quantum $\SL(N)$ homology or $\SL(N)$ MOY invariants, lines of thickness $N$ may be erased and lines of thickness $N-a$ converted to those of thickness $a$ with the opposite orientation. This procedure does not change the value of the MOY invariant, and can be made to preserve homology groups. In this case,  Figure~\ref{fig5_30} graphs may be reduced by erasing thickness $N$ interval and sometimes further simplifying, see Figure~\ref{fig5_31}.  

\vspace{0.1in}

\begin{figure}
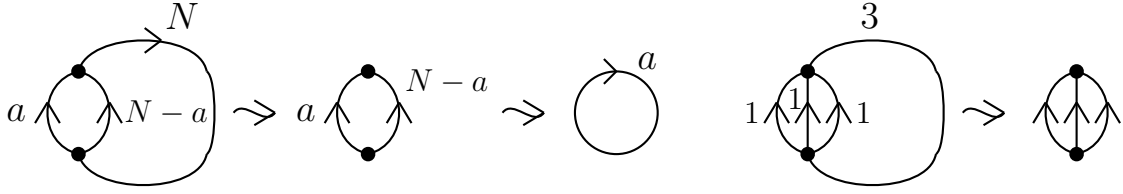

    \centering


    \caption{Left: graph $\Gamma_{(a,N-a)}$ turns into $\Gamma_a$ upon reducing from $\GL(N)$ to $\SL(N)$ homology by erasing the interval of thickness $N$. Right: Reducing $\GL(3)$ graph $\Gamma_{(1,1,1)}$ to the corresponding $\SL(3)$ graph, with all edges of the latter labeled $1$. In the $\SL(3)$ case, MOY graphs are equivalent to Kuperberg's $A_2$ spiders, see \cite{Ku}. }
    \label{fig5_31}
\end{figure}

Let $G=\GL(N,\C)$ or its maximal compact subgroup $G=\U(N)$, with the standard action on $\CC^N$. Consider the induced action of $G$ on the (partial) flag variety  \begin{equation}
\Fl(\underline{a}) \ := \ \{ 0\subset L_1\subset L_2\subset \ldots \subset  L_k\cong\CC^N \  |\  \dim(L_i)-\dim(L_{i-1})=a_i, \  i=1,\dots, k\},
\end{equation}
where $L_0 = 0$. 
The equivariant cohomology $\mthH_G(\Fl(\underline{a}))$ is naturally a module over the equivariant cohomology of a point $\mthH_G(\ast)\cong R$. There is a natural isomorphism of $R$-algebras
\begin{equation}
    \brak{\Gamma_{\unda}} \ \cong \  \mthH_G(\Fl(\underline{a})) \ \cong \ R_{\unda},
\end{equation}
where
\begin{equation}
    R_{\unda} \ := \ \Z[x_1,\dots, x_N]^{S_{\unda}}, \qquad  
    S_{\unda}\  := \ S_{a_1}\times \dots \times S_{a_k}\subset S_N,
\end{equation}
is the subring of invariants for the parabolic subgroup $S_{\unda}$ of $S_N$ acting on the ring of polynomials in $N$ variables. 

\vspace{0.1in} 

In the special case $\unda=(1,\dots, 1)=(1^N)$, the parabolic subgroup is trivial and 
\begin{equation}
    R_{(1^N)}= \Z[x_1,\dots, x_N] \cong \mthH_G(\Fl((1^N)))
\end{equation}
is isomorphic to the polynomial ring and to the equivariant cohomology of the full flag variety $\Fl((1^N))$, which we can also denote $\FF(N)$. 

\vspace{0.1in} 

The state spaces $\brak{\Gamma}$ are functorial, in a suitable sense. A graph cobordism $F$, which is a decorated combinatorial two-dimensional $CW$-complex with prescribed singularities~\cite{RW1} and embedded in $\R^2\times [0,1]$, also called a \emph{foam} or \emph{$\GL(N)$-foam}, induces a homomorphism of state spaces
\begin{equation} 
   \brak{F}\ : \ \brak{\partial_0 F}\lra \brak{\partial_1 F}.
\end{equation} 
Together, these homomorphisms form a functor from the category of $\GL(N)$-foams to the category of graded $R$-modules. 

\vspace{0.1in}

Suppose that a web $\Gamma$ admits a reflection symmetry about an axis $\ell$. Then $\brak{\Gamma}$ is naturally a unital associative Frobenius $R$-algebra, due to the presence of unit $\iota$, counit $\varepsilon$ and multiplication $m$ cobordisms as schematically shown in Figure~\ref{fig6_2} for the so-called $\Theta$-web, resembling the letter $\Theta$ (orientations and weights of edges are omitted for simplicity). The cobordisms $\iota,\varepsilon$ match halves of $\brak{\Gamma}$  by rotating one half into the other. The cobordism $m$ matches two halves of $\Gamma$ in $\Gamma\sqcup \Gamma$, leaving a single $\Gamma$ as the other boundary of the cobordism.  

\vspace{0.1in}

\begin{figure}
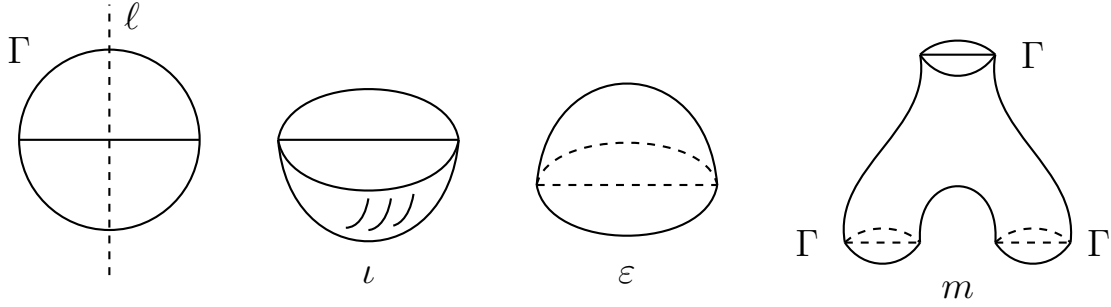

    \centering


\caption{A web $\Gamma$ with a symmetry axis $\ell$ and schematically depicted unit, counit, and multiplication morphisms. }
    \label{fig6_2}
\end{figure}

A homomorphism of commutative rings 
\begin{equation}
    \phi \ : \ R \lra S,
\end{equation}
where $S$ is not necessarily graded, 
can be used to define a version $\brak{\Gamma}_S$ of state spaces, a kind of base change from $R$ to $S$, such that $\brak{\Gamma}_S$ is a free  $S$-module of rank   $P(\Gamma)_{q=1}$ for any MOY graph $\Gamma$. Here $P(\Gamma)_{q=1}$ is the specialization of the Laurent polynomial $P(\Gamma)\in \Z_{+}[q,q^{-1}]$ to its value at $q=1$.

\vspace{0.1in} 

Due to all modules being free, one way to define it is by 
\begin{equation}
\brak{\Gamma}_S \ := \ \brak{\Gamma}\otimes_R S.  
\end{equation} 
A more intrinsic way to define $\brak{\Gamma}_S$ is via $S$-valued closed foam evaluation that uses $\phi$, see~\cite[Section~4]{KR1}    for a similar definition in a different case where the state spaces are not known to be free modules over the ground ring. 

\vspace{0.1in} 

Consider now a special case when the ground ring $S=\kk$ is a field and we pick a separable polynomial 
\begin{equation}
\label{eqn_separable_poly}
    f(x)=x^N+ u_{N-1}x^{N-1}+\ldots + u_0, \ \ u_i\in \kk, \ \  i=1,\dots, N-1, 
\end{equation}
irreducible over $\kk$. Let $K$ be a splitting field of $f(x)$ over $\kk$ and $F$ be the field 
\begin{equation}
    F \ := \ \kk[\alpha]/(f(\alpha)). 
\end{equation}
The polynomial $f(x)$ has $N$ roots $\alpha_1,\dots, \alpha_N \in K$, and each of them defines a homomorphism of $\kk$-algebras $F\lra K$. 

\vspace{0.1in} 

Consider the homomorphism 
\begin{equation}
\label{eq_hom_phi}
    \phi \ : \ R\lra \kk, \  \ \phi(E_i)= (-1)^i u_i. 
\end{equation}
and state spaces $\brak{\Gamma}_{\phi}$ associated to MOY graphs via the foam construction. The state spaces of the  empty graph and the $N$-circle are isomorphic to $\kk$, 
\begin{equation}
    \brak{\emptyset}_{\phi}\cong \brak{\Gamma_N}\cong \kk.
\end{equation}
The state space of the $1$-circle is isomorphic to $F$, 
\begin{equation}\label{eq_1_circ}
    \brak{\Gamma_1}\cong F,
\end{equation}
via a homomorphism that take a one-dotted disk with boundary $\Gamma_1$ to $\alpha$ (facets of foam may be labeled by symmetric functions in the number of variables equal to the thickness of the  facet). The state space $\Gamma_1$ is a free $\kk$ module with the basis of disks with $i$ dots, $i=0,\dots, N-1$, see Figure~\ref{fig6_1}. The corresponding basis of $F$ is that of powers of $\alpha$, $\{ 1, \alpha, \dots, \alpha^{N-1} \}$. 

\vspace{0.1in}

\begin{figure}
    \centering
\begin{tikzpicture}[scale=0.6]

\begin{scope}[shift={(0,0)}]

\draw[thick] (0,0) .. controls (0.15,1) and (2.85,1) .. (3,0);
\draw[thick] (0,0) .. controls (0.15,-1) and (2.85,-1) .. (3,0);
\draw[thick] (0,0) .. controls (.25,-2.5) and (2.75,-2.5) .. (3,0);
\node at (1.5,-2.85) {\Large $1$};
\end{scope}

\begin{scope}[shift={(6,0)}]

\draw[thick] (0,0) .. controls (0.15,1) and (2.85,1) .. (3,0);
\draw[thick] (0,0) .. controls (0.15,-1) and (2.85,-1) .. (3,0);
\draw[thick] (0,0) .. controls (.25,-2.5) and (2.75,-2.5) .. (3,0);
\draw[thick,fill] (2.35,-1.15) arc (0:360:0.25);

\node at (1.65,-2.90) {\Large $\alpha$};

\node at (4.5,-1) {\Large $\ldots$};
\end{scope}

\begin{scope}[shift={(12,0)}]

\draw[thick] (0,0) .. controls (0.15,1) and (2.85,1) .. (3,0);
\draw[thick] (0,0) .. controls (0.15,-1) and (2.85,-1) .. (3,0);
\draw[thick] (0,0) .. controls (.25,-2.5) and (2.75,-2.5) .. (3,0);
\draw[thick,fill] (2.35,-1.15) arc (0:360:0.25);
\node at (3.0,-1.25) {$i$};

\node at (1.8,-2.75) {\Large $\alpha^i$};

\node at (5,-1) {\Large $\ldots$};
\end{scope}

\begin{scope}[shift={(19,0)}]

\draw[thick] (0,0) .. controls (0.15,1) and (2.85,1) .. (3,0);
\draw[thick] (0,0) .. controls (0.15,-1) and (2.85,-1) .. (3,0);
\draw[thick] (0,0) .. controls (.25,-2.5) and (2.75,-2.5) .. (3,0);
\draw[thick,fill] (2.35,-1.15) arc (0:360:0.25);
\node at (3.9,-1.25) {$N-1$};

\node at (2,-2.75) {\Large $\alpha^{N-1}$};
\end{scope}

\end{tikzpicture}

    \caption{Basis of powers of a dot  (powers of $\alpha$) in $\brak{\Gamma_1}$.  }
    \label{fig6_1}
\end{figure}

 There is a surjective homomorphism 
 \begin{equation}
    \phi_1 \ : \  R_{(1^N)} \lra K , \ \ 
    \phi_1(x_i) = \alpha_i, \ \  
    i=1,\ldots, N, 
 \end{equation}
 into the splitting field $K$ that extends the homomorphism $\phi$, so the square below commutes
 

\begin{center}
{\Large
\begin{tikzcd}
 R_{(1^N)}\arrow{r}{\phi_1} &   K \\
 R \arrow[hookrightarrow]{u} \arrow{r}[swap]{\phi} &  \kk \arrow[hookrightarrow]{u}. \\
\end{tikzcd}
}
\end{center}

 Recall that the source ring of $\phi_1$ is the state space of $\Gamma_{(1^N)}$, which consists of $N$ weight $1$ lines that merge and split into an $N$-line.

 \begin{prop} The map $\phi_1$ induces a surjective homomorphism of $\kk$-algebras 
 \begin{equation}
     \widetilde{\phi} \ : \ \brak{\Gamma_{(1^N)}}_{\phi} \lra K.
 \end{equation}
 This map is an isomorphism if and only if the Galois group of the splitting field extension $K/\kk$ is the symmetric group $S_N$. 
 \end{prop}

\proof
The map $\phi_1$ induces a surjective homomorphism $\widetilde{\phi}$ of $\kk$-algebras since $f$ is a separable polynomial~\eqref{eqn_separable_poly} of degree $N$.
Galois groups are isomorphic to subgroups of symmetric groups, so $\widetilde{\phi}$ is an isomorphism if and only if $K$ is a splitting field of a separable polynomial.
\endproof

 Thus, $\widetilde{\phi}$ is an isomorphism if the splitting field extension $K/\kk$ has the largest possible degree $N!$ given that $\deg(f)=N$. 
 
 \vspace{0.1in} 
 
 Recall that the extension $K/\kk$ is Galois and there is a bijection between intermediate subfields of $K/\kk$ and subgroups of the Galois group $\Gal(K/\kk)$. 

\vspace{0.1in} 
 
Assuming that the Galois group is the largest possible given that $f$ has degree $N$, we can understand the state spaces of webs $\brak{\Gamma_{\unda}}_{\phi}$ for all decompositions $\unda$ via a part of the Galois correspondence. 
 
\begin{prop} \label{prop_gal_int} Suppose that $\Gal(K/\kk)=S_N$. Then 
for each decomposition $\unda$ of $N$, there is a ring isomorphism
\begin{equation}
    \brak{\Gamma_{\unda}} \cong K^{S_{\unda}}
\end{equation}
between the $\phi$-state  space of the web $\Gamma_{\unda}$ and the intermediate subfield $K^{S_{\unda}}$.
\end{prop}

\proof
The proposition follows by looking at subrings in $K$ of symmetric functions for the corresponding roots for each strand of thickness $a_1,\dots, a_k$ (symmetric functions in subsets of linear functions $x-\alpha_i$), which implies that 
$ \brak{\Gamma_{\unda}}$ is exactly the subfield of $K$ of $S_{\unda}$-invariant elements. 
\endproof

Inclusions of subfields as well as trace maps between different subfields correspond to foams that merge and split lines in these webs, corresponding to combining to consecutive parts of $\unda$ or splitting a part into two parts, 
\begin{equation*}
    (\dots, a_{i-1},a_i,a_{i+1},a_{i+2},\dots ) \leftrightarrow 
    (\dots, a_{i-1},a_i+a_{i+1},a_{i+2},\dots ). 
\end{equation*}

Thus, state spaces for theta-like webs $\Gamma_{\unda}$ correspond to subfields for the parabolic subgroups $S_{\unda}$. In this correspondence we do not encounter all intermediate subfields but only those that come from ``flattening" or ordering the set of roots of $f$ and can be matched to decompositions $S_{\unda}$, see Proposition~\ref{prop_gal_int}.  

\vspace{0.1in} 

For the partition $(1,N-1)$ the state space 
\begin{equation*}
    \brak{\Gamma_{(1,N-1)}} \ \cong F,
\end{equation*}
also see (\ref{eq_1_circ}) and Figure~\ref{fig5_31} left for  $a=1$.

\vspace{0.1in} 

We encounter the ground field $\kk$, field $F$, the splitting field $K$ as well as intermediate fields for the parabolic subgroups as $\phi$-state spaces of theta-like webs. These webs can be thought of as bubbling off an $N$-line or $N$-circle, see Figure~\ref{fig6_3}. 

\vspace{0.1in}

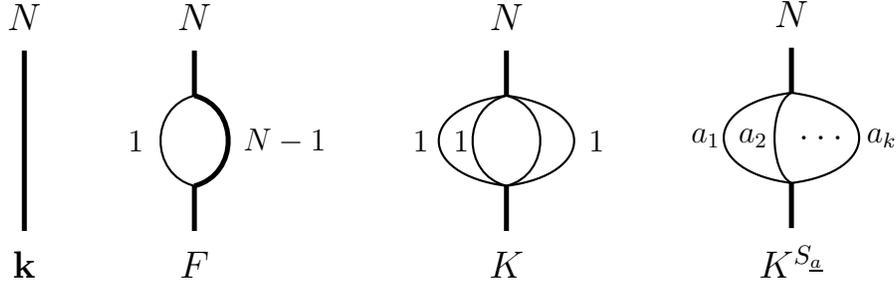
\begin{figure}
    \centering
\begin{tikzpicture}[scale=0.6]

\begin{scope}[shift={(0,0)}]

\draw[line width=0.027in] (0,0) -- (0,4);
\node at (0,-.75) {\Large $\kk$};
\node at (0,4.75) {\Large $N$};
\end{scope}

\begin{scope}[shift={(5,0)}]
\draw[line width=0.027in] (0,0) -- (0,1);
\draw[line width=0.027in] (0,3) -- (0,4);

\draw[line width=0.027in] (0,1) .. controls (1,1.25) and (1,2.75) .. (0,3);
\draw[thick] (0,1) .. controls (-1,1.25) and (-1,2.75) .. (0,3);

\node at (-1.3,2) {\large $1$};
\node at (2,2) {\large $N-1$};

\node at (0,-.75) {\Large $F$};
\node at (0,4.75) {\Large $N$};
\end{scope}

\begin{scope}[shift={(13.5,0)}]
\draw[line width=0.027in] (0,0) -- (0,1);
\draw[line width=0.027in] (0,3) -- (0,4);

\node at (1.9,2) {\large $1$};

\node at (1,2) {\large $1$};

\draw[thick] (0,1) .. controls (2,1.25) and (2,2.75) .. (0,3);

\draw[thick] (0,1) .. controls (1,1.25) and (1,2.75) .. (0,3);

\node at (-1,2) {\large $1$};

\draw[thick] (0,1) .. controls (-1,1.25) and (-1,2.75) .. (0,3);

\node at (0.00,2) {\Large $\ldots$};

\node at (-1.9,2) {\large $1$};

\draw[thick] (0,1) .. controls (-2,1.25) and (-2,2.75) .. (0,3);

\node at (0,-.75) {\Large $K$};
\node at (0,4.75) {\Large $N$};
\end{scope}

\begin{scope}[shift={(21.0,0)}]

\draw[line width=0.027in] (0,0) -- (0,1);
\draw[line width=0.027in] (0,3) -- (0,4);

\node at (2,2) {\large $a_k$};

\draw[thick] (0,1) .. controls (2,1.25) and (2,2.75) .. (0,3);

\node at (.65,2) {\Large $\ldots$};


\node at (-.85,2) {\large $a_2$};

\draw[thick] (0,1) .. controls (-.5,1.25) and (-.5,2.75) .. (0,3);

\node at (-1.9,2) {\large $a_1$};

\draw[thick] (0,1) .. controls (-2,1.25) and (-2,2.75) .. (0,3);

\node at (0,-.75) {\Large $K^{S_{\underline{a}}}$};
\node at (0,4.75) {\Large $N$};
\end{scope}

\end{tikzpicture}
    
    \caption{Basic webs along an $N$-line with field extensions of $\kk$ as state spaces. The top label is the thickness of the edge, and the bottom label is the corresponding field. }
    \label{fig6_3}
\end{figure}

When $[K:\kk]<N!$, the Galois group is a proper subgroup of $S_N$. For each permutation $s$ of $N$ roots of $f(x)$ in $K$ there is a surjective map 
\begin{equation}
   \phi_s \ : \  \kk[x_1,\dots, x_N] \lra K, \qquad  \phi_s(x_i) = x_{s(i)}, \ \  i=1,\dots, N
\end{equation}
that lifts homomorphism 
\begin{equation}
\phi_S: R\otimes_{\Z} \kk \cong \kk[E_1,\dots,E_N] \lra \kk, \qquad  \phi_S(E_i\otimes 1)= (-1)^i u_i,
\end{equation}
(similar to homomorphism (\ref{eq_hom_phi})). Map $\phi_s$ factors through a homomorphism
\begin{eqnarray*}
    \psi_s  & : & \brak{\Gamma_{(1^N)}}_{\phi} \lra K, \\
    \phi_s  &  : & \kk[x_1,\dots, x_N] \stackrel{\gamma}{\lra} \brak{\Gamma_{(1^N)}}_{\phi} \stackrel{\psi_s}{\lra} K,
\end{eqnarray*}
where $\gamma$ is the canonical quotient map, sending $x_i$ to the unit element cobordism into $\Gamma_{(1^N)}$ decorated by a dot on the $i$-th thin disk. 

\vspace{0.1in} 

For two $s$ that differ by an element of $\Gal(K/\kk)$ the two homomorphisms are related by an automorphism of $K$. Choose representatives $s_1, \dots, s_m$ of left cosets of $\Gal(K/\kk)$ acting on roots $\alpha_1,\dots, \alpha_N\in K$ of $f(x)$. Here $m=N!/[K:\kk]$ is also the index of $\Gal(K/\kk)$ as a subgroup of $S_N$ of all permutations of roots of $f(x)$.  
Each of these representatives determines a surjective 
homomorphism 
\begin{equation}
    \psi_{s_i}  \ : \ \brak{\Gamma_{(1^N)}}_{\phi} \lra K, \qquad  i=1,\dots, m.
\end{equation}

Note that $\brak{\Gamma_{(1^N)}}_{\phi}$ is a commutative $\kk$-algebra of dimension $N!$ and a quotient of $F\otimes_{\kk}F\otimes \dots \otimes F=F^{\otimes N}$. Consequently, it is a commutative semisimple $\kk$-algebra (since $\kk\subset F$ is a separable extension) and necessarily a direct product of field extensions of $\kk$. The product of homomorphisms
\begin{equation}\label{eq_gamma_iso}
\begin{tikzcd}
 \brak{\Gamma_{(1^N)}}_{\phi}  \arrow{r}{(\psi_i)_{i=1}^m} &  \displaystyle{\prod_{i=1}^m} K 
\end{tikzcd}
\end{equation}
is easily seen to be surjective and then necessarily an isomorphism.

\begin{prop}
 There is an isomorphism of algebras 
 \begin{equation}
     \brak{\Gamma_{(1^N)}}_{\phi} \ \cong \ K^{\times m}, \qquad  m = N!/[K:\kk], 
 \end{equation}
 given by (\ref{eq_gamma_iso}),
 between $\phi$-state space of $(1^N)$ theta web and the  direct product of $m$ copies of $K$, where $m$ is the index of the Galois group $\Gal(K/\kk)$ in $S_N$. 
\end{prop}

It is a reasonable question whether the above observations can be developed into something of interest to number theory or algebraic geometry, with the caveat that the Galois correspondence, that we see above in connection with webs and foams, is about 200 years old. One can ask whether it make sense to assign a commutative ring $A$ to a line and \'etale extensions $B$ of $A$ to webs $\Gamma$ that ``bubble off" that line, additionally admitting a symmetry axis, so that the state space of web $\Gamma$ is a ring $B$. Can \'etale cohomology be then connected to some version of foam theory?  

\vspace{0.1in} 

The universal extension $R\subset \Z[x_1,\dots, x_N]$ is used to build equivariant link homology. Specializing $N=2$ results in Khovanov homology. Further specializing to a separable degree two characteristic zero field extension $\kk\subset F$ results in Lee homology, i.e., see \cite{Le,KR2}. Lee homology groups depend on linking numbers only, but looking at the degeneration from the universal extension to a field extension allows to pull out the Rasmussen invariant~\cite{Ra} of knot concordance and its variations. This pattern extends to $N>2$, see \cite{Go, Wu, Lo, Lw}. Specializing to separable field extensions results in near-trivial theories, from the topological viewpoint, but the way the universal theory degenerates into those leads to a wealth of information about concordance of knots and links. One can wonder whether more advanced structures in Galois theory and number theory may admit such liftings or deformations relating them to non-trivial low-dimensional topology.

 
\subsection{Overlapping foams and Sylvester double sums} 
\label{subsec_over_sylv}

A straightforward extension of the Robert--Wagner evaluation formula to overlapping 
foams was proposed in~\cite[Section 3]{Kh4}. It allows to interpret the Sergeev--Pragacz formula for the supersymmetric Schur functions (hook Schur functions)~\cite{MJ1}, \cite[Chapter 4]{Mo} and the Day formula for Toeplitz determinants of rational functions  via overlapping foams, see~\cite{Kh4, Da, HJ}. The same paper also suggested a relation between overlapping foam evaluation and resultants and speculated on possible relevance of overlapping foams to categorification of quantum groups.

\vspace{0.1in} 

In this section we explain how to interpret Sylvester double sums and relations on them (the Exchange Lemma) as developed in~\cite{KSV} and earlier work (see references in~\cite{KSV}) via overlapping foams as well. We assume familiarity with Section 3 of~\cite{Kh4}, which we briefly summarize below.

\vspace{0.1in} 

A closed $\GL(N)$ foam $F$ is a decorated combinatorial two-dimensional CW-complex embedded in $\R^3$. It consists of oriented \emph{facets} (connected surfaces) each carrying  a \emph{thickness} from $1$ to $N$. Facets are joined along \emph{seams} where facets of thickness $a$ and $b$ merge into a facet of thickness $a+b$, subject to compatibility of orientations. A foam may contain \emph{vertices}, which are singular points that connect pairs of seams between two different ways of merging three facets of thicknesses $a,b,c$ into a facet of thickness $a+b+c$. A facet of thickness $k$ of a foam may contain dots labeled by symmetric functions in $k$ variables. A foam $F$ evaluates to $\brak{F}$ which is a symmetric polynomial in $N$ variables. We refer to~\cite{RW2,KK} and references in~\cite{KK} for details.

\vspace{0.1in} 

It is useful to label the set of variables by $X$ with $|X|=N$ and view $\brak{F}$ as a symmetric function in these variables, denoting the corresponding ring  of symmetric functions by $\Sym(X)$. 

\vspace{0.1in} 

When $F$ is a connected surface (a single facet) of maximal thickness $N$, with a dot on it labeled by $f(X)\in\Sym(X)$, the evaluation $\brak{F}=f(X)$ does not depend on  the genus of the  surface, see Figure~\ref{fig6_7}.  Of course, for most other foams, including surfaces of thickness less that $N=|X|$, the evaluation will strongly depend on the genera of components of the foam. 

\vspace{0.1in}

\begin{figure}
    \centering
\begin{tikzpicture}[scale=0.7]


\draw[line width=0.05cm] (3,1) arc (0:360:1.5);

\draw[line width=0.05cm] (0.8,1.15) .. controls (1,.4) and (2,.4) .. (2.2,1.15);

\draw[line width=0.05cm] (.9,.9) .. controls (1.1,1.4) and (1.9,1.4) .. (2.1,.9);

\draw[thick,fill] (2.4,1.75) arc (0:360:0.25);

\node at (1.60,1.75) {\large $f$};

\node at (-0.25,2.5) {\large $X$};

\node at (1,-1.25) {\Large $|X|$};

\node at (4.5,1) {\Large $=$};

\begin{scope}[shift={(1,0)}]

\draw[line width=0.05cm] (8,1) arc (0:360:1.5);

\draw[line width=0.05cm,dashed] (5,1) .. controls (5.25,1.4) and (7.75,1.4) .. (8,1);
\draw[line width=0.05cm,dashed] (5,1) .. controls (5.25,0.6) and (7.75,0.6) .. (8,1);

\draw[thick,fill] (7.45,1.75) arc (0:360:0.25);

\node at (6.6,1.8) {\large $f$};

\node at (9.5,1) {\Large $=$};

\node at (11.5,1) {\Large $f(X)$};

\node at (8.25,2.5) {\large $X$};

\node at (5.75,-1.25) {\Large $|X|$};

\end{scope}
   
\end{tikzpicture}
    \caption{When $F$ is a  single facet foam, of maximal thickness $N=|X|$, it evaluates to the product of symmetric functions over all the dots on $F$. In particular, the evaluation does not depend on the genus of the surface $F$. The figure shows the case of a single dot and surface $F$ having genus  $1$ or $0$. Dashed circle on the sphere is there to depict the sphere schematically (it is not a seam circle on a sphere separating it into two facets of complementary thickness). The latter seam lines appear in the next few figures.    }
    \label{fig6_7}
\end{figure}
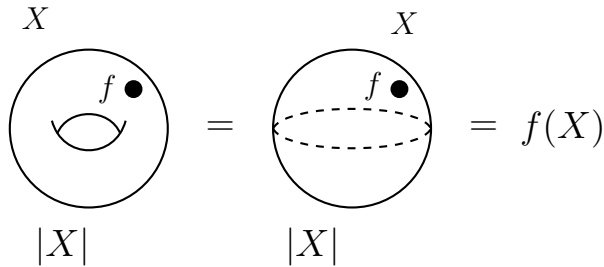

In~\cite[Section 3]{Kh4} an extension of this evaluation is proposed when several embedded foams for disjoint sets of variables overlap in $\R^3$. Foam evaluation $\brak{F}$, which is a sum of evaluations $\brak{F,c}$ over all colorings $c$, is modified by scaling $\brak{F,c}$ by  $(x_i-y_j)^{m(i,j,c)}$, where $m(i,j,c)$ is the number of circles in the intersection of the union $F_i(c)$ of facets colored $i$, $x_i \in X$ and the union $F_j(c)$ of facets colored $j$, $y_j\in Y$. The product of these terms is taken over all pairs $X,Y$ and $x_i\in X, y_j\in Y$. Ordering of each pair of sets $(X,Y)$ of foam labelings is fixed to have a well-defined term $x_i-y_j$ versus $y_j-x_i$.

\vspace{0.1in} 

For a closely related notion of an $\SL(N)$ foam and its evaluation, other seam lines are allowed as well, where oriented facets of thickness $a,b,c$ with $a+b+c=N$ or $a+b+c=2N$ meet along seams. Case $N=3$ and foams with $(a,b,c)=(1,1,1)$ seam lines have been treated in details in the literature, but for $N>3$ foam evaluation is mostly considered for $\GL(N)$ foams. See Section~2.3.1 in~\cite{RW1} for a brief discussion on modifying evaluation from $\GL(N)$ to $\SL(N)$ foams, with the caveat that what call $\GL(N)$ foams is referred to as $\mathfrak{sl}_N$ foams in~\cite{RW1}, and our $\SL(N)$ foams are called \emph{generalized foams} in~\cite{RW1}. 

\vspace{0.1in}

Given  finite sets of variables $Y$ and $Z$, define 
\begin{equation} 
\mcR(Y,Z) = \prod_{y\in Y, z\in Z} (y-z), 
\qquad 
\mcR(Y,Z)=1 
\quad  
\mathrm{if} \  Y=\emptyset \ \mathrm{or} \  Z=\emptyset. 
\end{equation}
Note that $\mcR(Y,Z)$ is a polynomial that is symmetric in variables in $Y$ and symmetric in `variables in $Z$, thus 
\begin{equation*}
    \mcR(Y,Z) \ \in \ \Sym(Y)\otimes \Sym(Z), 
\end{equation*}
where $\Sym(Y)$ stands for the ring of symmetric polynomials in $Y$ with coefficients in $\Z$ or in a field $\kk$, likewise for $\Sym(Z)$. Polynomial $\mcR(Y,Z)$ equals the evaluation as in~\cite[Section 3]{Kh4} of overlapping connected surfaces (foams with one facet) labeled by $Y$ and $Z$ and of maximal thickness  $|Y|$ and $|Z|$, respectively, see Figure~\ref{fig6_4}. 

\vspace{0.1in}

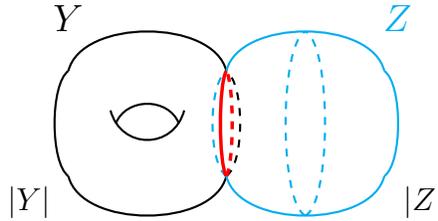
\begin{figure}
    \centering
\begin{tikzpicture}[scale=0.7]

\draw[line width=0.05cm] (0,2) .. controls (0.5,3) and (2.5,3) .. (3,2);

\draw[line width=0.05cm] (0,0) .. controls (0.5,-1) and (2.5,-1) .. (3,0);

\draw[line width=0.05cm] (0,0) .. controls (-0.35,0.5) and  (-0.35,1.5) .. (0,2);

\draw[line width=0.05cm,dashed] (3,0) .. controls (3.35,.25) and (3.35,1.75) .. (3,2);

\draw[line width=0.05cm] (0.8,1.25) .. controls (1,.5) and (2,.5) .. (2.2,1.25);

\draw[line width=0.05cm] (.9,1) .. controls (1.1,1.5) and (1.9,1.5) .. (2.1,1);

\draw[line width=0.06cm,dashed,cyan] (3,0) .. controls (2.65,.25) and (2.65,1.75) .. (3,2);

\draw[line width=0.020in,dashed,red] (3,0) .. controls (3.15,.25) and (3.15,1.75) .. (3,2);

\draw[line width=0.020in,red] (3,0) .. controls (2.85,.25) and (2.85,1.75) .. (3,2);

\draw[line width=0.06cm,cyan] (3,2) .. controls (3.5,3) and (5.5,3) .. (6,2);

\draw[line width=0.06cm,cyan] (3,0) .. controls (3.5,-1) and (5.5,-1) .. (6,0);

\draw[line width=0.06cm,cyan] (6,0) .. controls (6.35,0.5) and (6.35,1.5) .. (6,2);



\draw[line width=0.06cm,dashed,cyan] (4.5,-.75) .. controls (5,-.5) and (5,2.5) .. (4.5,2.75);

\draw[line width=0.06cm,dashed,cyan] (4.5,-.75) .. controls (4,-.5) and (4,2.5) .. (4.5,2.75);

\node at (-0.50,2.5) {\Large $Y$};
\node at ( 6.5,2.5) {\color{cyan}\Large $Z$};

\node at (-.75,-.5) {\large $|Y|$};
\node at (6.75,-.5) {\large $|Z|$};

\end{tikzpicture}
    \caption{Overlapped connected surfaces labeled $Y$ and $Z$ of maximal thickness  $|Y|$ and $|Z|$, respectively, evaluate to $\mcR(Y,Z)$. In the picture, the surfaces are a torus and  a sphere. }
    \label{fig6_4}
\end{figure}

Given sets of variables $A=\{\alpha_1,\dots, \alpha_m\}$ and $B=\{\beta_1,\dots, \beta_n\}$, the Sylvester double sum~\cite{Sy,DHKS,KSV}, for 
$0\le p\le m$ and $0\le q\le n$, is given as follows:
\begin{equation}
    \Syl_{p,q}(A,B)(x) \ := \ \sum_{\substack{A'\subset A, \, B'\subset B, \\
    |A'|=p, \, |B'|=q}} 
    \mcR(A',B')\mcR(A\setminus A',B\setminus B') 
    \frac{\mcR(x,A')\mcR(x,B')}{\mcR(A',A\setminus A')\mcR(B',B\setminus B')}.
\end{equation}
The sum is over all subsets of $A$ and $B$ of cardinality $p$ and $q$. We refer the reader to~\cite{LPr} and the above papers for applications of Sylvester double sums and their relation to subresultants. 

\vspace{0.1in}

Sylvester double sum 
is a polynomial in $x$ of degree at most $d:=p+q$. When $p=0$ or $q=0$, the expression is called a \emph{single sum}. Function $\Syl_{p,q}(A,B)(x)$ is a polynomial in $x$ with coefficients in the ring  $\Sym_{m,n}\cong \Sym(A)\otimes \Sym(B)$ which is the tensor product of rings of symmetric functions in the $m$ variables in $A$ and $n$ variables in $B$, respectively. 

\vspace{0.1in}

To interpret this sum via foam evaluation, we observe that denominator terms may come from seamed 2-spheres, since in foam evaluation their positive Euler characteristics make the corresponding products go into the denominators. These 2-spheres should have seam circles splitting the 2-spheres into pairs of discs of complementary thickness $p, m-p$ for the $A$-variables sphere and $q,n-q$ for the $B$-sphere. This would produce denominator terms $\mcR(A',A\setminus A')$  and $\mcR(B',B\setminus B')$ in the sum. 

\vspace{0.1in}

Furthermore, the 2-spheres should  intersect, to  account  for  the two other terms in the product that do not contain $x$. Finally, to  incorporate $x$, we introduce a third group of variables  $\{x\}$,  in addition to $A$ and  $B$, and  a connected surface  of thickness one for $\{x\}$ that intersects the 2-spheres labeled $A$  and  $B$ in one  circle each, to account for the terms in the product that contain $x$. These three components of the foam are shown in Figure~\ref{fig2_0}.

\vspace{0.1in}

\begin{figure}
\begin{center}
\begin{tikzpicture}[scale=0.6]

\begin{scope}[shift={(0,0)}]
\draw[line width=0.05cm] (3,0) arc (0:360:3);


\draw[line width=0.05cm] (-1.25,0) .. controls (-1,-1) and (1,-1) .. (1.25,0);
\draw[line width=0.05cm] (-1,-.45) .. controls (-.5,1) and (.5,1) .. (1,-.45);

\node at (2.6,2.6) {\Large $x$};
\node at (-2.5,-2.6) {\Large $1$};

\node at (0,-4.25) {\Large $(a)$};
\end{scope}

\begin{scope}[shift={(9,0)}]
\draw[line width=0.05cm] (3,0) arc (0:360:3);
\draw[line width=0.05cm,dashed] (0,-3) .. controls (1,-2.75) and (1,2.75) .. (0,3);
\draw[line width=0.05cm] (0,-3) .. controls (-1,-2.75) and (-1,2.75) .. (0,3);

\node at (-2.2,-3.05) {\Large $p$};
\node at (2.95,-3.05) {\Large $m-p$};

\node at (-2.2,3.05) {\Large $A'$};
\node at (2.95,3.05) {\Large $A\setminus A'$};

\node at (0,-4.25) {\Large $(b)$};
\end{scope}

\begin{scope}[shift={(19,0)}]

\draw[line width=0.05cm] (-4,0) .. controls (-3.75,2) and (3.75,2) .. (4,0);
\draw[line width=0.05cm] (-4,0) .. controls (-3.75,-2) and (3.75,-2) .. (4,0);

\draw[line width=0.05cm,dashed] (0,1.45) .. controls (1.35,1.3) and (1.35,-1.3) .. (0,-1.45);
\draw[line width=0.05cm] (0,1.45) .. controls (-1.35,1.3) and (-1.35,-1.3) .. (0,-1.45);

\node at (-2,-2.1) {\Large $B'$};
\node at (2.5,-2.2) {\Large $B\setminus B'$};

\node at (-2,1.9) {\Large $q$};
\node at (2.5,1.8) {\Large $n-q$};

\node at (0,-4.25) {\Large $(c)$};
\end{scope}

\end{tikzpicture}

\caption{Three components of the double sum foam in Figure~\ref{fig6_5}, from left to right: (a) connected surface (genus is unimportant, chosen to be one) of maximal thickness $1$ carrying variable set $X=\{x\}$, (b) seamed 2-sphere glued from disks of thickness $p$ and $m-p$, respectively, with the variable set $A$, (c) seamed 2-sphere glued from disks of thickness $q$ and $n-q$, respectively, with the variable set $B$. Colorings of (b) are in bijections with $A'\subset A$, $|A'|=p$, colorings of (c) are in bijections with $B'\subset B,$ $|B'|=q$.}
\label{fig2_0}
\end{center}
\end{figure}
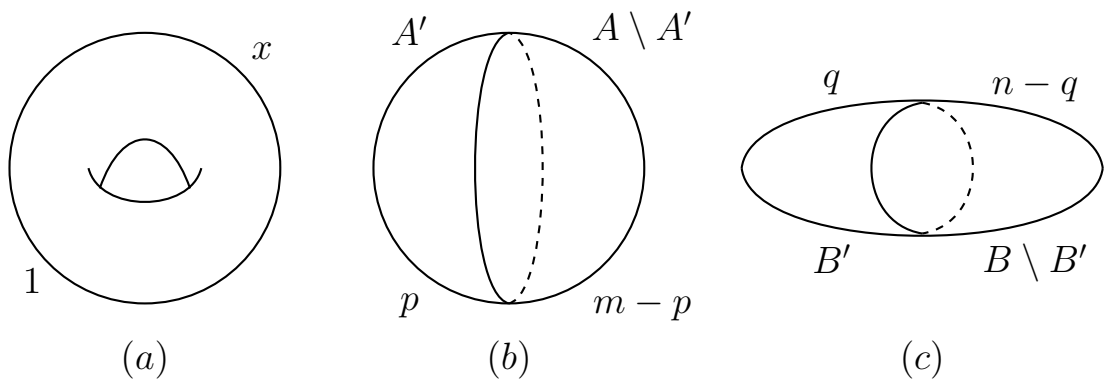

Figure~\ref{fig6_5} shows how these three foams can overlap in $\R^3$, with the resulting evaluation equal to $\Syl_{p,q}(A,B)(x)$.


\begin{figure}
    \centering
\begin{tikzpicture}[scale=0.65]

\draw[line width=0.05cm] (4,0) arc (0:360:4);

\draw[line width=0.05cm] (-14,0) .. controls (-13.8,-4) and (-2.2,-4) .. (-2,0);
\draw[line width=0.05cm] (-14,0) .. controls (-13.8,4) and (-2.2,4) .. (-2,0);

\draw[line width=0.05cm] (-6,0) .. controls (-5.75,-2) and (5.75,-2) .. (6,0);
\draw[line width=0.05cm] (-6,0) .. controls (-5.27,2) and (5.75,2) .. (6,0);

\draw[line width=0.05cm,dashed] (0,-4) .. controls (2,-3.75) and (2,3.75) .. (0,4);
\draw[line width=0.05cm] (0,-4) .. controls (-2,-3.75) and (-2,3.75) .. (0,4);

\draw[line width=0.05cm,dashed] (0,-1.45) .. controls (1,-1.2) and (1,1.2) .. (0,1.45);
\draw[line width=0.05cm] (0,-1.45) .. controls (-1,-1.2) and (-1,1.2) .. (0,1.45);

\begin{scope}[shift={(0.0,-0.125)}]
\draw[line width=0.05cm] (-9.5,0.5) .. controls (-9.25,-0.5) and (-6.75,-0.5) .. (-6.5,0.5);
\draw[line width=0.05cm] (-9,0) .. controls (-8.75,1) and (-7.25,1) .. (-7,0);
\end{scope}

\node at (-14.1,1.5) {\large $x$};
\node at (-8,2.10) {\large $1$};

\node at (-1.5,3) {\large $p$};
\node at (2.4,1.9) {$m-p$};
\node at (-1.80,0.75) {\large $q$};
\node at (2.75,0.65) {$n-q$};

\node at (-2.3,4) {\Large $A'$};
\node at (3.20,3.9) {\Large $A\setminus A'$};
\node at (6.95,-1) {\Large $B\setminus B'$};
\node at (-2,-2) {\Large $B'$};

\draw[thick,fill,brown] (-3.6, 1.1) arc (0:360:0.25);
\draw[thick,fill,brown] (-3.6,-1.1) arc (0:360:0.25);

\draw[thick,fill,magenta] (-3.15, 2.00) arc (0:360:0.25);
\draw[thick,fill,magenta] (-3.15,-2.00) arc (0:360:0.25);

\draw[thick,fill,cyan] (-2.35, 1.35) arc (0:360:0.25);
\draw[thick,fill,cyan] (-2.35,-1.35) arc (0:360:0.25);

\draw[thick,fill,blue] (4.05, 1.2) arc (0:360:0.25);
\draw[thick,fill,blue] (4.05,-1.2) arc (0:360:0.25);

\end{tikzpicture}
    \caption{Foam evaluating to  $\Syl_{p,q}(A,B)(x)$.  Four intersection circles of three components are shown schematically, as pairs of points of four different colors (blue, red, brown, cyan). The two seamed 2-spheres intersect along two circles (indicated as pairs of blue and brown points), and the third surface (shown as a torus, but its genus is irrelevant for the evaluation) intersects each  seamed 2-sphere along a circle (indicated as red and cyan pairs of points). A coloring of this foam consist of assigning a subset $A'\subset A$ of cardinality $p$ to the left disk of the $A$ sphere and a subset $B'\subset B$ of cardinality $q$ to the left disk of the $B$ sphere.}
    \label{fig6_5}
\end{figure}
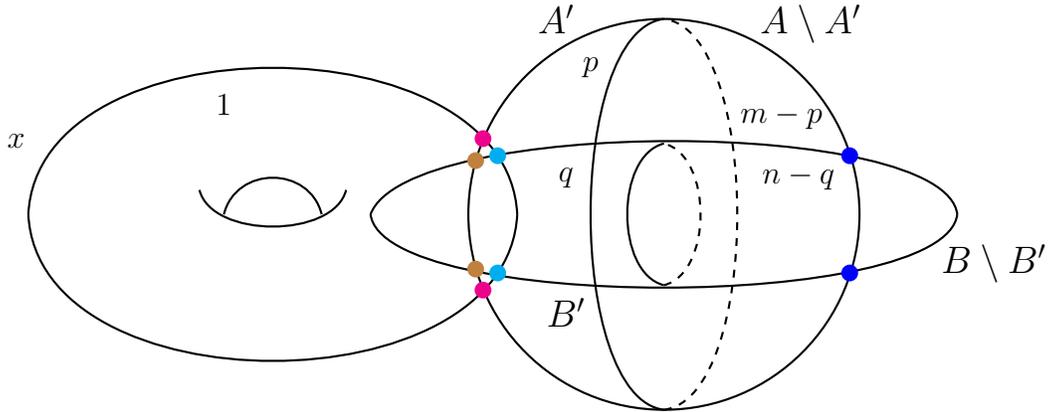

\begin{figure}
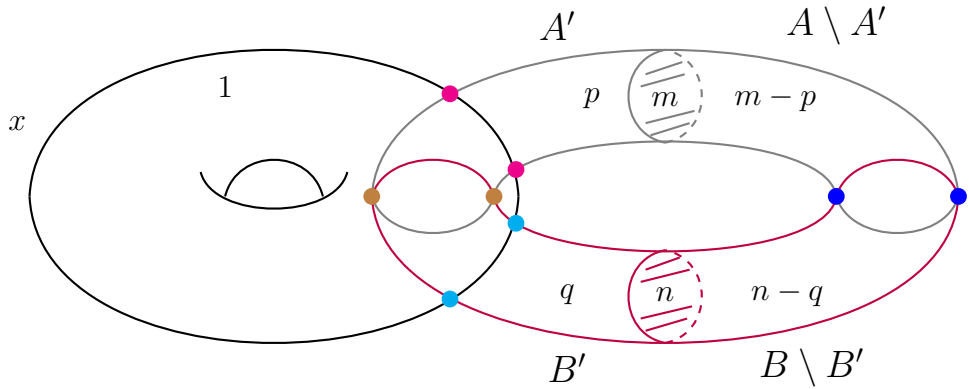

    \centering

    \caption{$\GL$ version of the foam in Figure~\ref{fig6_5}. $A$-foam (shown in grey) and $B$-foam (shown in purple) are theta-foams, with one disk of maximal thickness in each (shaded disks labeled $m$ and $n$). Intersection circles are schematically depicted by pairs identically colored points.
    }
    \label{fig6_8}
\end{figure}

\begin{remark}
In our evaluation of 2-spheres  with a seam line separating disks with complementary thickness we are tacitly considering $\SL$ evaluation. To convert to $\GL$ evaluation, 2-spheres  should be changed into theta-foams with one disk facet of maximal thickness. The relation is shown in Figure~\ref{fig6_6}. 

\vspace{0.1in} 

For each of the foam configurations in this section, it is easy to find an embedding into $\R^3$ that extends to an embedding of the corresponding $\GL$ foam, with seamed 2-spheres becoming theta-foams with the new disk facet of maximal thickness, while preserving the evaluation.  

\end{remark}

\vspace{0.1in}

\begin{figure}
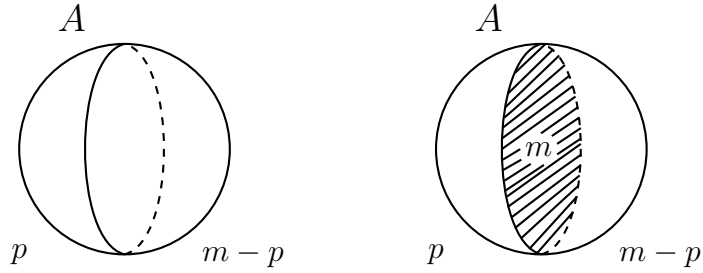

    \centering


    \caption{$\SL(m)$ vs $\GL(m)$ foams. Left: An  $\SL(m)$ foam 2-sphere made of two disks with complementary thicknesses $p$ and $m-p$. Right: A $\GL(m)$ theta-foam obtained by adding a disk of maximal thickness $m$ to the 2-sphere. There are two ways to orient the seam edge in the foams and the two evaluations differ by $(-1)^{p(m-p)}$, see~\cite{RW1,KK}.}
    \label{fig6_6}
\end{figure}

Figure~\ref{fig6_8} shows the $\GL$ version of the foam that evaluates to the Sylvester double sum.

\vspace{0.1in}

\begin{figure}
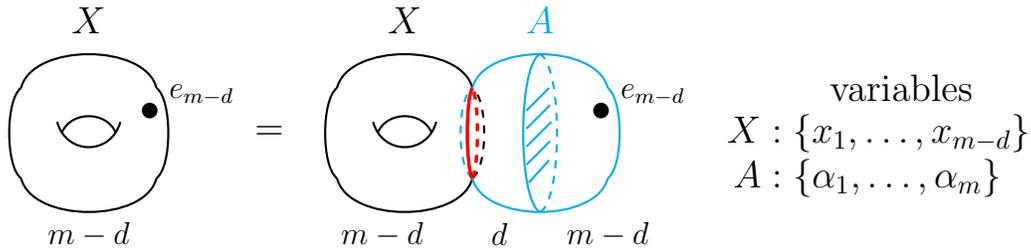

\begin{center}

\end{center}
    \caption{Foams for the relation (\ref{eq_x_alpha}). The thickness of $X$ is $m-d$. The thickness of the left portion of $A$ is $d$,  the thickness of its right portion is $m-d$. 
    The two components on the right hand side overlap along a circle.}
    \label{fig_1_1}
\end{figure}

\begin{figure}
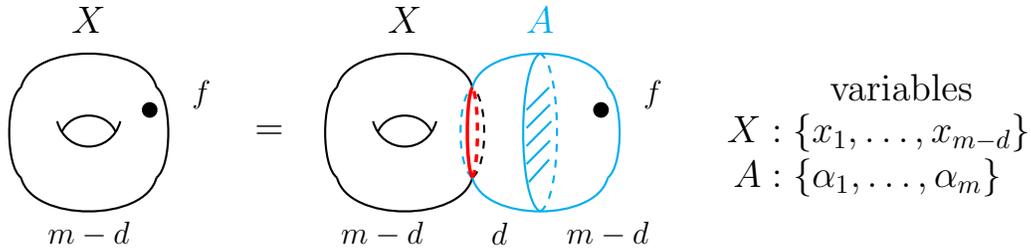

\begin{center}
%
\end{center}
\caption{Foams for the relation~\eqref{eq_x_alpha_2}, a generalization of Figure~\ref{fig_1_1} foam. Genus of  the $X$ component is unimportant.}
    \label{fig_1_1_gen}
\end{figure}


\begin{figure}
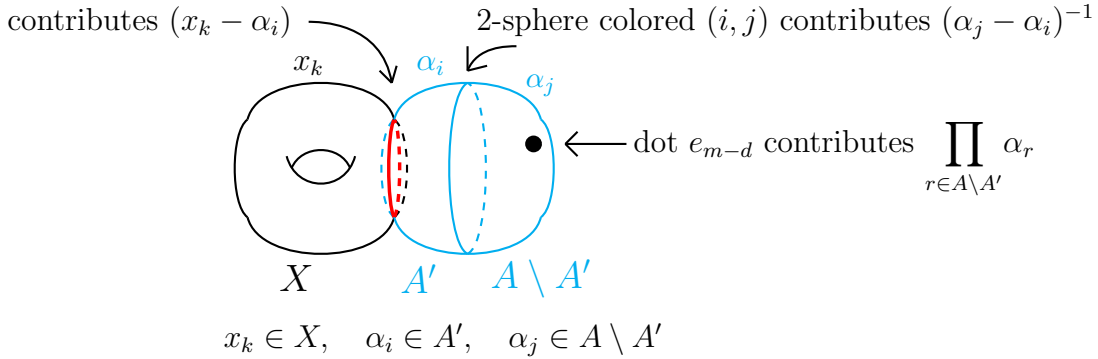

    \begin{center}
%
    \end{center}

    \caption{Foam for the right hand side of the identity (\ref{eq_x_alpha}). The 2-torus $X$ has maximal thickness $m-d=|X|$ and a unique coloring, by $X$. It contributes $1$ to the product. Left disk of 2-sphere is colored by $A'\subset A$, right disk by its complement $A'\setminus A$. The denominator term on the right hand side is the product of $\alpha_j-\alpha_i$, over $\alpha_i\in A'$ and $\alpha_j\in A\setminus A'$. The intersection circle contributes the product of $x_k-\alpha_i$, over all $k=1,\dots, m-d$ and $\alpha_i\in A'$. Dots on the left hand side and right hand side are labeled by the elementary symmetric function of the degree equal to the thickness $m-d$ of the facets and contribute $x_1\cdots x_{m-d}$, respectively product of $\alpha_j\in A\setminus A'$, to the terms. }
    \label{fig_1_2}
\end{figure}

Chen and Louck in \cite[Theorem 2.1]{CL} give a certain polynomial identity for a finite set of variables $A=\{\alpha_1,\dots, \alpha_m\}$ and set of variables $X=\{ x_1,\dots, x_{m-d}\}$. This is an identity in the ring of rational functions $\Q(\alpha_1,\dots, \alpha_m,x_1,\dots, x_{m-d})$: 
\begin{equation}\label{eq_x_alpha}
    x_1\cdots x_{m-d} = \sum_{A'\subset A, \:  |A'|=d} 
    \left( \prod_{\alpha_j\notin A'} \alpha_j \right) 
    \frac{\displaystyle{\prod_{x_j\in X, \:  \alpha_i\in A'}}(x_j-\alpha_i)}
    {\displaystyle{\prod_{\alpha_j\notin A', \: \alpha_i\in A'}} (\alpha_j-\alpha_i)}.
\end{equation}

More generally, they have  the  formula  
\begin{equation}\label{eq_x_alpha_2}
    f(X) = \sum_{A'\subset A, \: |A'|=d} 
    f(A\setminus A')  
    \frac{\displaystyle{\prod_{x_j\in X, \:  \alpha_i\in A'}}(x_j-\alpha_i)}
    {\displaystyle{\prod_{\alpha_j\notin A', \:  \alpha_i\in A'}} (\alpha_j-\alpha_i)}
\end{equation}
for a symmetric polynomial $f$ in $m-d$ variables such that the degree of $f$ in any of its variables is at most $d$. When $m=d+1$, so that $X=\{x_1\}$, their formula specializes to the classical Lagrange interpolation formula for a one-variable polynomial of degree at most $d$, see~\cite{CL}. 

\vspace{0.1in} 

Foam equivalents of formulas \eqref{eq_x_alpha} and \eqref{eq_x_alpha_2} are depicted in Figures~\ref{fig_1_1} and~\ref{fig_1_1_gen}, correspondingly. $X$ foams there have maximal thickness $m-d=|X|$ and a  surface of any genus can be chosen in place of a torus for that component. 
Figure~\ref{fig_1_2} shows in detail why the foam in the right hand side of Figure~\ref{fig_1_1} evaluates to the right hand side of formula \eqref{eq_x_alpha}. 


\vspace{0.1in} 

An important role in~\cite{KSV} and several related papers is played by the \emph{Exchange Lemma}. To state it, following~\cite{KSV}, 
take $A$ and $B$ to be disjoint sets of cardinalities $m$ and $n$, respectively. Then  

\begin{equation}\label{eq_exchange}
    \sum_{A'\subset A, \: |A'|=d} \mcR(A\setminus A',B) \frac{\mcR(X,A')}{\mcR(A\setminus A',A')}  = 
    \sum_{B'\subset B, \: |B'|=d} 
    \mcR(A,B\setminus B')\frac{\mcR(X,B')}{\mcR(B',B\setminus B')}, 
\end{equation}

\begin{figure}[ht]
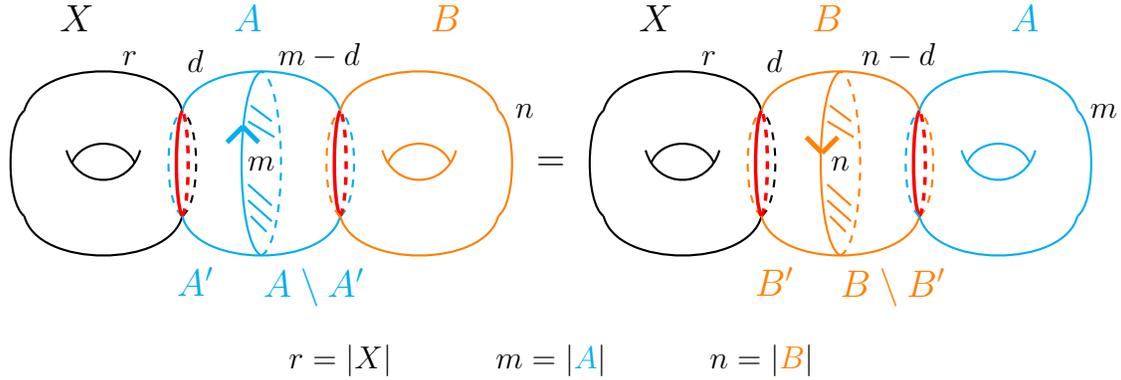

\begin{center}


\end{center}

\caption{Exchange Relation written via foam evaluation. Seam circles of blue and orange  theta-foams are oriented  oppositely, to incorporate implicit sign in formula \eqref{eq_exchange} that appears if in one of the denominators the order of a set and its complement is reversed, see~\cite{KSV}. $X$ and $B$ foams on the left hand side and $X$ and $A$ foams on the right hand side may carry any genus; we chose genus $1$ for all four. }
\label{fig_1_3}
\end{figure}

Foam interpretation of the both sides of this identity  is  shown  in Figure~\ref{fig_1_3}. 

\vspace{0.1in}

\begin{figure}
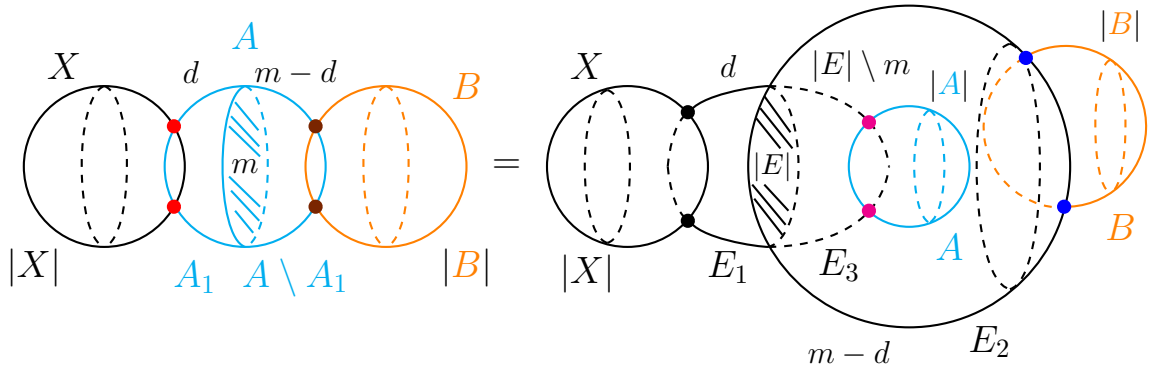

    \centering

    \caption{Foams for the formula~\eqref{DKSV-prop2-1}, also see~\cite[Proposition 2.1]{DKSV}. On the right hand side, generalized theta-foam in the middle for variable set $E$ consists of a thickness $|E|$ disk with three adjacent disks of thicknesses $d$, $m-d$ and $|E|-m$, respectively. Each of these three disks intersects one of the spheres for variable sets $A,B,X$.  }
    \label{fig6_9}
\end{figure}

As another example, consider the formula in~\cite[Proposition 2.1]{DKSV}. To state it, 
let $A,B$ be finite sets with $|A|=m$ and $|B|=n$, and choose $0\leq d\leq m$. Let $X,E$ be finite sets such that  
\[ 
|E| \geq \max \{|X|+d\, ,\:  m+n-d\, ,  \:  m  \}. 
\] 
Then 
\begin{equation}
\label{DKSV-prop2-1}
\begin{split}
&\sum_{
    \substack{
    A_1\sqcup A_2=A\\ 
    |A_1|=d, \:\:  |A_2|=m-d
    }
}
 \frac{\mcR(A_2,B)\mcR(X,A_1)}{\mcR(A_1,A_2)} = \\
= &\sum_{
    \substack{
    E_1\sqcup E_2\sqcup E_3=E \\ 
    |E_1|=d,\:\: |E_2|=m-d,\:\:  |E_3|=|E|-m
    }
}
\frac{\mcR(A,E_3)\mcR(E_2,B)\mcR(X,E_1)}{\mcR(E_1,E_2)\mcR(E_1,E_3)\mcR(E_2,E_3)}.
\end{split}
\end{equation}

Foam interpretation of this formula is shown in Figure~\ref{fig6_9}. In the evaluation of overlapping foams we assume that variable sets (in this  example, $A,B,X,E$) are disjoint, but perhaps this condition can be relaxed (formula (\ref{DKSV-prop2-1}) holds as well when  these sets have non-empty intersections, see~\cite{DKSV}). 

\vspace{0.1in} 

We leave it to the reader to write a similar foam interpretation of the relation between Sylvester double and single sums, see formula (2) in~\cite{DKSV}. 

\vspace{0.1in} 

From the present examples and those in~\cite[Section 3]{Kh4} one can make a natural guess that there exists a meaningful theory of overlapping foams, but it is not clear to the authors how to develop it. One possible direction is to use an extension of Sylvester's subresultants to polynomials with multiple roots constructed in~\cite{DKS,DKSV} to search for the symmetric analogue of the Robert--Wagner foam evaluation~\cite{RW1}.  Robert--Wagner work and many prior papers (see~\cite{KK} for an  incomplete survey) deal with \emph{exterior foams} that are used to categorify networks on intertwiners between quantum exterior powers of the fundamental $U_q(\mathfrak{sl}_N)$ representations. Papers~\cite{Ca,QRS,RW2} indicate that a similar theory should exists for \emph{symmetric foams} that would   categorify networks of quantum symmetric powers of the fundamental representation, but a definition and  evaluation of symmetric foams is unknown as of today.

%
%

\section{Appendix (joint with Lev Rozansky): Comparison with matrix factorizations} 
\label{subsec_mf}
 Each finite degree field extension $\kk\subset F$ is Frobenius. Any nonzero $\kk$-linear map $\varepsilon:F\lra \kk$ is a non-degenerate trace making $F$ a commutative Frobenius algebra over $\kk$. For separable extensions, there is a canonical trace $\tr_{F/\kk}$ used above. 

\vspace{0.1in} 

Matrix factorizations deliver a supply of commutative Frobenius algebras and two-dimensional TQFTs with corners~\cite{KRz,CM,DM}. 
A nondegenerate potential $w\in \kk[x_1,\dots, x_k]$ defines the Jacobi algebra 
\begin{equation}
    J(w) \ := \ \kk[x_1,\dots, x_k]/(\partial_1 w, \dots, \partial_k w ), 
    \quad  
    \partial_i w := \partial w/\partial x_i
\end{equation}
(a potential is called nondegenerate when this quotient algebra is finite-dimensional). The Jacobi algebra is commutative Frobenius and carries a canonical trace $\trG$, given by the Grothendieck residue, {\it i.e.}, see \cite{AGV, GH}. When $\kk$ is a subfield of $\C$, the trace may be written as a complex integral
\begin{equation}
    \trG (p(\undx)) = \frac{1}{(2\pi i)^k}\int_{|\partial_i w|=R} \frac{p(\undx)}{\partial_1 w\cdots \partial_k w}dx_1\cdots dx_k, 
    \qquad  
    p(\undx)\in \kk[x_1,\dots,  x_k]
\end{equation}
over a contour that contains all roots of the  system of equations $\partial_1 w=\ldots = \partial_k w=0$.

\vspace{0.1in} 

Suppose that $F$ is a subfield of $\C$ (in particular, $\mchar(\kk)=0$). Since $F/\kk$ is a simple extension, there is a generating element $\alpha\in F,$  $\kk(\alpha)=F$, and 
\begin{equation}
    F \cong \kk[x]/(f(x)),
\end{equation}
where $f$ is the minimal polynomial of $\alpha$ over $\kk$, 
\begin{equation}
    f(x) = x^n + a_{n-1}x^{n-1}+\ldots + a_0, 
    \qquad  
    a_i \in \kk. 
\end{equation}
We can realize $F$ as the Jacobi algebra of the singularity with the potential $w(x)$ in a single variable $x$ such that $w'(x)=f(x)$, 
\begin{equation}
    w(x) = \frac{1}{n+1}x^{n+1} +\frac{a_{n-1}}{n}x^n + \ldots + a_0 x.
\end{equation}
The polynomial $f(x)$ is irreducible over $\kk$ and can be fully  factored in the algebraic closure $\okk\subset \C$: 
\begin{equation}
    f(x) \ = \ (x-\lambda_1)\cdots (x-\lambda_n), 
    \qquad  
    \lambda_i\in \okk, 
    \quad 
    \lambda_i\not= \lambda_j. 
\end{equation}
The Hessian 
\begin{equation}
    w''(x) = f'(x) = \sum_{i=1}^n \prod_{j\not= i} (x-\lambda_j),
\end{equation}
and 
\begin{equation}
    w''(\lambda_i) \ = \ \prod_{j\not= i} (\lambda_i-\lambda_j).
\end{equation}
For a single variable $x$ and a potential $w(x)$ with $w'(x)=f(x)$ having simple roots only, 
the Grothendieck trace is given by 
\begin{equation}
    \trG(p(x)) = 
    \frac{1}{2\pi i} \int_{|f(x)|=R}\frac{p(x)}{f(x)} dx  = 
    \sum_{i=1}^n \frac{p(\lambda_i)}{\displaystyle{\prod_{j\not=i}} (\lambda_i-\lambda_j)}, 
    \qquad   
    p(x)\in \kk[x],  \ \ R\gg 0.  
\end{equation}


To compare the two traces, note that the 
canonical trace $\tr_{F/\kk}$ in a finite separable field extension can also be characterized as follows. The tensor product $F\otimes_{\kk} \okk$ of $F$ with the algebraic closure $\okk$ of $\kk$ is isomorphic to the direct product of $n$ copies of $\okk$, where $n$ is the degree $[F:\kk]$, 
\begin{equation}
    F\otimes_{\kk} \okk \ \cong \  \okk\times \dots \times \okk.
\end{equation}
This algebra contains $n$ minimal idempotents $e_1,\dots, e_n$, one for each term in the product. 
Trace $\tr_{F/\kk}$ extends $\okk$-linearly to a trace 
\begin{equation*}
\overline{\tr}_{F/\kk} \ : \ F\otimes_{\kk}\okk\lra \okk
\end{equation*} 
that is characterized uniquely by its taking value $1$ on each minimal idempotent, $\overline{\tr}_{F/\kk}(e_i)$ $=$ $1$.

\vspace{0.1in} 

Any other nondegenerate trace $\varepsilon: F\lra \kk$ extends likewise to a $\okk$-linear trace 
\begin{equation*}
\overline{\varepsilon}: F\otimes_{\kk}\okk\lra \okk
\end{equation*}
taking a nonzero value on each idempotent $e_i$, with at least one of these values different from $1$. 

\vspace{0.1in} 

Minimal idempotents $e_k(x) \in \okk[x]/(f(x))$ are given by 
\begin{equation}
    e_k(x) \ = \ \prod_{j\not= k} \frac{x-\lambda_j}{\lambda_k-\lambda_j}. 
\end{equation}
Indeed, $e_k(\lambda_i)=\delta_{i,k}$, so these are delta functions when evaluated on the roots of $f(x)$.  Evaluating the Grothendieck trace on them gives
\begin{equation}
    \trG(e_k(x)) = \sum_{i=1}^n \frac{e_k(\lambda_i)}{
    \displaystyle{\prod_{j\not= i}}(\lambda_i-\lambda_j)}=\frac{1}{\displaystyle{\prod_{j\not= k}}(\lambda_k-\lambda_j)}.
\end{equation}
Thus, values of the two traces on minimal idempotents are 
\begin{equation}
\overline{\tr}_{F/\kk}(e_k)=1, 
\qquad 
\trG(e_k) = \prod_{j\not= k}  \frac{1}{(\lambda_k-\lambda_j)}, 
\qquad 
1\le k\le n,
\end{equation}
and the field extension trace can be written as 
\begin{equation}
    \trG(p(x)) = 
    \frac{1}{2\pi i} \int_{|f(x)|=R}\frac{w''(x) p(x)}{w'(x)}dx = 
    \sum_{i=1}^n p(\lambda_i), 
    \qquad 
    p(x)\in \kk[x]. 
\end{equation}
Notice that we added the Hessian $w''(x)$ to the numerator of the integral and kept the denominator. We see that the two traces differ by multiplication by the Hessian, 
\begin{equation}\label{eq_hessian}
    \tr_{F/\kk}(p(x)) \ = \ \trG(w''(x)p(x)). 
\end{equation}
The second and first derivatives $w''(x),w'(x)$ have no common roots, since all roots of $w'(x)=f(x)$ are simple, and $w''(x)$ is an invertible element of $\kk[x]/(f(x))\cong F$ (the latter ring is a field anyway). In the 2D TQFT of the Landau--Ginzburg model for the potential $w(x)$ the value of a one-holed torus, as an element of the Jacobi algebra (the state space of the circle), is the Hessian $w''(x)$, see Figure~\ref{fig5_3_1}.

\vspace{0.1in}

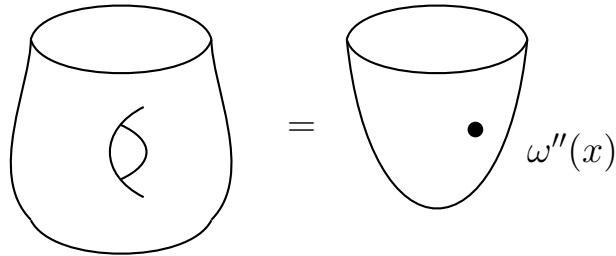
\begin{figure}
    \centering
\begin{tikzpicture}[scale=0.6]
\draw[thick] (0,0) .. controls (0.25,1) and (3.75,1) .. (4,0);
\draw[thick] (0,0) .. controls (0.25,-1) and (3.75,-1) .. (4,0);

\draw[thick] (0,0) .. controls (0,-1) and (-1,-3) .. (0,-4); 
\draw[thick] (4,0) .. controls (4,-1) and (5,-3) .. (4,-4); 
\draw[thick] (0,-4.01) .. controls (0.5,-5) and (3.5,-5) .. (4,-4.01);

\draw[thick] (2.5,-1.5) .. controls (1.5,-2) and (1.5,-3) .. (2.5,-3.5);
\draw[thick] (2,-1.9) .. controls (2.75,-2.25) and (2.75,-2.75) .. (2,-3.1);

\node at (6,-2) {\Large $=$};

\draw[thick] (7,0) .. controls (7.25,1) and (10.75,1) .. (11,0);
\draw[thick] (7,0) .. controls (7.25,-1) and (10.75,-1) .. (11,0);
\draw[thick] (7,0) .. controls (7.5,-5) and (10.5,-5) .. (11,0);

\draw[thick,fill] (10.15,-2.15) arc (0:360:0.25);
\node at (12,-2.5) {\Large $\omega''(x)$};
\end{tikzpicture}
    \caption{One-holed torus represents the element $w''(x)$ in the Jacobi algebra of a one-variable potential.}
    \label{fig5_3_1}
\end{figure}

Consequently, the field extension trace (the map induced by the cap in the TQFT associated to $(F,\kk,\tr_{F/\kk})$) can be written as the cap with the genus one surface (holed torus) in the Landau--Ginzburg TQFT associated to a given generating element $\alpha\in F$, as described earlier, see Figure~\ref{fig5_3_2}.

\vspace{0.1in}

\begin{figure}
    \centering
\begin{tikzpicture}[scale=0.6]
 
\node at (-7.5,3) {\Large $\tr_{F/\kk}$};

\draw[thick,dashed] (-7,0) .. controls (-6.75,1) and (-3.25,1) .. (-3,0);
\draw[thick] (-7,0) .. controls (-6.75,-1) and (-3.25,-1) .. (-3,0);
\draw[thick] (-7,0) .. controls (-6.5,5) and (-3.5,5) .. (-3,0);

\node at (-2,2) {\Large $=$};

\draw[thick,dashed] (0,0) .. controls (0.25,1) and (3.75,1) .. (4,0);
\draw[thick] (0,0) .. controls (0.25,-1) and (3.75,-1) .. (4,0);

\draw[thick] (0,0) .. controls (0,1) and (-1,3) .. (0,4); 
\draw[thick] (4,0) .. controls (4,1) and (5,3) .. (4,4); 
\draw[thick] (0,4.01) .. controls (0.5,5) and (3.5,5) .. (4,4.01);

\draw[thick] (2.5,1.5) .. controls (1.5,2) and (1.5,3) .. (2.5,3.5);
\draw[thick] (2,1.9) .. controls (2.75,2.25) and (2.75,2.75) .. (2,3.1);

\node at (6,2) {\Large $=$};

\draw[thick,dashed] (7,0) .. controls (7.25,1) and (10.75,1) .. (11,0);
\draw[thick] (7,0) .. controls (7.25,-1) and (10.75,-1) .. (11,0);
\draw[thick] (7,0) .. controls (7.5,5) and (10.5,5) .. (11,0);

\draw[thick,fill] (10.15,2) arc (0:360:0.25);
\node at (12,2.5) {\Large $\omega''(x)$};

\node at (8,5.1) {\large MF/LG TQFT maps};

\end{tikzpicture}
    \caption{Cap given by the field extension trace equals the genus one cap trace in the matrix factorization (Landau-Ginzburg) TQFT, for any choice of generator $x$ and the corresponding potential $w(x)$.}
    \label{fig5_3_2}
\end{figure}
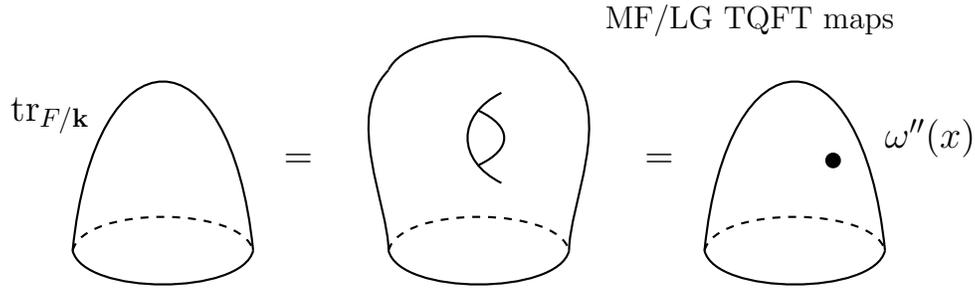

Choosing a different generator $\alpha$ for $F$ will, in general, change the polynomial $f(x)$, potential $w(x)$ and the value of the trace on idempotents of $F\otimes_{\kk}\okk$, while the trace $\tr_{F/\kk}$ is defined canonically. At the same time, it is given by capping off by the holed torus, in any one-variable matrix factorization TQFT realization of $F$ as the Jacobi algebra. 

\vspace{0.1in}

This amusing relation between matrix factorizations and field extensions may be worth a further exploration. Notice, in particular, that $F $ may be realized as the Jacobi algebra, $F\cong J(w)$,  for a multivariable potential $w(\undx)\in \kk[x_1,\dots,x_k]$. Equivalently, $F$ is the zero-dimensional complete intersection of hypersurfaces $\partial_i w=0$, $i=1,\dots, k$. It should be interesting to find nontrivial presentations of that kind for various $F$ with $k>1$ or locate them in the literature. 

\vspace{0.1in} 

The Jacobi algebra $J(w)$ is the endomorphism ring of the canonical matrix factorization
\begin{equation} 
K(w) \ = \ \bigotimes_{i=1}^k K(x_i-y_i,u_i), 
\end{equation} 
a Koszul factorization with the potential $w_{12}=w(\undx)-w(\undy)$ in $2n$ variables $x_1,\dots, x_n$, $y_1,\dots, y_n$. Here $u_i$ are any polynomials in $x$'s and $y$'s such that 
\begin{equation*}
    w_{12} = \sum_{i=1}^k (x_i-y_i)u_i, 
\end{equation*}
and $K(v,u)$ is the factorization
\begin{equation*}
    \kk[\undx,\undy]\stackrel{v}{\lra} \kk[\undx,\undy]
    \stackrel{u}{\lra} \kk[\undx,\undy],
\end{equation*}
see~\cite{KRz}. Matrix factorization $K(w)$ represents the identity functor on the triangulated category $\MF(w)$ of matrix factorizations with potential $w$ and morphisms being homs of matrix factorizations modulo homotopies~\cite{KRz}.  

\vspace{0.1in} 

Assume that $F/\kk$ is a finite Galois extension in characteristic $0$ and consider the Galois group $G=\Gal(F/\kk)$. One can ask to find presentations of $F$ as the Jacobi algebra $J(w)$ such that Galois symmetries $\sigma\in G$ lift to endofunctors of $\MF(w)$ defining an action of the Galois group on that category.  Precisely, for each $\sigma$ we would like to have a matrix factorization $M(\sigma)=M_{\undx,\undy}(\sigma)$ (using subindices to specify sets of variables) with the potential $w(\undx)-w(\undy)$ together with isomorphisms in the 
homotopy category $\MF(\undx-\undz)$ of matrix factorizations with the potential $w(\undx)-w(\undz)$
\begin{equation}\label{eq_mf_isomo}
    M_{\undx,\undy}(\sigma)\otimes_{\undy}M_{\undy,\undz}(\tau) \cong M_{\undx,\undz}(\sigma\tau), \ \ \sigma,\tau\in G,
\end{equation}
such that $M_{\undx,\undy}(1)\cong K(w)$ . One can further require that these isomorphisms satisfy compatibility relations so that $G$ acts on $\MF(w)$ in a strong sense. Furthermore, factorization $M(\sigma)$ should induce the symmetry $\sigma$ on $J(w)\cong F$ upon taking the trace of the identity endomorphism of $M(\sigma)$. Diagrammatically, following notations from~\cite{KRz},  denote $M(\sigma)=M_{\undx,\undy}(\sigma)$ by a dot labeled $\sigma$ on an oriented line with endpoints labeled $\undx,\undy$, see Figure~\ref{fig5_5_1}.

\vspace{0.1in}

\begin{figure}
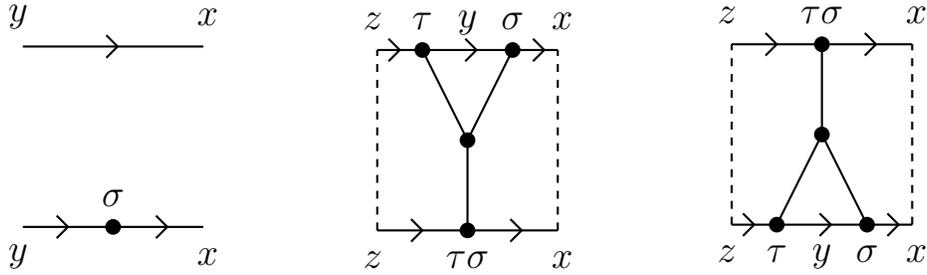

    \centering

    \caption{Top left arc represents the identity factorization $K(w)$. Bottom left arc carrying dot $\sigma$ represents the factorization $M(\sigma)$. Trivalent vertices on the middle and right pictures show mutually-inverse isomorphisms \eqref{eq_mf_isomo}.  }
    \label{fig5_5_1}
\end{figure}

The identity map of $M(\sigma)$ is depicted by a defect interval, shown as a vertical interval in  Figure~\ref{fig5_5_2} left. Taking the trace of the identity map corresponds, on the diagrammatic side, to closing of the square into an annulus with a defect circle on it, see Figure~\ref{fig5_5_2}.

\vspace{0.1in}

\begin{figure}[ht]
    \centering
\begin{tikzpicture}[scale=0.6,decoration={
    markings,
    mark=at position 0.50 with {\arrow{>}}}]

\begin{scope}[shift={(0,0)}]

\draw[thick,postaction={decorate}] (0,0) -- (2,0);

\draw[thick,postaction={decorate}] (2,0) -- (4,0);

\draw[thick,postaction={decorate}] (0,4) -- (2,4);

\draw[thick,postaction={decorate}] (2,4) -- (4,4);

\draw[thick,dashed] (0,0) -- (0,4);
\draw[thick,dashed] (4,0) -- (4,4);

\draw[thick] (2,0) -- (2,4);

\node at (-0.1,4.65) {\Large $y$};
\node at ( 4.1,4.65) {\Large $x$};

\node at (   2,4.70) {\Large $\sigma$};

\draw[thick,fill] (2.25,4) arc (0:360:0.25);

\draw[thick,fill] (2.25,0) arc (0:360:0.25);

\end{scope}

\begin{scope}[shift={(11.5,2)}]

\draw[thick,dashed] (1,0) arc (0:360:1);
\draw[thick] (2,0) arc (0:360:2);
\draw[thick,dashed] (3,0) arc (0:360:3);

\node at (-0.45,0) {\Large $x$};
\node at (-1.4,-0.5) {\Large $\sigma$};
\node at (-2.75,-2.5) {\Large $y$};

\node at (4.5,0) {\Large $\simeq$};

\begin{scope}[shift={(1,0)}]

\draw[thick,dashed] (5,2) .. controls (5.25,2.5) and (8.75,2.5) .. (9,2);
\draw[thick,dashed] (5,2) .. controls (5.25,1.5) and (8.75,1.5) .. (9,2);
\node at (9.60,2) {\Large $x$};

\draw[thick,dashed] (5,0) .. controls (5.25,.5) and (8.75,.5) .. (9,0);
\draw[thick] (5,0) .. controls (5.25,-.5) and (8.75,-.5) .. (9,0);
\node at (9.60,-0.1) {\Large $\sigma$};

\draw[thick,dashed] (5,-2) .. controls (5.25,-1.5) and (8.75,-1.5) .. (9,-2);
\draw[thick,dashed] (5,-2) .. controls (5.25,-2.5) and (8.75,-2.5) .. (9,-2);
\node at (9.60,-1.75) {\Large $y$};

\draw[thick] (5,-2) -- (5,2);
\draw[thick] (9,-2) -- (9,2);
\end{scope}

\end{scope}

\end{tikzpicture}
    \caption{Left: the identity endomorphism of $\sigma$. Middle and right: its trace is a  defect circle on an annulus. Boundaries of the annulus correspond to closures of the identity factorization $K(w)$, given by equating variables $\undx=\undy$ in that factorizations and taking cohomology of the resulting 2-periodic complex. Cohomology is precisely the Jacobi algebra $J(w)$, and the annulus with the $\sigma$-circle defines a linear endomorphism of it. }
    \label{fig5_5_2}
\end{figure}
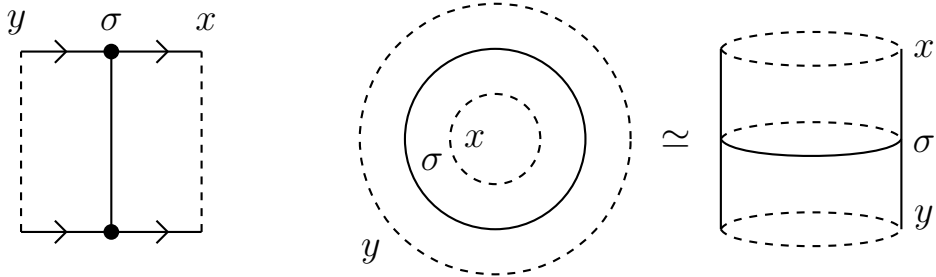

In general, a defect circle on an annulus would only  give a linear endomorphism of the Jacobi algebra, not an algebra homomorphism. For that, we would additionally want the equality shown in Figure~\ref{fig5_5_3} left, which may  come from a more local relation in Figure~\ref{fig5_5_3} right. 

\vspace{0.1in}

\begin{figure}
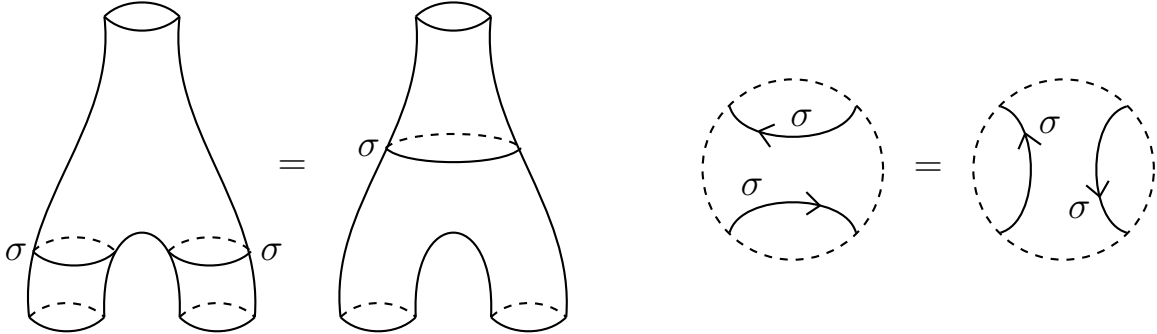

    \centering

    \caption{Left: $\sigma$-circle defining an algebra endomorphism of $J(w)$. Right: a sufficient local relation for that.   }
    \label{fig5_5_3}
\end{figure}

For a general separable finite field extension $F/\kk$, it seems hard to impossible to pick a potential $\omega\in \kk[x_1,\dots, x_n]$ with the Jacobi algebra $J(\omega)\cong F$ and invertible factorizations $M(\sigma), \sigma\in\Gal(F/\kk)$,  giving an action of the Galois group $\Gal(F/\kk)$ on the homotopy category of matrix factorizations $\HMF_{\omega}$ with potential $\omega$ such that the action induces the Galois group action on $F$. 

\vspace{0.1in} 

Potentially related structures appear in the theory of Landau--Ginzburg orbifolds, where a group $G$ acts on $\kk[x_1,\dots, x_k]$ preserving the potential $w$, leading to the category of $G$-equivariant matrix factorizations. In those examples usually $\kk=\CC$, and it is unclear whether some version of LG orbifold theory may be adapted to relate Galois extensions and matrix factorizations.



\end{document}